\documentclass[sts,
	reqno,
]{imsart_old}

\usepackage{amsthm, amssymb, amsmath}
\usepackage{bbm}
\usepackage[toc,page]{appendix}
\usepackage{mathtools, thmtools}
\usepackage{euscript}
\usepackage{color}
\usepackage{tikz}
\usetikzlibrary{arrows}
\usepackage[utf8]{inputenc}
\usepackage[shortlabels]{enumitem}
\usepackage{url}
\usepackage{etex}
\usepackage{alltt}
\usepackage{memhfixc}
\usepackage{bbold}%
\usepackage{cases}
\usepackage[noadjust]{cite}
\definecolor{labelkey}{rgb}{0.0, 0.8, 0.3}
\usepackage[top=2.5cm,bottom=2.5cm,left=2.5cm,right=2.5cm,includeheadfoot,asymmetric]{geometry}
\usepackage{algorithm}
\usepackage{xfrac}
\usepackage{wrapfig}

\usepackage{style}

\numberwithin{equation}{section}
\usepackage{hyperref}
\hypersetup{
  colorlinks = true,
  urlcolor = blue, 
  linkcolor = blue,
  citecolor = blue}
 
\usepackage{cleveref}

\declaretheorem[name=Definition, style=definition]{definition}

\declaretheorem[name=Lemma]{lemma}
\declaretheorem[name=Proposition]{proposition}

\declaretheorem[name=Remark, style=remark]{remark}
\declaretheorem[name=Theorem]{theorem}
\declaretheorem[name=Assumption, style=remark]{assumption}

\renewcommand{\abstractname}{Abstract}

\newenvironment{blockquote}{%
  \par%
  \medskip
  \leftskip=4em\rightskip=2em%
  \noindent\ignorespaces}{%
  \par\medskip}

\def\calN{\mathcal{N}}
\def\calP{\mathcal{P}}
\def\simiid{\stackrel{iid}{\sim}}
\def\TV{\sf{TV}}
\def\H{\sf{H}}
\def\one{\mathbbm{1}}
\def\nbyp#1{{\color{red}[YP: #1]}}
\def\E{\mathbb{E}}
\def\eqdef{\triangleq}
\def\simiid{\stackrel{iid}{\sim}}
\newcommand{\poi}{\operatorname{Poi}}
\newcommand{\ppoi}{\operatorname{poi}}
\renewcommand{\sf}{\mathsf}
\renewcommand{\cal}{\mathcal}
\newcommand{\bb}{\mathbb}
\newcommand{\cov}{\operatorname{Cov}}
\newcommand{\coef}{\operatorname{Coef}}
\newcommand{\ber}{\operatorname{Ber}}
\renewcommand{\rm}{\mathrm}
\renewcommand{\it}[1]{\textit{#1}}
\newcommand{\co}{\operatorname{co}}
\newcommand{\bin}{\operatorname{Bin}}
\renewcommand{\KL}{\sf{KL}}
\newcommand{\unif}{\operatorname{Unif}}
\newcommand{\mult}{\operatorname{Mult}}
\newcommand{\Id}{\sf{Id}}
\newcommand\subsetsim{\mathrel{%
  \ooalign{\raise0.2ex\hbox{$\subset$}\cr\hidewidth\raise-0.8ex\hbox{\scalebox{0.9}{$\sim$}}\hidewidth\cr}}}

\newtheorem{open_problem}{Open problem}

\newcommand\AoScite[2]{#2}
\begin{document}

\begin{frontmatter}

	\title{Likelihood-free hypothesis testing}
	\runtitle{Likelihood-free hypothesis testing}
	\author{Patrik Róbert Gerber \hfill prgerber@mit.edu \\ 
	Yury Polyanskiy \hfill yp@mit.edu \\}

	\address{{Department of Mathematics} \\
		{Massachusetts Institute of Technology}\\
		{77 Massachusetts Avenue,}\\
		{Cambridge, MA 02139, USA}}
	\address{
		 {Department of Electrical Engineering and Computer Science} \\
		{Massachusetts Institute of Technology}\\
		{32 Vassar St,}\\
		{Cambridge, MA 02142, USA}
	}
	
	\runauthor{Gerber and Polyanskiy}
	\begin{abstract}  
	\small{

Consider the problem of binary hypothesis testing. Given $Z$ coming from either $\bb P^{\otimes m}$ or $\bb Q^{\otimes m}$, to decide between the two with small probability of error it is sufficient, and in many cases necessary, to have $m\asymp1/\eps^2$, where $\eps$ measures the separation between $\bb P$ and $\bb Q$ in total variation ($\TV$). Achieving this, however, requires complete knowledge of the distributions and can be done, for example, using the Neyman-Pearson test. In this paper we consider a variation of the problem which we call likelihood-free hypothesis testing, where access to $\bb P$ and $\bb Q$ is given through $n$ i.i.d. observations from each. In the case when $\bb P$ and $\bb Q$ are assumed to belong to a non-parametric family, we demonstrate the existence of a fundamental trade-off between $n$ and $m$ given by $nm\asymp n_\sf{GoF}^2(\eps)$, where $n_\sf{GoF}(\eps)$ is the minimax sample complexity of testing between the hypotheses $H_0:\, \bb P=\bb Q$ vs $H_1:\, \TV(\bb P,\bb Q)\geq\eps$. We show this for three families of distributions, in addition to the family of all discrete distributions for which we obtain a more complicated trade-off exhibiting an additional phase-transition. Our results demonstrate the possibility of testing without fully estimating $\bb P$ and $\bb Q$, provided $m \gg 1/\eps^2$. 
}
\end{abstract}
\end{frontmatter}

\tableofcontents
\section{Introduction}\label{sec:introduction}

A setting that we call \textit{likelihood-free inference (LFI)}, also known as simulation based inference (SBI),   has independently emerged in many areas of
science over the past decades. Given an expensive to collect ``experimental'' dataset and the ability to simulate
from a high fidelity, often mechanistic, stochastic model, whose output distribution and likelihood is intractable
and inapproximable, how does one perform model selection, parameter estimation or construct
confidence sets? The list of disciplines where such highly complex black-box simulators are used
is long, and include particle physics, astrophysics, climate science, epidemiology, neuroscience and ecology to
just name a few. For some of the above fields, such as climate modeling, the bottleneck resource is in fact the simulated data as opposed to the experimental data. In either case, understanding the trade-off between the number of simulations and experiments necessary to do valid inference is crucial. Our aim in this paper is to introduce a theoretical framework under which LFI can
be studied using the tools of non-parametric statistics and information theory.

To illustrate we draw an example from high energy physics, where LFI methods are used and
developed extensively. The discovery of the Higgs boson in 2012 \cite{chatrchyan2012observation,
adam2015higgs} is regarded as the crowning achievement of the Large Hadron collider (LHC) - the
most expensive instrument ever built. Using a composition of complex simulators
\cite{agostinelli2003geant4, frixione2007matching,corcella2001herwig, sjostrand2006pythia,
alwall2007madgraph} modeling the standard model and the detection
process, physicists are able to simulate the results of LHC experiments. Given actual data $Z_1,\ldots, Z_m$ from
the collider, to verify existence of the Higgs boson one tests whether the null hypothesis 
(physics without the Higgs boson, or $Z_i \simiid \bb P_0$) or the alternative hypothesis (physics with the Higgs boson, or $Z_i \simiid \bb P_1$)
 describes the experimental data more accurately. The standard Neyman-Pearson likelihood ratio test is not implementable since $\bb P_0$ and $\bb P_1$ are only available via simulators.
 
 How was this statistical test actually performed? 
First, a probabilistic
classifier $C$ was trained on simulated data to distinguish the two hypotheses (a boosted decision
tree to be more specific). Then, the proportion of real data points falling
in the set $S = \{x \in \R^d:C(x)\leq t\}$ was computed, where $t$ is chosen to maximize an
asymptotic approximation of the power. Finally, $p$-values are reported based on the asymptotic
distribution under a Poisson sampling model \cite{cowan2011asymptotic, lista2017statistical}. Summarizing, the ``Higgs boson'' test was performing the simple comparison
\begin{equation}\label{eqn:scheffe test def intro}\tag{$\textsf{Scheff\'e}$}
    \frac1m\sum\limits_{i=1}^m \one\{Z_i \in S\} \lessgtr \gamma, 
\end{equation}
where $Z_1,\dots,Z_m$ are the real data and $\gamma$ is some threshold. Such count-based tests, named after Scheff\'e in folklore \cite[Section 6]{devroye2001combinatorial}, are quite intuitive. 

Notice that Scheff\'e's test converts each observation $Z_i$ into a binary $0/1$ value. This extreme quantization certainly helps robustness, but should raise the suspicion of potential loss of power. Indeed, when the distributions under both hypotheses are completely known, the optimal Neyman-Pearson test thresholds the sum of \textit{real-valued} logarithms of the likelihood-ratio. Thus, it is natural to expect that a good test should aggregate non-binary values. This is what motivated this work originally, although follow-up work \cite{gerber2023minimax} has shown that Scheff\'e's test with a properly trained classifier can also be optimal.

Let us describe the test that we study for most of this paper. Given estimates $\widehat p_0, \widehat p_1$ of the density of the null and alternative distributions based on simulated samples, our test proceeds via the comparison
\begin{equation}\label{eqn:l2 test intro}
    \frac2m\sum\limits_{i=1}^m (\widehat p_0(Z_i)-\widehat p_1(Z_i)) \lessgtr \gamma
\end{equation}
where $Z_1, \dots, Z_m$ are the real data. Tests of this kind originate from the famous
goodness-of-fit work of Ingster~\cite{ingster1987minimax}, which corresponds to taking $\widehat p_0=p_0$, as the null-density is known exactly.\footnote{In the case of discrete
distributions on a finite (but large) alphabet, the idea was rediscovered by the computer science
community startin with \cite{goldreich2000testing}. Moreover, the difference of $L^2$-norms statistic was first studied in~\cite{kelly2010universal}. See Section~\ref{sec:prior work} for more on the
latter.}   The surprising observation of
Ingster was that such a test is able to reject the null hypothesis that $Z_i\simiid p_0$ even when
the true distribution of $Z$ is much closer to $p_0$ than described by the optimal
density-estimation rate; in other words \textit{goodness-of-fit testing is significantly easier than estimation}. In fact we will use $\gamma
= \|\widehat p_0\|_2^2 - \|\widehat p_1\|_2^2$ in which case~\eqref{eqn:l2 test intro} boils down to the
comparison of two squared $L^2$-distances.

Our overall goal is to understand the trade-off between the 
number  $n$ of simulated observations and the size of the actual data set $m$. The characterization of this tradeoff is reminiscent of the rate-regions in multi-user information theory, but there is an important difference that we wanted to emphasize for the reader. In information theory, the problem is most often stated in the form ``given a distribution $P_{X,Y,Z}$, or a channel $P_{Y,Z|X}$, find the rate region'', with the distribution being completely specified ahead of time. In minimax statistics, however, distributions are a priori only known to belong to a certain class. In \textit{estimation problems} the fundamental limits are thus defined by minimizing the estimation error over this class, and the theoretical goal is to characterize the worst-case rate at which this error converges to zero as the sample size grows to infinity. The definition of the fundamental limit in \textit{testing problems}, however, is more subtle.  If the total variation separation $\eps$ between the null and alternative distribution is fixed, and the number of samples is taken to infinity, then the rate of convergence trivializes and becomes exponentially decreasing in $n$. By now a standard definition of fundamental limit, as suggested by Ingster following ideas of Pittman efficiency, is to vary $\eps$ with $n$ and to find the fastest possible decrease of $\eps$ so as to still have an acceptable probability of error. This is the approach taken in the literature on goodness-of-fit and two-sample testing, and also the one we adopt here. This perspective is also widely used in TCS where the optimal value of $n$, as a function of $\eps$, is referred to as the ``sample complexity" of the problem. 

Specifically, we assume that it is known a priori that the two distributions $\bb P_0,\bb P_1$  belong to a known class $\cal P$ and are
$\eps$-separated under total variation. Given a large number $n$ of samples simulated from $\bb P_0$ and $\bb P_1$ and $m$ samples $Z_1,\ldots,Z_m$ from the experiment, our goal is to test which of the $\bb P_i$ generated the data. If $n$ is sufficiently large to estimate $\bb P_i$ in total variation to precision $\eps/10$, then one can perform the hypothesis test with $m\asymp 1 / \eps^2$ experimental samples, which is information-theoretically optimal even under oracle knowledge of $\bb P_i$'s. However, looking at the test~\eqref{eqn:scheffe test def intro} one may wonder if the full estimation of the distributions $\bb P_i$ is needed, or whether perhaps a suitable decision boundary could be found with a lot fewer simulated samples $n$.
Unfortunately,
our \textit{first main result} disproves this intuition: \textit{any
test using the minimal $m\asymp1/\eps^2$ dataset size will require $n$ so large as to be enough to estimate the
distributions of $\bb P_0$ and $\bb P_1$ to within accuracy $\asymp \eps$}, which is the distance
separating the two hypotheses. In particular, any
method minimizing $m$ performs no different in the worst case, than pairing off-the-shelf
density estimators $\widehat{p}_0,\widehat p_1$ and applying \eqref{eqn:scheffe test def intro} with $S=\{\widehat p_1 \geq \widehat p_0\}$. 

This result appears rather pessimistic and seems to invalidate the whole attraction of LFI, which after all hopes to 
circumvent the exorbitant number of simulation samples required for fully learning high-dimensional
distributions.
Fortunately, \textit{our second result} offers a resolution: if more data samples $m
\gg 1/\eps^2$ are collected, then testing is possible with
$n$ much smaller than required for density estimation. More precisely, when neither $p_0$ nor $p_1$ are known except through $n$ i.i.d. samples from each, the
test~\eqref{eqn:l2 test intro} is able to detect which of the two distributions generated the
$Z$-sample, \textit{even when the number of samples $n$} is insufficient for any estimate $\widehat p_i$ to be within a distance $\asymp \eps = \TV(p_0,p_1)$ from the true values. In other words, the test is
able to reliably detect the true hypotheses even though the estimates $\widehat p_i$ themselves have
accuracy that is orders of magnitude larger than the separation $\eps$ between the hypotheses. 

In summary, this paper shows that likelihood-free hypothesis testing (LFHT) is possible without learning the densities when $m\gg1/\eps^2$, but not
otherwise. It turns out that (appropriate analogues of) the simple
test \eqref{eqn:l2 test intro} has minimax optimal sample complexity up to constants in both $n$ and $m$ in all ``regular'' settings, see also the discussion at the end of \Cref{sec:nonparametric classes}.

\subsection{Informal statement of the main result}

Let us formulate the problem using the notation used throughout the rest of the paper. Suppose that we observe true data $Z \sim \bb P_\sf{Z}^{\otimes m}$ and that we have two candidate parameter settings for our simulator, from which we generate two artificial datasets $X\sim\bb P_\sf{X}^{\otimes n}$ and $Y\sim\bb P_\sf{Y}^{\otimes n}$. If we are convinced that one of the settings accurately reflects reality, we are faced with the problem of testing the hypothesis
\begin{align}\label{eqn:LFHT definition intro}
    H_0:\bb P_\sf{X}=\bb P_\sf{Z} \qquad\text{versus}\qquad H_1:\bb P_\sf{Y}=\bb P_\sf{Z}.
\end{align} 

\begin{remark}
We emphasize that $\bb P_\sf{X}$ and $\bb P_\sf{Y}$ are known only through the $n$ simulated
samples. Thus, \eqref{eqn:LFHT definition intro} can be interpreted as binary hypothesis testing
with approximately specified hypotheses. Alternatively, using the language of machine learning, we may think
of this problem as having $n$ labeled samples from both classes, and $m$ unlabeled samples. The twist is that the unlabeled samples are guaranteed to have the
same common label, that is, they all come from a single class. One can think of many
examples of this setting occurring in genetic, medical and other studies. 
\end{remark}

To put \eqref{eqn:LFHT definition intro} in a minimax framework, suppose that $\bb P_\sf{X},\bb P_\sf{Y}\in\cal P$ for a known class $\cal P$, and that $\TV(\bb P_\sf{X}, \bb P_\sf{Y}) \geq \eps$. Clearly \eqref{eqn:LFHT definition intro} becomes ``easier'' if we have a lot of data (large sample sizes $n$ and $m$) or if the hypotheses are well-separated (large $\eps$). We are interested in characterizing the
pairs of values $(n,m)$ as functions of $\eps$ and $\cal P$, for which the hypothesis test \eqref{eqn:LFHT definition intro}
can be performed with constant type-I and type-II error. Letting $n_\sf{GoF}(\eps, \cal P)$ denote the minimax sample complexity of goodness-of-fit testing (\Cref{def:gof}), we show for \textit{several different classes} of $\cal P$, that \eqref{eqn:LFHT definition intro} is possible with total error, say, $5\%$ if and only if
\begin{equation*}
m \gtrsim 1/\eps^2 \qquad\text{and}\qquad n\gtrsim n_\sf{GoF} \qquad\text{and}\qquad mn \gtrsim n_\sf{GoF}^2. 
\end{equation*}
We also make the observation that $n^2_\sf{GoF}\,\eps^2\asymp n_\sf{Est}$ for these classes, where $n_\sf{Est}(\eps,\cal P)$ denotes the minimax complexity of density estimation to $\eps$-accuracy (\Cref{def:est}) with respect to total variation. This provides additional meaning to the mysterious formula of Ingster \cite{ingster1987minimax} for the sample complexity of goodness-of-fit testing over the class of $\beta$-smooth densities over $[0,1]^d$, see Table~\ref{table:prior results TV} below.\footnote{A possible reason for this observation having been missed previously  is that fundamental limits in statistics are usually presented in the form of \textit{rates} of loss decrease with $n$, for example $r_\sf{Est}(n)\eqdef n^{-1}_\sf{Est}(n) = 1/n^{\beta/(2\beta+d)}$ and $r_\sf{GoF}(n) \eqdef n^{-1}_\sf{GoF}(n) = 1/n^{\beta/(2\beta + d/2)}$ for $\beta$-smooth densities. Unlike $n_\sf{Est}\asymp n_\sf{GoF}^2 \eps^2$ there seems to be no simple relation between $r_\sf{Est}$ and $r_\sf{GoF}$.} More importantly, it allows us to interpret \eqref{eqn:LFHT definition intro} as an  ``interpolation" between different fundamental statistical procedures, namely 
\begin{enumerate}
    \item[$\sf{A}$] $\leftrightarrow$ Binary hypothesis testing,
    \item[$\sf{B}$] $\leftrightarrow$ Estimation followed by robust binary hypothesis testing,
    \item[$\sf{C}$] $\leftrightarrow$ Two-sample testing,
    \item[$\sf{D}$] $\leftrightarrow$ Goodness-of-fit testing, 
\end{enumerate}
corresponding to the extreme points $\sf{A},\sf{B},\sf{C},\sf{D}$ on \Cref{fig:phase diagram}.

\subsection{Related work}\label{sec:prior work}
LHFT as defined in \eqref{eqn:LFHT definition intro} initially appeared in Gutman's paper \cite{gutman1989asymptotically}, building on Ziv's work \cite{Ziv1988OnCW}, where the problem is studied for distributions on a fixed, finite alphabet. Ziv called the problem \textit{classification with empirically observed statistics}, to emphasize the fact that hypotheses are specified only in terms of samples and the underlying true distributions are unknown. In \cite{zhou2020second} it is shown that the error exponent of Gutman's test is second order optimal. Recent work \cite{hsu2020binary, he2020distributed,
haghifam2021sequential, boroumand2022universal} extends this problem to distributed and sequential testing. However, the setting of these papers is fundamentally different from ours, a point which we expand on below. 

Given two arbitrary, unknown $\bb P_\sf{X}, \bb P_\sf{Y}$ over a finite
alphabet of fixed size, Gutman's test (see \cite[Equation (4)]{zhou2020second}) rejects the null hypothesis $H_0:\bb P_\sf{Z}=\bb P_\sf{X}$ in favor of the alternative $H_1:\bb P_\sf{Z}=\bb P_\sf{Y}$ if the statistic $\operatorname{GJS}(\widehat{\bb P}_\sf{X}, \widehat{\bb P}_\sf{Z}, \alpha)$ is large, where $\widehat{\bb P}$ denotes empirical measures, $\operatorname{GJS}$ denotes the generalized Jensen-Shannon divergence defined in \cite[Equation (3)]{zhou2020second} and $\alpha=n/m$. In other words, it simply performs a two-sample test using the samples from $\bb P_\sf{X}$ and $\bb P_\sf{Z}$ of size $n$ and $m$ respectively, and completely discards the sample from $\bb P_\sf{Y}$. In light of our sample complexity results this is strictly sub-optimal due to minimax lower bounds on two-sample testing, see the difference of light gray and striped regions in Figure \ref{fig:phase diagram}. 

More generally, the method of types, which is a crucial tool for the works cited above, cannot be used to derive our results, because in the regime where the alphabet size $k$ scales with the sample size $n$, the usual ${n \choose k} = e^{o(n)}$ approximation no longer holds, i.e. these factors affect estimation rates and do not lead to tight minimax results. As a consequence, one cannot deduce results about the minimax sample complexity of LFHT from works on the classical regime because the latter do not quantify the speed of convergence of the error terms as a function of the alphabet size. Specifically, let us examine \cite[Theorem 1]{zhou2020second}, which is a strengthening of the results of \cite{gutman1989asymptotically}. Paraphrasing, it states that for any fixed ratio $\alpha=n/m$ and pair of distributions $(\bb P_\sf{X},\bb P_\sf{Y})$, Gutman's test has type-II error bounded by $1/3$ when given samples from $\bb P_\sf{X}$ and $\bb P_\sf{Y}$ as input, and type-I error bounded by $\exp(-\lambda n)$ given arbitrary input, where 
\begin{equation}\label{eqn:zhou thm1}\tag{1}
    \lambda = \operatorname{GJS}(\bb P_\sf{X},\bb P_\sf{Y}, \alpha) + \sqrt{\frac{V(\bb P_\sf{X},\bb P_\sf{Y}, \alpha)}{n}}\Phi^{-1}(1/3) + \cal O\left(\frac{\log(n)}{n}\right)
\end{equation}
as $n\to\infty$. Here $V$ denotes the dispersion function defined in \cite[Equation (9)]{zhou2020second} and $\Phi$ is the standard normal cdf. The crucial point we make here is that in \eqref{eqn:zhou thm1} the dependence of the $\cal O(\log(n)/n)$ term on $\bb P_\sf{X},\bb P_\sf{Y}$, and in particular their support size $k$ and the ratio $\alpha=n/m$ is unspecified. Because of this, \eqref{eqn:zhou thm1} and similar results cannot be used to derive minimax sample complexities as $\min\{n,m,k\} \to \infty$ jointly at possibly different rates. 

This distinction between the fixed alphabet size setting studied in \cite{gutman1989asymptotically,Ziv1988OnCW,zhou2020second} and similar works, and our large alphabet setting was recognized by \cite{huang2012classification,huang2013generalized,kelly2010universal,kelly2012classification} whose results are much closer to those of this paper. In \cite{huang2013generalized} Huang and Meyn introduce the concept of ``generalized error exponent'' to deal with support sizes that grow superlinearly with sample size (referred to as the ``sparse sample regime'' by them) in the setting of uniformity testing.\footnote{Uniformity testing is the problem of goodness-of-fit testing where the null is given by a uniform distribution.} In \cite{huang2012classification} they extend this idea to LFHT and say, quote, 
\begin{blockquote}
``In the classification problem, the classical error exponent analysis has been applied to the case of fixed alphabet in \cite{Ziv1988OnCW} and \cite{gutman1989asymptotically}.... However, in the sparse sample problem, the classical error exponent concept is again not applicable, and thus a different scaling is needed."
\end{blockquote}

Moving on to \cite{kelly2010universal, kelly2012classification}, their authors study \eqref{eqn:LFHT definition intro} with $n=m$ over the class of discrete
distributions $p$ with $\min_i p_i \asymp \max_i p_i \asymp 1/n^\alpha$, which they call
\textit{$\alpha$-large sources}. Disregarding the dependence on the $\TV$-separation $\eps$, effectively setting $\eps$ to a constant, they find that achieving non-trivial minimax error is
possible if and only if $\alpha \leq 2$, using in fact the same \textit{difference of squared
$L^2$-distances} test \eqref{eqn:l2 test intro} that we study in this paper. Follow-up work
\cite{huang2012classification} extends to the case $m\neq n$ and the class of distributions on
alphabet $[k]$ with $\max_i p_i \lesssim 1/k$, we also cover this class under the name $\cal{P}_{\sf{Db}}$. In the regime of constant separation $\eps=\Theta(1)$ and $n,m\to \infty$ they show that LFHT with vanishing error is possible if and only if $k=o(\min(n^2, mn))$, thus discovering for the first time the \textit{trade-off} between $m$ and $n$.\footnote{The paper~\cite{huang2012classification} contains implicitly other interesting results. For example, it appears that the constructive (upper bound) part of their proof if done carefully can also handle the case of variable $\eps\to0$ in the regime $m,n \lesssim k$. Specifically, we believe they also show that for the minimax error $\delta \in (0,1)$ LFHT is possible
if $k \log(1/\delta) / \eps^4 \lesssim \min(n^2, nm)$. The lower bound appears to show LFHT is possible only if $k \log(1/\delta) \lesssim \min(n^2,nm)$. In addition they also apply the flattening technique, later re-discovered in~\cite{diakonikolas2016new}.} Contrasting with our work, we are the first to characterize the full $m,n,\eps$ trade-off in the regime of constant probability of error, and we also consider three other classes of distributions, in addition to $\cal{P}_{\sf{Db}}$.

Another related problem is that of two-sample testing with unequal sample sizes, studied in \cite{bhattacharya2015testing, diakonikolas2016new} for the class of discrete distributions $\cal P_\sf{D}$. In \Cref{sec:reductions} we present reductions that show that our problem's sample complexity equals, up to constant factors, to that of two-sample testing in the case $m\geq n$. We emphasize that the distinction between $m\geq n$ and $m \leq n$ is necessary for this equivalence: in the latter case the sample complexities of the two problems are not the same. Moreover, our reduction doesn't help us solve classes other than $\cal P_\sf{D}$, as two-sample testing with unequal sample size exhibits a trade-off between $n$ and $m$ only in classes
for which $n_\sf{TS}\neq n_\sf{GoF}$, see also the discussion at the end of \Cref{sec:nonparametric classes}.

The test \eqref{eqn:scheffe test def intro} has been considered previously \cite{devroye2002note,friedman2004multivariate,lopez2016revisiting,gutmann2018likelihood,liu2020learning,kim2021classification,hediger2022use} and is also known as a ``classification accuracy'' test (CAT). Follow-up work \cite{gerber2023minimax} to the present paper shows that CATs are able to attain a (near-)minimax optimality in all settings studied here, and also achieve optimal dependence on the probability of error (in this paper we only consider a fixed error probability).

\subsection{Contributions} Though the likelihood-free hypothesis testing
problem~\eqref{eqn:LFHT definition intro} has previously appeared under various disguises and was
studied in different regimes for the class of bounded discrete distributions, it omitted the key question of understanding the dependence of the sample complexity on the separation $\eps$. Our work fully characterizes the dependence on the separation $\eps$ (\Cref{THM:P_SGDB UPPER,THM:P_D UPPER}). We discover the existence
of a rather non-trivial trade-off between the $m$ and $n$ showing that in the likelihood-free setting statistical performance ($m$) can be traded for computational resources ($n$). Our results are shown for not just one but multiple distribution classes. In addition, we also demonstrate that LFHT naturally interpolates between its special cases corresponding to goodness-of-fit testing, two-sample testing and density-estimation. As a by-product we observe the relation $n_\sf{GoF}^2\,\eps^2\asymp n_\sf{Est}$ that holds over several classes of distributions and measures of separation, hinting at some universality property. On the technical side we provide a unified 
upper bound analysis for all regular classes we consider, and prove matching lower bounds using techniques of Tsybakov, Ingster and
Valiant. Our upper bound analysis is inspired by Ingster \cite{ingster1982minimax,ingster1987minimax} whose
$L^2$-distance testing approach, originally designed for goodness-of-fit in smooth-density
classes, has been rediscovered in the discrete-alphabet world~\cite{kelly2010universal, kelly2012classification,
goldreich2000testing}. Compared to Ingster's work, the new ingredient needed in the discrete
case is a ``flattening" reduction \cite{huang2012classification,diakonikolas2016new,goldreich2017introduction},
which we also utilize. Several minor extensions are also shown along the way, namely, robustness with respect to $L^2$-misspecification (\Cref{THM:ROBUST}) and characterization of $n_\sf{GoF}$ for the class of $\beta$-smooth densities with $\beta\leq1$ under Hellinger separation (\Cref{THM:HELLINGER GOF}).

\subsection{Structure}
\Cref{sec:preliminaries} defines the statistical problems and the classes of distributions that are studied throughout the paper, and discusses multiple tests for likelihood-free hypothesis testing. \Cref{sec:results} contains our main results and the discussion linking to goodness-of-fit and two-sample testing, estimation and robustness. In \Cref{sec:sketch} we provide sketch proofs for our results. Finally, in \Cref{sec:open problems} we discuss possible future directions of research. The detailed proofs of \Cref{THM:P_SGDB UPPER,THM:P_D UPPER,THM:ROBUST,THM:HELLINGER GOF} and all auxiliary results are included in the \AoScite{supplement \cite{supplement}}{Appendix}. 

\subsection{Notation}
For $k\in\bb N$ we write $[k]\eqdef\{1,2,\dots,k\}$. For $x,y\in\bb R$ we write $x\land y\eqdef\min\{x,y\}$, $x \lor y \eqdef \max\{x,y\}$. We use the Bachmann–Landau notation $\Omega, \Theta, \cal O, o$ as usual and write $f \lesssim g$ for $f=\cal O(g)$ and $f\asymp g$ for $f=\Theta(g)$. For $c\in\R$ and $A\subseteq\R^2$ we write $cA \eqdef \{(ca_1,ca_2)\in\R^2:(a_1,a_2)\in A\}$. For two sets $A,B \subseteq \bb R^2$ we write $A \asymp B$ if there exists $c\in[1, \infty)$ with $\frac1cA\subseteq B\subseteq cA$. For two probability measures $\mu,\nu$ dominated by $\eta$ with densities $p,q$ we define the following divergences: $\TV(\mu,\nu) \eqdef \frac12 \int|p-q|\rm{d}\eta$, $\H(\mu,\nu) \eqdef ( \int(\sqrt{p}-\sqrt{q})^2\rm{d}\eta)^{1/2}$, $\KL(\mu\|\nu) \eqdef \int p \log(p/q) \rm{d}\eta$, $\chi^2(\mu\|\nu) \eqdef \int ((p-q)^2/q)\rm{d}\eta$. Abusing notation, we sometimes write $(p,q)$ as arguments instead of $(\mu, \nu)$. We write $\|\cdot\|_p$ for the $L^p$ and $\ell^p$ norms, where the base measure shall be clear from the context.

\section{Sample complexity, non-parametric classes and tests}\label{sec:preliminaries}

In the first two parts of this section we go over the technical background and definitions that are required to understand the rest of the paper, after which we give an exposition of multiple alternative approaches for our problem in \Cref{sec:tests}. 

\subsection{Five fundamental problems in Statistics}\label{sec:fundamental problems}
Formally, we define a hypothesis as a set of probability measures. Given two hypotheses $H_0$ and $H_1$ on some space $\cal X$, we say that a function $\psi:\cal X \to \{0,1\}$ successfully tests
the two hypotheses against each other if
\begin{equation}\label{eqn:successful test}
\max\limits_{i=0,1} \sup\limits_{P \in H_i} \bb P_{S \sim P}(\psi(S) \neq i) \leq 1/3.
\end{equation}
\begin{remark}\label{rem:test bootstrap}
For our purposes, the constant $1/3$ above is unimportant and could be replaced by any number less than $1/2$. Throughout the paper we
are interested in the asymptotic order of the sample complexity, and $\Omega(\log(1/\delta))$-way sample splitting followed by a majority vote decreases the overall error probability to $\cal O(\delta)$ of any successful tester, at the cost of inflating the sample complexity by a multiplicative $\cal O(\log(1/\delta))$ factor. Unfortunately, the resulting dependence on $\delta$ is sub-optimal except for binary hypothesis testing, see for example \cite[Theorem 4.7]{bar2002complexity}. Recent results for uniformity \cite{diakonikolas2018sample} and two-sample testing \cite{diakonikolas2021optimal}, and our follow-up work on LFHT \cite{gerber2023minimax} resolves the optimal dependence to be $\sqrt{\log(1/\delta)}$ or even $\sqrt[3]{\log(1/\delta)}$ in some regimes. 
\end{remark}
Throughout this section let $\cal P$ be a class of probability distributions on
$\cal X$. Suppose we observe independent samples
$X \sim \bb P_{\sf{X}}^{\otimes n}$, $Y \sim \bb P_{\sf{Y}}^{\otimes n}$
and $Z\sim\bb P_{\sf{Z}}^{\otimes m}$ whose distributions $\bb P_{\sf{X}},\bb P_{\sf{Y}}, \bb P_{\sf{Z}} \in \cal P$ are
 \it{unknown} to us. Finally, $\bb P_0,\bb P_1 \in \cal P$ refer to distributions
that are \it{known} to us. We now define five fundamental problems in statistics that we refer to throughout this paper.

\begin{definition}\label{def:ht}
\textbf{Binary hypothesis testing} is the problem of testing
\begin{equation}\label{eqn:HT definition}\tag{$\sf{HT}$}
H_0:\bb P_\sf{X}=\bb P_0 \qquad\text{against}\qquad H_1:\bb P_\sf{X}=\bb P_1
\end{equation}
based on the sample $X$. We use $n_\sf{HT}(\eps, \cal P)$ to denote the \emph{minimax sample complexity of binary hypothesis testing}, which is the smallest number such that for all $n \geq n_\sf{HT}(\eps, \cal P)$
and all $\bb P_0,\bb P_1 \in \cal P$ with $\TV(\bb P_0,\bb P_1) \geq \eps$ there
exists a function $\psi:\cal X^n \to \{0,1\}$, which given $X$ as input successfully tests
$H_0$ against $H_1$ in the sense of \eqref{eqn:successful test}.
\end{definition}

It is well known that the complexity of binary hypothesis testing is controlled by the Hellinger divergence. 
\begin{lemma}\label{lem:n_HT}
For all $\eps$ and $\cal P$ with $|\cal P|\geq2$, the relation $$n_\sf{HT}(\eps, \cal P) = \Theta\Big(\sup_{\bb P_0,\bb P_1 \in \cal P: \TV(\bb P_0, \bb P_1) \geq \eps} \H^{-2}(\bb P_0, \bb P_1)\Big)$$holds, where the implied constant is universal.
\end{lemma}
\begin{proof}
We include the proof in \AoScite{the supplement \cite{supplement}}{\Cref{sec:proof of lem1}} for completeness. 
\end{proof}
For all $\cal P$ considered in this paper $n_\sf{HT}=\Theta(1/\eps^2)$ holds. Therefore, going forward we usually refrain from the general notation $n_\sf{HT}$ and simply write $1/\eps^2$.

\begin{definition}\label{def:gof}\textbf{Goodness-of-fit testing} is the problem of testing
\begin{equation}\label{eqn:GoF definition}\tag{$\sf{GoF}$}
H_0: \bb P_\sf{X}=\bb P_0 \qquad\text{against}\qquad H_1: \TV(\bb P_\sf{X}, \bb P_0) \geq \eps\text{ and }\bb P_\sf{X}\in\cal P
\end{equation}
based on the sample $X$. We write $n_\sf{GoF}(\eps, \cal P)$ for the \emph{minimax sample complexity of goodness-of-fit testing}, which is the smallest value such that for all $n\geq n_\sf{GoF}(\eps,\cal P)$
and $\bb P_0 \in \cal P$ there exists a function $\psi : \cal X^n \to \{0,1\}$, which given $X$ as input successfully tests  $H_0$
against $H_1$ in the sense of \eqref{eqn:successful test}.
\end{definition}

\begin{definition}
\textbf{Two-sample testing} is the problem of testing
\begin{equation}\label{eqn:TS definition}\tag{$\sf{TS}$}
H_0: \bb P_\sf{X}=\bb P_\sf{Z}\text{ and }\bb P_\sf{X}\in\cal P \qquad\text{against}\qquad H_1: \TV(\bb P_\sf{X}, \bb P_\sf{Z}) \geq \eps\text{ and }\bb P_\sf{X},\bb P_\sf{Z}\in\cal P
\end{equation}
based on the samples $X$ and $Z$. We write $\cal R_\sf{TS}(\eps, \cal P)$ for the maximal subset of $\R^2$ such that for any $(n,m) \in \N^2$ for which there exists $(x,y)\in\cal R_\sf{TS}(\eps, \cal P)$ with $(n,m)\geq (x,y)$ coordinate-wise, there also exists a function $\psi : \cal X^n \times \cal X^m \to \{0,1\}$, which given $X$ and $Z$ as input successfully tests between
$H_0$ and $H_1$ in the sense of \eqref{eqn:successful test}. We will use the abbreviation $n_\sf{TS}(\eps, \cal P) = \min \{\ell \in \N: (\ell,\ell) \in \cal R_\sf{TS}(\eps, \cal P)\}$ and refer to it as the \emph{minimax sample complexity of two-sample testing}. 
\end{definition}

\begin{definition}\label{def:est}
The \emph{minimax sample complexity of \textbf{estimation}} is the smallest value $n_\sf{Est}(\eps, \cal P)$ 
such that for all $n \geq n_\sf{Est}(\eps,\cal P)$ there exists an estimator $\widehat{\bb P}_\sf{X}$, which given $X$ as input satisfies
\begin{equation}\label{eqn:EST definition}\tag{$\sf{Est}$}
\bb E \TV(\widehat{\bb P}_\sf{X}, \bb P_\sf{X}) \leq \eps.
\end{equation}
\end{definition}
In order to simplify the presentation of our final definition, let us temporarily write $\cal P_\eps = \{(\bb Q_0,\bb Q_1)\in\cal P^2 : \TV(\bb Q_0,\bb Q_1)\geq\eps\}$. That is, $\cal P_\eps$ is the set of pairs of distributions in the class $\cal P$ which are $\eps$ separated in total variation. 
\begin{definition}\label{def:lfht}
\textbf{Likelihood-free hypothesis testing} is the problem of testing
\begin{equation}\label{eqn:LFHT definition}\tag{$\sf{LF}$}
H_0: \bb P_\sf{Z}=\bb P_\sf{X}\text{ and }(\bb P_\sf{X}, \bb P_\sf{Y})\in\cal P_\eps \qquad\text{against}\qquad H_1: \bb P_\sf{Z}=\bb P_\sf{Y}\text{ and }(\bb P_\sf{X}, \bb P_\sf{Y})\in\cal P_\eps 
\end{equation}
based on the samples $X,Y$ and $Z$. Write $\cal R_\sf{LF}(\eps, \cal P)$ for the maximal subset of $\R^2$ such that for any $(n,m)\in\bb N^2$ for which there exists $(x,y)\in\cal R_\sf{LF}(\eps,\cal P)$ with $(n,m)\geq (x,y)$ coordinate-wise, there also exists a function $\psi : \cal X^n \times \cal X^n\times\cal X^m \to \{0,1\}$, which given $X,Y$ and $Z$ as input
successfully tests $H_0$ against $H_1$ in the sense of \eqref{eqn:successful test}.
\end{definition}

Requiring $\mathcal R_\sf{TS}(\eps,\cal P)$ to be maximal is well defined, because for any $(n_0,m_0) \in \cal R_\sf{TS}(\eps,\cal P)$ and $(n,m)\in\bb N^2$ with $(n_0,m_0) \leq (n,m)$ coordinate-wise, it must also hold that $(n,m) \in \cal R_\sf{LF}$, since $\psi$ can simply disregard the extra samples. Clearly the same applies also to $\cal R_\sf{LF}(\eps,\cal P)$.

\begin{remark}\label{rem:sfd}
All five definitions above can be modified to measure separation with respect to an arbitrary function $\sf{d}$ instead of $\TV$. We will write $n_\sf{GoF}(\eps, \sf{d}, \cal P)$ et cetera for the corresponding values. 
\end{remark}

\subsection{Four classes of distributions}\label{sec:nonparametric classes}
All of our definitions in the previous section assumed that we have some class of distributions $\cal P$ at hand. Below we introduce the classes that we study throughout the rest of the paper. 

\begin{enumerate}[(i)]
\item \textbf{Smooth density.} Let $\cal C(\beta, d, C)$ denote the set of functions $f:[0,1]^d \to \bb R$ that are $\underline\beta\eqdef\lceil \beta-1\rceil$-times differentiable and satisfy
\begin{align*}
    \|f\|_{\cal C_\beta} \eqdef \max\left\{\max\limits_{0\leq |\alpha| \leq \underline\beta} \|f^{(\alpha)}\|_\infty, \sup\limits_{x\neq y \in [0,1]^d, |\alpha|=\underline\beta} \frac{|f^{(\alpha)}(x)-f^{(\alpha)}(y)|}{\|x-y\|^{\beta-\underline\beta}_2}\right\} \leq C, 
\end{align*}
where we write $|\alpha|=\sum_{i=1}^d\alpha_i$ for the multiindex $\alpha \in \N^d$ as usual. We further define $\cal P_\sf{H}(\beta, d, C)$ to be the class of distributions with Lebesgue-densities in $\cal C(\beta, d, C)$. 

\item \textbf{Gaussian sequence model on the Sobolev ellipsoid.}
Given $C>0$ and a smoothness parameter $s>0$, we define the Sobolev ellipsoid
$$\cal E(s,C)\eqdef\Big\{\theta\in\bb R^{\bb N}: \sum_{j=1}^\infty j^{2s} \theta_j^2 \leq C\Big\}.$$
Our second distribution class is given by 
\begin{equation*}
    \cal P_\sf{G}(s, C) \eqdef \left\{\mu_\theta\,: \theta \in \cal E(s, C)\right\}, 
\end{equation*}
where $\mu_\theta = \otimes_{i=1}^\infty \cal N(\theta_i, 1)$. It is well known that this class models an $s$-smooth signal under Gaussian white noise, see for example \cite[Section 1.7.1]{tsybakov} for an exposition of this connection.

\item[(iii)-(iv)] \textbf{Distributions on a finite alphabet.} For $k \geq 2$, let
\begin{align*}
  \cal P_\sf{D}(k) &\eqdef \{\textrm{all distributions on }\{1,2,\dots,k\}\} \\
  \intertext{denote the class of all discrete distributions, and set}
    \cal P_\sf{Db}(k, C) &\eqdef \{p \in \cal P_\sf{D}(k): \|p\|_\infty \leq C/k\}
\end{align*}
for all $C > 1$. In other words, $\cal P_\sf{Db}$ are those distributions with support in $\{1,2,\dots,k\}$ that are bounded by a constant multiple of the uniform distribution. 
\end{enumerate}

Note that depending on the choice of $C$ some of the above distribution classes may be empty. To avoid such issues, throughout the rest of paper we implicitly operate under the following assumption. 
\begin{assumption}\label{assumption}
    We always assume that $C>1$ when referring to $\cal P_\sf{H}(\beta,d,C)$ and $\cal P_\sf{Db}(k,C)$. 
\end{assumption} 

As we shall see in \Cref{sec:results TV} when discussing our results, the behaviour of $\cal P_\sf{D}$ is qualitatively different from the other three classes introduced above. Consequently, we will sometimes refer to $\cal P_\sf{Db}$ as the ``regular discrete" class, and we will see that its minimax sample complexities are similar to $\cal P_\sf{H}$ and $\cal P_\sf{G}$ but different from $\cal P_\sf{D}$. More generally we will call the classes $\cal P_\sf{H},\cal P_\sf{G}, \cal P_\sf{Db}$ ``regular", characterized by the fact that $n_\sf{GoF} \asymp n_\sf{TS}$, or equivalently, by the fact that $\cal R_\sf{TS} \asymp \{(n,m) : \min\{n, m\} \geq n_\sf{TS}\}$.

\subsection{Tests for LFHT}\label{sec:tests}

We start this section by reintroducing the difference of $L^2$-distances statistic that our results are based on, and which we've already seen in \eqref{eqn:l2 test intro}. Then, in \Cref{sec:alternative tests} we mention some natural alternative approaches to the problem, which we however do not study further. Therefore, the reader that wishes to proceed to our results without delay may safely skip over \Cref{sec:alternative tests}. 

\subsubsection{Ingster's $L^2$-distance test}\label{sec:ingster test} For simplicity we focus on the case of discrete distributions. This case is more general than may first appear: for example in the case of smooth densities on $[0,1]^d$ one can simply take a regular grid (whose resolution is determined by the smoothness of the densities) and count the number of datapoints falling in each cell. Let $\widehat{p}_\sf{X}, \widehat{p}_\sf{Y}, \widehat{p}_\sf{Z}$ denote the empirical probability mass functions of the finitely supported distributions $\widehat{\bb P}_\sf{X}, \widehat{\bb P}_\sf{Y}, \widehat{\bb P}_\sf{Z}$. The test proceeds via the comparison
\begin{equation}\label{eqn:l2 diff test definition}
    \|\widehat{p}_\sf{X}-\widehat{p}_\sf{Z}\|_2 \lessgtr \|\widehat{p}_\sf{Y}-\widehat{p}_\sf{Z}\|_2. 
\end{equation}
Squaring both sides and rearranging, we arrive at the form
\begin{equation*}
    \frac 1m \sum\limits_{i=1}^m (\widehat{p}_\sf{Y}(Z_i)-\widehat{p}_\sf{X}(Z_i)) \lessgtr \gamma, 
\end{equation*}
where $\gamma = (\|\widehat{p}_\sf{Y}\|^2-\|\widehat{p}_\sf{X}\|^2)/2$. As mentioned in the introduction,
variants of this $L^2$-distance based test have been invented and re-invented multiple times
for goodness-of-fit \cite{ingster1987minimax, goldreich2000testing} and two-sample testing
\cite{batu2013testing, arias2018remember}. The exact statistic \eqref{eqn:l2 diff test definition}
with application to $\cal P_\sf{Db}$ has appeared in \cite{kelly2010universal,
kelly2012classification}, and Huang and Meyn \cite{huang2012classification} proposed an ingenious
improvement restricting attention exclusively to bins whose counts are one of $(2,0),(1,1),(0,2)$
for the samples $(X,Z)$ or $(Y,Z)$. We attribute \eqref{eqn:l2 diff test definition} to Ingster
because his work on goodness-of-fit testing for smooth densities is the first occurence of the
idea of comparing empirical $L^2$ norms, but we note that~\cite{kelly2010universal}
and~\cite{goldreich2000testing} arrive at this influential idea apparently independently. 

We emphasize the following subtlety. Let us rewrite~\eqref{eqn:l2 diff test definition} as
\begin{equation}\label{eq:l2dtd2} \|\widehat{p}_\sf{X}-\widehat{p}_\sf{Z}\|_2^2 -  \|\widehat{p}_\sf{Y}-\widehat{p}_\sf{Z}\|_2^2 \lessgtr 0\,.
\end{equation} 
As we shall see from our proofs, this difference results in an optimal test for the full range of possible values of $n$ and $m$ for $\cal P_\sf{Db}$. However, this does not mean that each term by itself is a meaningful estimate of the corresponding distance: rejecting the null by thresholding just $\|\widehat{p}_\sf{X}-\widehat{p}_\sf{Z}\|_2^2$ would not work. Indeed, the variance of  $\|p_\sf{X}-\widehat{p}_\sf{Z}\|_2^2$ is so large that one needs $m \gtrsim n_\sf{GoF} \gg 1/\eps^2$ observations to obtain a reliable estimate of $\|p_\sf{X}-p_\sf{Z}\|_2^2$. The ``magic" of the $L^2$-difference test is that the two terms in~\eqref{eq:l2dtd2} separately have high variance, and thus are not good estimators of their means, but their difference cancels the high-variance terms.   

\begin{remark}
While testing \eqref{eqn:LFHT definition}, practitioners are usually interested in obtaining a $p$-value, rather than purely a decision whether to reject the null hypothesis. For this we propose the following scheme. Let $\sigma_1,\dots,\sigma_P$ be i.i.d. uniformly random permutations on $n+m$ elements. Let $\widehat T = \|\widehat p_\sf{X}-\widehat p_\sf{Z}\|_2^2-\|\widehat p_\sf{Y}-\widehat p_\sf{Z}\|_2^2$ be our statistic, and write $\widehat T_i$ for the statistic $\widehat T$ evaluated on the permuted dataset where $\{X_1,\dots,X_n,Z_1,\dots,Z_m\}$ are shuffled according to $\sigma_i$. Under the null the random variables $\widehat T, \widehat T_1,\dots,\widehat T_P$ are exchangeable, thus reporting the empirical upper quantile of $\widehat T$ in this sample yields an estimate of the $p$-value. Studying the variance of this estimate or the power of the test that rejects when the estimated $p$-value is less than some threshold, is beyond the scope of this work. 
\end{remark}

\subsubsection{Alternative tests for LFHT}\label{sec:alternative tests}
In this section we discuss a variety of alternative tests that may be considered for \eqref{eqn:LFHT definition} instead of \eqref{eq:l2dtd2}. These are included only to provide additional context for our problem, and the reader may safely skip it and proceed to our results in \Cref{sec:results}. The approaches we consider are
\begin{enumerate}[(i)]
\item Scheff\'e's test, 
\item Likelihood-free Neyman-Pearson test and
\item Huber's and Birg\'e's robust tests.
\end{enumerate}

The tests $(i$-$ii)$ are based on the idea of using the simulated samples to learn a set or a function that separates $\bb P_\sf{X}$ from
$\bb P_\sf{Y}$. The test $(iii)$ and \eqref{eq:l2dtd2} use the simulated samples to obtain density estimates of $\bb P_\sf{X}, \bb P_\sf{Y}$ directly. All of them,
however, are of the form 
\begin{equation}\label{eqn:sum s(Z_i)}
\sum_{i=1}^m s(Z_i) \lessgtr 0
\end{equation}
with only the function $s$ varying.

Variants of \emph{Scheff\'e's test} using machine-learning enabled classifiers are the subject of current research in two-sample testing \cite{lopez2016revisiting,liu2020learning, gutmann2018likelihood, kim2021classification, hediger2022use} and are used in practice for LFI specifically in high energy physics, see also our discussion of the Higgs boson discovery in \Cref{sec:introduction}. Thus, understanding the performance of Scheff\'e's test in the context of \eqref{eqn:LFHT definition} is of great practical importance. Suppose that using the simulated samples we train a probabilistic classifier $C:\cal X \to [0,1]$ on the labeled data $\cup_{i=1}^n \{(X_i,0), (Y_i,1)\}$. The specific form of the classifier here is arbitrary and can be anything from logistic regression to a deep neural network. Given thresholds $t, \gamma \in [0,1]$ chosen to satisfy our risk appetite for type-I vs type-II errors, Scheff\'e's test proceeds via the comparison
\begin{equation}\label{eqn:scheffe test def}
    \frac 1m \sum\limits_{i=1}^m \one\{C(Z_i) \geq t\} \lessgtr \gamma. 
\end{equation}
We see that \eqref{eqn:scheffe test def} is of the form \eqref{eqn:sum s(Z_i)} with $s(z) = (\one\{C(z) \geq t\}-\gamma)/m$. The follow-up work \cite{gerber2023minimax} studies the performance of Scheff\'e's test in great detail, finding that it is (near-)minimax optimal in all cases considered in this paper. It is found that the optimal classifier $C$ must be trained \emph{not} purely to minimize misclassification error, but rather must also keep the variance of its output small.

If the distributions $\bb P_\sf{X}, \bb P_\sf{Y}$ are fully known, then
the likelihood-ratio test corresponds to 
\begin{equation}\label{eq:neyman_pearson}
	\sum_{i=1}^m s_\sf{NP}(Z_i) \lessgtr \gamma\qquad s_\sf{NP}(z) = \log\left(\frac{\D\bb P_\sf{X}}{\D\bb P_\sf{Y}}(z)\right)\,, 
\end{equation}
where $\gamma$ is again chosen to satisfy our type-I vs type-II error trade-off preferences. It is well known that the above procedure is optimal due to the Neyman-Pearson lemma. Recall that in our setting $\bb P_\sf{X}, \bb P_\sf{Y}$ are known only up to i.i.d. samples, and therefore it seems natural to try to estimate $s_\sf{NP}$ from samples. It is not hard to see that $s_\sf{NP}$ minimizes the population \textit{cross-entropy/logistic loss}, that is
$$ s_\sf{NP} = \argmin_s \E_{z\sim \bb P_\sf{X}}[\ell(s(z),1)]  + \E_{z\sim \bb P_\sf{Y}}[\ell(s(z),0)]\,, $$
where $\ell(s,y) = \log(1+e^s) -ys$. In practice, the majority of today's classifiers are obtained by
running some form of gradient descent on the problem
$$ \widehat s = \argmin_{s\in \cal G} \E_{z\sim \widehat{\bb P}_\sf{X}}[\ell(s(z),1)]  + \E_{z\sim \widehat{\bb P}_\sf{Y}}[\ell(s(z),0)]\,, $$
where $\cal G$ is, say, a parametric class of neural networks and $\widehat{\bb  P}_\sf{X}, \widehat{\bb P}_\sf{Y}$ are
empirical distributions. Given such an estimate $\widehat s$, we can replace the unknown $s_\sf{NP}$
in~\eqref{eq:neyman_pearson} by $\widehat s$ to obtain the \textit{likelihood-free Neyman-Pearson
test}. For recent work on this approach in LFI see for example \cite{dalmasso2020confidence}. 
Studying properties of this test is outside the scope of this paper.

The final approach is based on the idea of \emph{robust testing}, first proposed by Huber~\cite{huber1965robust, huber1973minimax}. Huber's seminal result implies that if one has approximately correct distributions $\widehat{\bb P}_\sf{X}, \widehat{\bb P}_\sf{Y}$ satisfying 
$$\max\Big\{\TV(\widehat{\bb P}_\sf{X}, \bb P_\sf{X}), \TV(\widehat{\bb P}_\sf{Y}, \bb P_\sf{Y})\Big\} \le \eps/3\,\qquad\text{and}\qquad \TV(\bb P_\sf{X}, \bb P_\sf{Y}) \geq \eps,$$
then for some $c_1 < c_2$ the test
$$ \sum_{i=1}^m s_\sf{H}(Z_i) \lessgtr 0\qquad\text{where}\qquad s_\sf{H}(z)=\min\left\{\max\left\{c_1,  \log \left(\frac{\D\widehat{\bb P}_\sf{X}}{\D\widehat{\bb P}_\sf{Y}}(z)\right) \right\}, c_2 \right\}$$
has type-I and type-II error bounded by $\exp(-\Omega(m\eps^2))$, and is in fact minimax optimal for all sample sizes analogously to the likelihood-ratio test in the case of binary hypothesis testing. From the above formula we can see that Scheff\'e's test can be interpreted as an approximation of the maximally robust Huber's test. Let $\widehat{\cal L}(z) = (\D\widehat{\bb P}_\sf{Y}/\D\widehat{\bb P}_\sf{X})(z)$ denote the likelihood-ratio of the estimates. The values of $c_1,c_2$ are given as the solution to 
\begin{equation*}
    \eps /3 = \E_{z \sim \widehat{\bb P}_\sf{X}}\left[ \one\left\{\widehat{\cal L}(z) \leq c_1\right\} \frac{c_1-\widehat{\cal L}(z)}{1+c_1}\right] = \E_{z \sim \widehat{\bb P}_\sf{Y}} \left[\one\left\{\widehat{\cal L}(z) \geq c_2\right\}\frac{\widehat{\cal L}(z)-c_2}{1+c_2}\right]\, , 
\end{equation*}
which can be easily approximated to high accuracy given samples from $\widehat{\bb
P}_\sf{X}, \widehat{\bb P}_\sf{Y}$. This suggests both a theoretical construction, since $\widehat{\bb P}_\sf{X}, \widehat{\bb P}_\sf{Y}$ can be obtained with high probability from simulation
samples via the general estimator of Yatracos~\cite{yatracos1985rates}, and a practical rule:
instead of the possibly brittle likelihood-free Neyman-Pearson test $(ii)$, one should try clamping the
estimated log-likelihood ratio from above and below. 

Similar results hold due to Birg\'e \cite{birge1979theoreme, birge2013robust} in the case when distance is measured by Hellinger divergence:
\begin{equation*}
    \max\Big\{\H(\widehat{\bb P}_\sf{X}, \bb P_\sf{X}), \H(\widehat{\bb P}_\sf{Y}, \bb P_\sf{Y})\Big\} \leq \eps/3\qquad\text{and}\qquad \H(\bb P_\sf{X}, \bb P_\sf{Y}) \geq \eps.
\end{equation*}
For ease of notation, let $\widehat p_\sf{X}, \widehat p_\sf{Y}$ denote the densities of $\widehat{\bb
P}_\sf{X},\widehat{\bb P}_\sf{Y}$ with respect to some base measure $\mu$. Regarding
$\sqrt{\widehat{p}_\sf{X}}$ and $\sqrt{\widehat{p}_\sf{Y}}$ as unit vectors of the Hilbert space $L^2(\mu)$, let
$\gamma:[0,1] \to L^2(\mu)$ be the constant speed geodesic on the unit sphere of $L^2(\mu)$ with
$\gamma(0)=\sqrt{\widehat{p}_\sf{X}}$ and $\gamma(1)=\sqrt{\widehat{p}_\sf{Y}}$. It is easily checked that
each $\gamma_t$ is positive, and Birgé showed that the test
\begin{equation*}
    \sum\limits_{i=1}^m \log\left(\frac{\gamma^2_{1/3}}{\gamma^2_{2/3}}(Z_i)\right) \lessgtr 0
\end{equation*}
has both type-I and type-II errors bounded by $\exp(-\Omega(m\eps^2))$. For an exposition of this result see also \cite[Theorem 7.1.2]{gine2021mathematical}

\section{Results}\label{sec:results}

In this section we describe our results on the sample complexity of likelihood-free hypothesis testing. 
\subsection{General reductions}\label{sec:reductions}
In this first section, we give reductions that hold in great generality and show the relationship of our problem with other classical testing and estimation problems that were introduced in \Cref{sec:fundamental problems}. The result below holds for a generic class $\cal P$ of distributions and a generic measure of separation $\sf{d}$, see also \Cref{rem:sfd}. 

\begin{proposition}\label{PROP:REDUCTIONS}
Let $\cal P$ be a generic family of distributions and $\sf{d}:\cal P^2 \to \R$ be any function used to measure separation. There exists a universal constant $c>0$ such that for $n,m \in \N$ the following implications hold. 
\begin{align}
(n,m) \in \cal R_\sf{LF} &\implies m\geq n_\sf{HT},\label{eqn:LF -> HT} \\
(n,m) \in \cal R_\sf{TS} &\implies n\land m \geq n_\sf{GoF}\label{eqn:TS -> GoF} \\
(n,m) \in \cal R_\sf{LF}&\implies  cn\geq n_\sf{GoF}, \label{eqn:LF -> GoF}\\
(n,m) \in \cal R_\sf{TS} &\implies (n,m)\in\cal R_\sf{LF}, \label{eqn:TS -> LF}\\ 
m\geq n \text{ and }(n,m) \in \cal R_\sf{LF} &\implies (cn,cm) \in \cal R_\sf{TS}\label{eqn:LF -> TS},  
\end{align}
where we omit the argument $(\eps, \sf{d}, \cal P)$ throughout for simplicity. In particular, 
\begin{align}\label{eqn:LF TS equiv}
    \N^2_{n\leq m} \cap \cal R_\sf{LF} \asymp \N^2_{n\leq m} \cap \cal R_\sf{TS}, 
\end{align}
where $\N^2_{n \leq m} = \{(n,m)\in\N^2:n\leq m\}$. 
\end{proposition}

\begin{proof}
    
In what follows, let $\Psi_\sf{LF}, \Psi_{\sf{TS}}$ be minimax optimal tests for \eqref{eqn:LFHT definition} and \eqref{eqn:TS definition} respectively. Throughout the proof we omit the arguments $(\eps, \sf d, \cal P)$ for notational simplicty. 

\textbf{Reducing hypothesis testing to \eqref{eqn:LFHT definition}}
Suppose $(n,m) \in \cal R_\sf{LF}$. Let $\bb P_0,\bb P_1 \in \cal P$ be given with $\sf d(\bb P_0, \bb P_1) \geq \eps$ and suppose $Z$ is an i.i.d. sample with $m$ observations. We wish to test the hypothesis $H_0:Z_i \sim \bb P_0$ against $H_1: Z_i \sim \bb P_1$. To this end generate $n$ i.i.d. observations $X,Y$ from $\bb P_0, \bb P_1$ respectively, and simply output $\Psi_\sf{LF}(X,Y,Z)$. This shows that if $(n,m) \in \cal R_\sf{LF}$ then $m \geq n_\sf{HT}$ and concludes the proof of \eqref{eqn:LF -> HT}.

\textbf{Reducing goodness-of-fit testing to two-sample testing}
Suppose $(n,m) \in \cal R_\sf{TS}$. Then obviously $(n\land m,\infty) \in \cal R_\sf{TS}$. However, two-sample testing with sample sizes $n\land m, \infty$ is equivalent to goodness-of-fit testing with a sample size of $n\land m$. Therefore, $n\land m \geq n_\sf{GoF}$ must hold, concluding the proof of \eqref{eqn:TS -> GoF}.

\textbf{Reducing goodness-of-fit testing to \eqref{eqn:LFHT definition}}
Suppose $(n,m) \in \cal R_\sf{LF}$ with $m\leq n$. Let a distribution $\bb P_0 \in \cal P$ be given as well as an i.i.d. sample $X$ of size $cn$ with unknown distribution $\bb P_\sf{X}$, where $c \in \N$ is a large integer. We want to test $H_0: \bb P_\sf{X} = \bb P_0$ against $H_1:\bb P_\sf{X} \in \cal P, \sf d(\bb P_\sf{X}, \bb P_0) \geq \eps$. Generate $c\times2$ i.i.d. samples $Y^{(i)},Z^{(i)}$ for $i=1,\dots,c$ of size $n,m$ respectively, all from $\bb P_0$. Split the sample $X$ into $c$ batches $X^{(i)},i=1,\dots,c$ of size $n$ each and form the variables
\begin{equation*}
    A_i = \Psi_\sf{LF}(X^{(i)}, Y^{(i)}, Z^{(i)}) - \Psi_\sf{LF}(X^{(i)}, Y^{(i)}, X^{(i+1)}_{1:m})
\end{equation*}
for $i=1,3,\dots,2\lfloor c/2\rfloor-1$, where $X^{(i)}_{1:m}$ denotes the first $m$ observations in the batch $X^{(i)}$. Note that the $A_i$ are i.i.d. and bounded random variables. Under the null hypothesis we have $\E A_i=0$, while under the alternative they have mean $\E A_i \geq 1/3$ (since $\Psi_\sf{LF}$ is a successful tester in the sense of \eqref{eqn:successful test}). Therefore, a constant number $c/2$ observations suffice to decide whether $\bb P_\sf{X} = \bb P_0$ or not. In particular, $c n \geq n_\sf{GoF}$ which concludes the proof of \eqref{eqn:LF -> GoF} for the case $m \leq n$. The case $n \leq m$ follows from \eqref{eqn:LF -> TS} and \eqref{eqn:TS -> GoF}.

\textbf{Reducing \eqref{eqn:LFHT definition} to two-sample testing}
Suppose $(n,m) \in \cal R_\sf{TS}$. Let three samples $X,Y,Z$ be given, of sizes $a,a,b$ from the unknown distributions $\bb P_\sf{X},\bb P_\sf{Y},\bb P_\sf{Z}$ respectively, where $\{a,b\}=\{n,m\}$. We want to test the hypothesis $H_0:\bb P_\sf{X}=\bb P_\sf{Z}$ against $H_1:\bb P_\sf{Y}=\bb P_\sf{Z}$, where $\sf{d}(\bb P_\sf{X},\bb P_\sf{Y}) \geq \eps$ under both. Then, the test
\begin{equation*}
    \widetilde{\Psi}_\sf{LF}(X,Y,Z) \eqdef \Psi_\sf{TS}(X,Z)
\end{equation*}
shows that $(n,m), (m,n) \in \cal R_\sf{LF}$ and concludes the proof of \eqref{eqn:TS -> LF}.

\textbf{Reducing two-sample testing to \eqref{eqn:LFHT definition}}
Suppose $(n,m) \in \cal R_\sf{LF}$ where $m\geq n$. Let two samples $X,Y$ be given, from the unknown distributions $\bb P_\sf{X}, \bb P_\sf{Y} \in \cal P$ and of sample size $cn, cm$ respectively, where $c \in \N$ is a large integer. We wish to test the hypothesis $H_0:\bb P_\sf{X}=\bb P_\sf{Y}$ against $H_1: \sf d(\bb P_\sf{X}, \bb P_\sf{Y}) \geq \eps$. Split the samples $X, Y$ into $2\times c$ batches $X^{(i)},Y^{(i)}, i=1,\dots,c$ of sizes $n,m$ respectively, and form the variables
\begin{equation*}
    A_i = \Psi_\sf{LF}(X^{(i)}, Y_{1:n}^{(i)}, Y^{(i+1)}) - \Psi_\sf{LF}(Y_{1:n}^{(i)}, X^{(i)}, Y^{(i+1)})
\end{equation*}
for $i=1,3,\dots,2\lfloor c/2\rfloor-1$, where $Y^{(i)}_{1:n}$ denotes the first $n$ observations in the batch $Y^{(i)}$. The variables $A_i$ are i.i.d. and bounded. Under the null hypothesis we have $\E A_i=0$ while under the alternative $\E A_i \geq 1/3$ holds. Therefore a constant number $c/2$ observations suffice to decide whether $\bb P_\sf{X} = \bb P_\sf{Y}$ or not. In particular, $(cn,cm) \in \cal R_\sf{TS}$ which concludes the proof of \eqref{eqn:LF -> TS}. 

\textbf{Equivalence between two-sample testing and \eqref{eqn:LFHT definition}} Equation \eqref{eqn:LF TS equiv} follows immediately from \eqref{eqn:LF -> TS} and \eqref{eqn:TS -> LF}. 
\end{proof}

\Cref{eqn:LF TS equiv} tells us that the problems of likelihood-free hypothesis testing and two-sample testing are equivalent, \emph{but only for} $m \geq n$, that is, when we have more real data than simulated data. We will see in the next section, and on \Cref{fig:phase diagram} visually, that this distinction is necessary.

\subsection{Sample
complexity of likelihood-free hypothesis testing}\label{sec:results TV}
In this section we present our results on the sample complexity of \eqref{eqn:LFHT definition} for the specific classes $\cal P$ that were introduced in \Cref{sec:fundamental problems}, with separation measured by $\TV$. In all results below the parameters $\beta, s, d, C$ are regarded as constants, we only care about the dependence on the separation $\eps$ and the alphabet size $k$ (in the case of $\cal P_\sf{D},\cal P_\sf{Db}$). Where convenient we omit the arguments of $n_\sf{GoF}, n_\sf{TS}, \cal R_\sf{TS}, n_\sf{Est}, \cal R_\sf{LF}$ to ease notation, whose value should be clear from the context. 

\begin{theorem}\label{THM:P_SGDB UPPER}
Under $\TV$-separation, for each choice $\cal P \in \{\cal P_\sf{H}, \cal P_\sf{G}, \cal P_\sf{Db}\}$, we have
\begin{equation*}
    \cal R_\sf{LF} \asymp \Big\{(n,m) \,:\, m \geq 1/\eps^2,\, n\geq n_\sf{GoF},\, mn \geq n^2_\sf{GoF}\Big\},
\end{equation*}
where the implied constants do not depend on $k$ (in the case of $\cal P_\sf{Db}$) or $\eps$. 
\end{theorem}

For each class $\cal P$ in \Cref{THM:P_SGDB UPPER}, the entire region $\cal R_\sf{LF}$ (within universal constant) is attained by a suitable modification of Ingster's $L^2$-distance test from \Cref{sec:ingster test}. The region $\cal R_\sf{LF}$ is visualized on \Cref{fig:phase diagram} on a $\log$-$\log$ scale, with each corner point $\{\sf{A}, \sf{B}, \sf{C},\sf{D}\}$ having a special interpretation, as per the reductions presented in \Cref{PROP:REDUCTIONS}. The point $\sf{A}$ corresponds to binary hypothesis testing and $\sf{D}$ can be reduced to goodness-of-fit testing. Similarly, $\sf{B}$ and $\sf{C}$ can be reduced to the well-known problems of
estimation followed by robust hypothesis testing and two-sample testing respectively. In other words, \eqref{eqn:LFHT definition} allows us to naturally interpolate between multiple statistical problems. Finally, we make an interesting observation: since the product of $n$ and $m$ remains constant on the line segment $[\sf{B}, \sf{C}]$ on the left plot of \Cref{fig:phase diagram}, it follows that 
\begin{equation}\label{eqn:n_Est = n_GoF^2}
    n_\sf{Est}(\eps, \cal P) \asymp n_\sf{GoF}^2 (\eps, \cal P)\,\eps^2 
\end{equation}
for each class $\cal P$ treated in \Cref{THM:P_SGDB UPPER}. This relation between the sample complexity of estimation and goodness-of-fit testing has not been observed before to our knowledge, and understanding the scope of validity of this relationship is an exciting future direction.\footnote{ \textit{Added in print:} for example in \cite{zeyu} it is demonstrated that for the Gaussian sequence model (see definition $(ii)$ in \Cref{sec:nonparametric classes}) with the Sobolev ellipsoid replaced by the set $\Theta = \{\theta \in \ell^2 : \sum_{i=1}^\infty i|\theta_i| \leq 1\}$, it holds that $n_\sf{Est} \ll n_\sf{GoF}^2/\eps^2$. } 

Turning to our results on $\cal P_\sf{D}$ the picture is less straightforward. As first identified in \cite{batu2000testing} and fully resolved in \cite{chan2014optimal}, the sample complexity of two-sample testing undergoes a phase transition when $k\gtrsim1/\eps^4$. This phase transition appears also in likelihood-free hypothesis testing. 
\begin{theorem}\label{THM:P_D UPPER}
Let $\alpha = \max\big\{1, \min\big\{\frac kn, \frac km\big\}\big\}$. Then 
\begin{equation*}
\cal R_\sf{LF}(\eps, \cal P_\sf{D}(k)) \asymp_{\log(k)} \Big\{(n,m)\,:\, m\geq 1/\eps^2,\, n\geq n_\sf{GoF}(\eps, \cal P_\sf{D}(k)) \cdot \sqrt{\alpha},\, mn \geq n_\sf{GoF}(\eps, \cal P_\sf{D}(k))^2\cdot\alpha\Big\}, 
\end{equation*}
where the equivalence is up to a logarithmic factor in the alphabet size $k$. 
\end{theorem}

\begin{figure}[!ht]
\includegraphics[width=\textwidth]{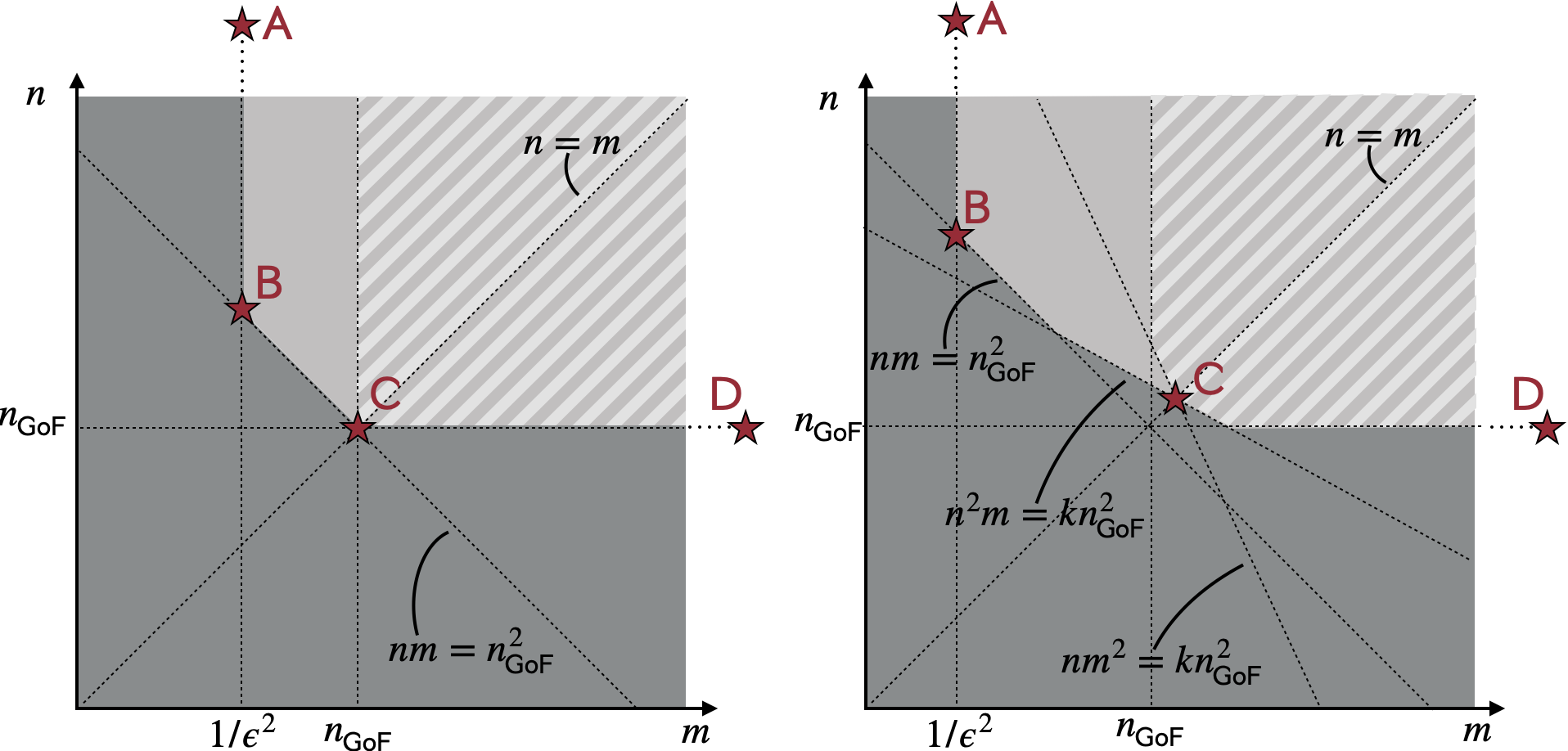}\caption{Light and dark gray show $\cal R_\sf{LF}$ and its complement resp. on $\log$ scale; the striped region depicts $\cal R_\sf{TS} \subsetneq \cal R_\sf{LF}$. Left plot is valid for $\cal P \in \{\cal P_\sf{H}, \cal P_\sf{G}, \cal P_\sf{Db}\}$ for all settings of $\eps, k$. For $\cal P_\sf{D}$ the left plot applies when $k \lesssim \eps^{-4}$ and the right plot otherwise. }\label{fig:phase diagram}
\end{figure}

The $\log k$ factor in our analysis originates from a union bound, and it is possible that it may be removed. It follows from follow up work \cite{gerber2023minimax} and past results on two-sample testing \cite{diakonikolas2021optimal} that the $\log(k)$ factor can be removed in all regimes, thus fully characterizing the sample complexity of \eqref{eqn:LFHT definition}, but using a different test from ours.

\Cref{table:prior results TV} summarizes previously known tight results for the values of $n_\sf{GoF}, n_\sf{TS}, \cal R_\sf{TS}$ and $n_\sf{Est}$. The fact that $n_\sf{HT}=\Theta(1/\eps^2)$ for reasonable classes is classical, see \Cref{lem:n_HT}. The study of goodness-of-fit testing within a minimax framework was pioneered by Ingster \cite{ingster1982minimax, ingster1987minimax} for $\cal P_\sf{H},\cal P_\sf{G}$, and independently studied by the computer science community \cite{goldreich2000testing, valiant2017automatic} for $\cal P_\sf{D}, \cal P_\sf{Db}$ under the name \textit{identity testing}. Two-sample testing (a.k.a. \textit{closeness testing}) was solved in \cite{chan2014optimal} for $\cal P_\sf{D}$ (with the optimal result for $\cal P_\sf{Db}$ implicit) and \cite{ingster1987minimax, arias2018remember, li2019optimality} consider $\cal P_\sf{H}$. The study of the rate of estimation $n_\sf{Est}$ is older, see \cite{ibragimov1977estimation, tsybakov, johnstone, gine2021mathematical} and references for $\cal P_\sf{H},\cal P_\sf{G}$ and \cite{canonne2020short} for $\cal P_\sf{D},\cal P_\sf{Db}$. 
\begin{table}[!ht]\caption{Prior results on testing and estimation}\label{table:prior results TV}
\begin{center}
\begin{tabular}{c|c | c | c | c}
 & $n_\sf{HT}$ & $n_\sf{GoF}$  & $\cal R_\sf{TS}$ & $n_\sf{Est}$ \\ [0.5ex]
 \hline\hline
 $\cal P_\sf{G}$ & $1/\eps^2$ &  $1/\eps^{(2s+1/2)/s}$  & $n \land m\geq n_\sf{GoF}$ & $\eps^2\, n_\sf{GoF}^2 $ \\
 \hline
 $\cal P_\sf{H}$ & $1/\eps^2$ &  $1/\eps^{(2\beta+d/2)/\beta}$  & $n \land m\geq n_\sf{GoF}$ & $\eps^2\,n_\sf{GoF}^2$ \\
  \hline
 $\cal P_\sf{Db}$ & $1/\eps^2$ &  $\sqrt{k}/\eps^2$ & $n \land m\geq n_\sf{GoF}$ & $\eps^2\, n_\sf{GoF}^2 $ \\
 \hline
 $\cal P_\sf{D}$ & $1/\eps^2$ & $\sqrt{k}/\eps^2$ & $n\lor m \geq \frac{\sqrt{k}}{\eps^2} \lor \frac{k^{2/3}}{\eps^{4/3}} \asymp n_\sf{TS},\,\, n\land m \geq n_\sf{GoF}\sqrt\alpha$ & $\eps^2\, n^2_\sf{GoF}$
\end{tabular}
\end{center}
\end{table}

\subsection{$L^2$-robust likelihood-free hypothesis testing}\label{sec:estimation limit}

Even before seeing \Cref{THM:P_SGDB UPPER,THM:P_D UPPER} one might guess that estimation in $\TV$ followed by a robust hypothesis test should work whenever $m\gtrsim1/\eps^2$ and $n\geq n_\sf{Est}(c\eps)$ for a small enough constant $c$. This strategy does indeed work, which can be deduced from the work of Huber and Birgé \cite{huber1965robust, birge2013robust} for total variation and Hellinger separation respectively, see also \Cref{sec:tests} for a brief discussion of these robust tests. In other words, we have the informal theorem 
\begin{equation*}
\text{if separation is measured by $\TV$ or $\H$, then } (n \geq n_\sf{Est} \text{ and } m \geq n_\sf{HT}) \implies (cn,cm) \in \cal R_\sf{LF}.
\end{equation*}
In the case of total variation separation, in fact an even simpler approach succeeds: if $\widehat p_\sf{X}$ and $\widehat p_\sf{Y}$ are minimax optimal density estimators with respect to $\TV$, then Scheff\'e's test using the classifier $C(x) = \one\{\widehat p_\sf{Y}(x) \geq \widehat p_\sf{X}(x)\}$ can be shown to achieve the optimal sample complexity by Chebyshev's inequality. 

The upshot of these observations is that they provide a solution to \eqref{eqn:LFHT definition} that is robust to model misspecification, specifically at the corner point $\sf{B}$ on \Cref{fig:phase diagram}. This naturally leads us to the question of robust likelihood-free hypothesis testing: can we construct robust tests for the full $m$ vs $n$ trade-off? 

As before, suppose we observe samples $X,Y,Z$ of size $n,n,m$ from distributions belonging to the class $\cal P$ with densities $f,g,h$ with respect to some base measure $\mu$. Given any $u\in\cal P$, let $\sf{B}_u(\eps,\cal P)\subseteq \cal P$ denote a region around $u$ against which we wish to be robust. Recall the notation $\cal P_\eps = \{(\bb Q_0,\bb Q_1)\in\cal P^2 : \TV(\bb Q_0,\bb Q_1)\geq\eps\}$ from \Cref{def:lfht}. We compare the hypotheses
\begin{equation}\label{eqn:rLFHT definition}\tag{$\sf{rLF}$}
        H_0: h\in\sf{B}_f(\eps,\cal P), (f,g) \in \cal P_\eps \qquad\text{versus}\qquad H_1: h\in\sf{B}_g(\eps,\cal P), (f,g) \in \cal P_\eps, 
\end{equation}
and write $\cal R_{\sf{rLF}}(\eps, \cal P, \sf{B}_\cdot)$ for the region of $(n,m)$-values for which \eqref{eqn:rLFHT definition} can be performed successfully, defined analogously to $\cal R_\sf{LF}(\eps,\cal P)$. Note that $\cal R_\sf{rLF} \subseteq \cal R_\sf{LF}$ provided $u\in\sf{B}_u$ for all $u\in\cal P$, that is, the range of sample sizes $n,m$ for which robustly testing \eqref{eqn:LFHT definition} is possible ought to be a subset of $\cal R_\sf{LF}$. 
\begin{theorem}\label{THM:ROBUST}
\Cref{THM:P_SGDB UPPER,THM:P_D UPPER} remain true if we replace $\cal R_\sf{LF}(\eps, \cal P)$ by $\cal R_\sf{rLF}(\eps, \cal P, \sf B_\cdot)$ for the following choices:

\begin{enumerate}[(i)]
    \item for $\cal P_\sf{H}(\beta,d,C)$ and $\sf{B}_u=\{v\in\calP_\sf{H}(\beta,d,C):\|u-v\|_2\leq c\eps\}$ for a constant $c>0$ independent of  $\eps$,  
    \item for $\cal P_\sf{G}(s,C)$ and $\sf{B}_{\mu_\theta} = \{\mu_{\theta'}:\theta'\in\cal E(s, C), \|\theta-\theta'\|_2\leq \eps/4\}$, 
    \item for $\cal P_\sf{Db}(k,C)$ and $\sf{B}_u=\{v:\|u-v\|_2\leq \eps/(2\sqrt{k})\}$, and 
    \item for $\cal P_\sf{D}(k)$ and $\sf B_u = \{v:\|u-v\|_2\leq c\eps/\sqrt k, \|v/u\|_\infty \leq c\}$ for a constant $c>0$ independent of $k$ and $\eps$. 
\end{enumerate}
\end{theorem}

\subsection{Beyond total variation}\label{sec:results H}
Recall from \Cref{rem:sfd} the notation $n_\sf{GoF}(\eps, \sf{d}, \cal P)$ etc. that is applicable when separation is measured with respect to a general measure of discrepancy $\sf{d}$ instead of $\TV$. In recent work \cite[Theorem 1]{nazer2017information} and~\cite[Lemma 3.6]{pensia2022communication} it is shown that any test that first quantizes the data by a map $\Phi:\cal X \to \{1,2,\dots,M\}$ for some $M\geq2$ must decrease the Hellinger distance between the two hypotheses by a $\log$ factor in the worst case. This implies that for every class $\cal P$ rich enough to contain such worst case examples, a quantizing test, such as Scheff\'e's, can hope to achieve $m \asymp \log(1/\eps)/\eps^2$ at best, as opposed to the optimal $m\asymp1/\eps^2$. Thus, if separation is assumed with respect to Hellinger distance, Scheff\'e's test should be avoided. This example shows that the choice of $\sf{d}$ can have surprising effects on the performance of specific tests that would be optimal under other circumstances. Understanding the sample complexity of \eqref{eqn:LFHT definition} for $\sf{d}$ other than $\TV$ might lead to new algorithms and insights. 

This motivates us to pose the question: does a trade-off analogous to that identified in \Cref{THM:P_SGDB UPPER} hold for other choices of $\sf{d}$, and $\H$ in particular? In the case of $\cal P_\sf{G}$ we obtain a simple, almost vacuous answer. From \Cref{lem:gaussian equivalence} it follows immediately that the results of \Cref{table:prior results TV} and \Cref{THM:P_SGDB UPPER} continue to hold for $\cal P_\sf{G}$ for any of $\sf{d} \in \{\H, \sqrt{\KL}, \sqrt{\chi^2}\}$, to name a few.

\begin{lemma}\label{lem:gaussian equivalence}
Let $C>0$ be a constant. For any $\theta \in \ell^2$ with $\|\theta\|_2\leq C$
\begin{equation*}
    \TV(\mu_\theta, \mu_0) \asymp \H(\mu_\theta, \mu_0) \asymp \sqrt{\KL(\mu_\theta \| \mu_0)} \asymp \sqrt{\chi^2(\mu_\theta\|\mu_0)} \asymp \|\theta\|_2,
\end{equation*}
where $\mu_\theta \eqdef \otimes_{i=1}^\infty \cal N(\theta_i, 1)$ and the implied constant depends on $C$. 
\end{lemma}

The case of $\cal P_\sf{D}$ is more intricate. Substantial recent progress \cite{diakonikolas2016new, kamath2018modern, daskalakis2018distribution, canonne2020short} has been made, where among others, the complexities $n_\sf{GoF}, n_\sf{TS}, n_\sf{Est}$ for Hellinger separation are identified.
\begin{table}[h!]
\label{table:H prior results}
\begin{center}
\begin{tabular}{c |c | c | c | c}
 & $n_\sf{HT}$ & $n_\sf{GoF}$  & $n_\sf{TS}$ & $n_\sf{Est}$ \\ [0.5ex]
 \hline\hline
 $\cal P_\sf{D}$ & $1/\eps^2$ &  $\sqrt{k}/\eps^2$  & $k^{2/3}/\eps^{8/3} \land k^{3/4}/\eps^2$ & $n_\sf{Gof}^2\, \eps^2$ \\
 $\cal P_\sf{H}$ & $1/\eps^2$ & $?$ & $?$ & $1/\eps^{2(\beta+d)/\beta}$ \\
\end{tabular}
\caption{Prior results for $\sf{d}=\H$.}
\end{center}
\end{table}
Since our algorithm for \eqref{eqn:LFHT definition} is $\|\cdot\|_2$-based, we could immediately derive achievability bounds for $\cal R_\sf{LF}(\eps, \H, \cal P_\sf{D})$ via the inequality $\|\cdot\|_2 \geq \H^2/\sqrt{k}$, however such a naive technique yields suboptimal results, and thus we omit it. Studying \eqref{eqn:LFHT definition} under Hellinger separation for $\cal P_\sf{D}$ and $\cal P_\sf{Db}$ is beyond the scope of this work. 

Finally, we turn to  $\cal P_\sf{H}$. Due to the nature of our proofs, the results of \Cref{THM:P_SGDB UPPER} easily generalize to $\sf{d} = \|\cdot\|_p$ for any $p\in[1,2]$. The simple reason for this is that $(i)$ our algorithm is $\|\cdot\|_2$-based and $\|\cdot\|_2\geq\|\cdot\|_p$ by Jensen's inequality and $(ii)$ the lower bound construction involves perturbations near $1$, where all said norms are equivalent. In the important case $\sf{d}=\H$ the estimation rate $n_\sf{Est}(\eps, \H, \cal P_\sf{H})\asymp1/\eps^{2(\beta+d)/\beta}$ was obtained by Birgé \cite{birge1986estimating}, our contribution here is the study of $n_\sf{GoF}$.

\begin{theorem}\label{THM:HELLINGER GOF}
For any $\beta>0,C>1$ and $d\geq1$ there exists a constant $c>0$ such that 
 \begin{equation*}
     n_\sf{GoF}(\eps, \H, \cal P(\beta, d, C)) \geq c/\eps^{2(\beta+d/2)/\beta}.
 \end{equation*}
 If in addition we assume that $\beta \in (0,1]$, $c$ can be chosen such that 
 \begin{equation*}
     cn_\sf{GoF}(\eps, \H, \cal P) \leq  1/\eps^{2(\beta+d/2)/\beta}. 
 \end{equation*}
 In particular, $n_\sf{Est} \asymp n_\sf{GoF}^2\,\eps^2$.
 \end{theorem}
 
\section{Sketch proof of main results}\label{sec:sketch}

In this section we briefly sketch the proofs of the main results of the paper. 

\subsection{Upper bounds for \Cref{THM:P_SGDB UPPER,THM:P_D UPPER,THM:ROBUST,THM:HELLINGER GOF}} 

\subsubsection{Bounded discrete distributions} Consider first the case when $\bb P_\sf{X}$ and $\bb P_\sf{Y}$ belong to the class $\cal P_\sf{Db}$, that is, they are supported on the discrete set $\{1,2,\dots, k\}$ and bounded by the uniform distribution. Let $\widehat{p}_\sf{X}, \widehat{p}_\sf{Y}, \widehat{p}_\sf{Z}$ denote empirical probability mass functions based on the samples $X,Y,Z$ of size $n,n,m$ from $\bb P_\sf{X}, \bb P_\sf{Y}, \bb P_\sf{Z}$ respectively. Define the test statistic
\begin{align*}
T_\sf{LF} &= \|\widehat{p}_\sf{X}-\widehat{p}_\sf{Z}\|_2^2-\|\widehat{p}_\sf{Y}-\widehat{p}_\sf{Z}\|_2^2
\end{align*}
and the corresponding test $\psi(X,Y,Z) = \one\{T_\sf{LF} \geq 0\}$. The proof of \Cref{THM:P_SGDB UPPER,THM:P_D UPPER} hinge on the precise calculation of the mean and variance of $T_\sf{LF}$. Due to symmetry it is enough to compute these under the null. The proof of the upper bound is then completed via Chebyshev's inequality: if $n,m$ are such that $(\bb E T_\sf{LF})^2 \gtrsim \var(T_\sf{LF})$ for large enough implied constant on the right then $\psi$ tests \eqref{eqn:LFHT definition} successfully in the sense of \eqref{eqn:successful test}.  
\begin{proposition}[informal]\label{prop:inform proof}
Suppose $\|p_\sf{X}+p_\sf{Y}+p_\sf{Z}\|_\infty \leq C_\infty / k$. Then $\psi$ successfully tests \eqref{eqn:LFHT definition} if 
\begin{equation}\label{eqn:informal proof}
    \underbracket[0.140ex]{\,\,\frac{\eps^4}{k^2}}_{(\E T_\sf{LF})^2} \gtrsim \underbracket[0.140ex]{\frac{C_\infty\eps^2}{k^2} \left(\frac1n+\frac1m\right) + \frac{C_\infty}{k}\left(\frac{1}{n^2} + \frac{1}{nm}\right)}_{\operatorname{var}(T_\sf{LF})}. 
\end{equation}
\end{proposition}

From \eqref{eqn:informal proof} one can immediately see where each constraint in the region $\cal R_\sf{LF}(\eps, \cal P_\sf{Db}(k,C))$ in \Cref{THM:P_SGDB UPPER} emerges. The first two terms in the variance require that both $m$ and $n$ be larger than $\Omega(1/\eps^2)$. The $1/n^2$ term in the variance requires that $n$ be at least $\Omega(\sqrt k/\eps^2) \asymp n_\sf{GoF}$, and the $1/(nm)$ term requires that the product $nm$ be at least $\Omega(n_\sf{GoF}^2)$.

\subsubsection{Smooth densities} 
Next we describe how \Cref{prop:inform proof} can be applied to the class $\cal P_\sf{H}$ of smooth densities. Divide $[0,1]^d$ into into $\kappa^d$ regular grid cells for some $\kappa \in \N$. Discretize the three samples $X,Y,Z$ over this grid and simply apply the optimal test for $\cal P_\sf{Db}$, observing the crucial fact that this discretization belongs to $\cal P_\sf{Db}$. The following lemma, originally due to Ingster \cite{ingster1987minimax} controls the approximation error of the discretization. 
\begin{lemma}[{{\cite[Lemma 7.2]{arias2018remember}}}]\label{lem:ingster_approx}
Let $P_\kappa$ denote the $L^2$ projection onto the space of functions constant on each grid cell. For any $\beta>0,C>1$ and $d\geq1$ there exist constants $c,c' > 0$ such that for any $f,g \in \cal P_\sf{H}(\beta, d, C)$ the following holds:
\begin{equation*}
    \|f-g\|_2 \geq \|P_\kappa(f-g)\|_2 \geq c \|f-g\|_2 - c'\kappa^{-\beta}. 
\end{equation*}
\end{lemma}

Based on \Cref{lem:ingster_approx} we set $\kappa \asymp \eps^{-1/\beta}$. This resolution is chosen to ensure that the discrete approximation to any $\beta$-smooth density is sufficiently accurate, that is, approximate $\eps$-separation is maintained even after discretization. We see now that our problem is reduced entirely to testing over $\cal P_\sf{Db}$, so we may apply \Cref{prop:inform proof} with $k=\kappa^d \asymp \eps^{-d/\beta}$, which yields the minimax optimal rates from \Cref{THM:P_SGDB UPPER,THM:ROBUST}.

Our proof of the achieavability direction in \Cref{THM:HELLINGER GOF} follows similarly by reduction to goodness-of-fit testing for discrete distributions \cite{daskalakis2018distribution} under Hellinger separation, where it is known that $n_\sf{GoF}(\eps, \H, \cal P_\sf{D}) \asymp \sqrt{k}/\eps^2$. The key step is to prove a result similar to \Cref{lem:ingster_approx} but for $\H$ instead of $\|\cdot\|_2$. 

\begin{proposition}\label{prop:H separation}
For any $\beta \in (0,1]$, $C>1$ and $d\geq1$ there exists a constant $c>0$ such that 
\begin{equation*}
    c\H(f, g) \leq \H(P_\kappa f, P_\kappa g) \leq \H(f,g)
\end{equation*}
holds for any $f,g \in \cal P_\sf{H}(\beta, d, C)$, provided we set $\kappa = (c\eps)^{-2/\beta}$.
\end{proposition}

\subsubsection{Gaussian sequence model}
Let us briefly discuss the Gaussian sequence class $\cal P_\sf{G}(s,C)$. 
Here our approach is not to discretize the distributions, but conceptually the test is very similar to the cases we've already covered. Let us write $\bb P_\sf{X} = \mu_{\theta_\sf{X}}$ and define ${\theta_\sf{Y}}, {\theta_\sf{Z}}$ analogously. For a given cutoff $r$, we simply 
\begin{equation}\label{eqn:gaussian test sketch}
    \text{reject the null if } T_{\sf{LF},\sf{G}} \eqdef \sum_{i=1}^r \Big\{\big(\widehat \theta_{\sf{X}, i} - \widehat \theta_{\sf{Z},i}\big)^2 - \big(\widehat\theta_{\sf{Y},i}-\widehat\theta_{\sf{Z},i}\big)^2\Big\} \geq 0, 
\end{equation}
where $\widehat\theta_{\sf{X},i} = \frac1n\sum_{j=1}^n X_{ji}$ and $\widehat\theta_{\sf{Y}}, \widehat\theta_\sf{Z}$ are defined analogously. Once again, a precise calculation of the mean and the variance of the sum above, yields the following result. 
\begin{proposition}[informal]
    Set $r\asymp\eps^{-1/s}$. The test \eqref{eqn:gaussian test sketch} succeeds if
    \begin{equation}\label{eqn:gauss mean-var informal}
        \underbracket[0.140px]{\,\,\eps^4}_{(\E T_{\sf{LF},\sf{G}})^2} \gtrsim \underbracket[0.140px]{\eps^2\left(\frac1n+\frac1m\right) + \eps^{-1/s}\left(\frac{1}{n^2} + \frac{1}{nm}\right)}_{\operatorname{var}(T_{\sf{LF},\sf{G}})}. 
    \end{equation}
\end{proposition}

Similarly to \eqref{eqn:informal proof}, we can again read of the constraints that define the region $\cal R_\sf{LF}(\eps, \cal P_\sf{G}(s,C))$ from \eqref{eqn:gauss mean-var informal}. The first and second terms in the variance ensure that $n,m =\Omega(1/\eps^2)$ and $n^2,mn =  \Omega(n_\sf{GoF}^2) = \Omega(\eps^{-(4s+1)/s})$ respectively. 

\subsubsection{General discrete distributions} 
Finally, we comment on $\cal P_\sf{D}$. Here we can no longer assume that $C_\infty = \cal O(1)$ in \Cref{prop:inform proof}, in fact $C_\infty = \Omega(k)$ is possible. We get around this by utilizing the reduction based approach of \cite{diakonikolas2016new, goldreich2017introduction}. We take the first half of the data and compute $$B_i=1+\#\left\{j \leq \frac{\min\{k, n\}}{2}:X_j=i\right\}+\#\left\{j \leq \frac{\min\{k, n\}}{2}:Y_j=i\right\}+\#\left\{j\leq\frac{\min\{k, m\}}{2}:Z_j=i\right\}$$ for each $i\in[k]$. Then, we divide the $i$'th support element into $B_i$ bins, uniformly. This transformation preserves pairwise total variation, but reduces the $\ell^\infty$-norms of $p_\sf{X},p_\sf{Y},p_\sf{Z}$ with high probability,  to order $1/(k\land (n \lor m))$, after an additional step that we omit here. We can then perform the usual test with these new ``flattened" distributions, using the untouched half of the data. 

It is insightful to interpret the ``flattening" procedure followed by $L^2$-distance comparison as a one-step procedure that simply compares a different divergence of the empirical measures.  Intuitively, in contrast to the regular classes, one needs to mitigate the effect of potentially massive differences in the empirical counts on bins $i\in[k]$ where both $p_\sf{X}(i)$ and $p_\sf{Y}(i)$ are large but their difference $|p_\sf{X}(i)-p_\sf{Y}(i)|$ is moderate. Let $\sf{LC}_\lambda$ be the ``weighted Le-Cam divergence" which we define as $\sf{LC}_\lambda(p\|q)=\sum_i (p_i-q_i)^2/(p_i + \lambda q_i)$ for two probability mass functions $p,q$. One may interpret the two step procedure (flattening followed by comparing $L^2$ distances) as approximately comparing empirical weighted Le-Cam divergences. Performing the test in two steps is a proof device, and we expect the test that directly compares, say, the Le-Cam divergence of the empirical probability mass functions to have the same minimax optimal sample complexity. Such a one-shot approach is used for example in the paper \cite{chan2014optimal} for two-sample testing. While Ingster \cite{ingster1987minimax} only considers goodness-of-fit testing to the uniform distribution, his notation also suggests the idea of normalizing by the bin mass under the null.

\subsection{Lower bounds for \Cref{THM:P_SGDB UPPER,THM:P_D UPPER,THM:ROBUST,THM:HELLINGER GOF}}\label{sec:lower sketch}
The reductions given in \Cref{PROP:REDUCTIONS} immediately yields a number of tight lower bounds on $n$ and $m$. Namely, \eqref{eqn:LF -> HT} gives $m\gtrsim1/\eps^2$ and \eqref{eqn:LF -> GoF} gives $n\gtrsim n_\sf{GoF}$. Obtaining the lower bound on the product term $m n$ proves more challenging. First we introduce the well known information theoretic tools we use to prove our minimax lower bounds. 

Suppose that we have two (potentially composite) hypotheses $H_0,H_1$ that we test against each other. Our strategy relies on the method of two fuzzy hypotheses \cite{tsybakov}, which is a generalization of Le-Cam's two point method. Write $\cal M(\cal X)$ for the set of probability measures on the set $\cal X$. 
\begin{lemma}\label{lem:lower bound main}
Take two hypotheses $H_i \subseteq \cal M(\cal X)$ and random $P_i \in \cal M(\cal X)$. Then
\begin{align*}
    2\,\inf\limits_{\psi}\max_{i=0,1}\sup\limits_{P \in H_i} P(\psi\neq i) \geq 1-\TV(\bb E P_0, \bb E P_1) -\sum_i\bb P(P_i\notin H_i),
\end{align*}
where the infimum is over all tests $\psi:\cal X \to \{0,1\}$.
\end{lemma}
\begin{proof}
We may assume without loss of generality that $\bb P(P_i\in H_i)>0$ for both $i=0$ and $i=1$, as otherwise the claim is vacuous. Let $\tilde P_i$ be distributed as $P_i | \{P_i \in H_i\}$. Then for any set $A \subset \cal X$ we have
\begin{equation*}
    \left|\bb E \tilde{P}_i(A) - \bb E P_i(A)\right| = \bb P( P_i \notin H_i)\Big|\bb E[P_i(A) | P_i \in H_i] - \bb E[P_i(A) | P_i \notin H_i]\Big| \leq \bb P(P_i \notin H_i).
\end{equation*}
In particular, $\TV(\bb E\tilde{P}_0, \bb E\tilde{P}_1) \leq \TV(\bb EP_0, \bb EP_1)+\sum_i\bb P(P_i\notin H_i)$. Therefore, for any $\psi$
\begin{align*}
    \max_{i=0,1}\sup\limits_{\bb P_i\in H_i} \bb P_i(\psi\neq i) &\geq \frac12(1-\TV(\bb E \tilde{P}_0, \bb E\tilde{P}_1)) \geq \frac12\left(1-\TV(\bb EP_0, \bb EP_1) - \sum_i\bb P(P_i\notin H_i)\right).
\end{align*}
\end{proof}

For clarity, we formally state \eqref{eqn:LFHT definition} as testing between the hypotheses
\begin{equation}\label{eqn:lfht lower main}
\begin{aligned}
    H_0&=\{\bb P_\sf{X}^{\otimes n} \otimes \bb P_\sf{Y}^{\otimes n} \otimes \bb P_\sf{X}^{\otimes m}:\,\bb P_\sf{X},\bb P_\sf{Y}\in\cal P,\,\TV(\bb P_\sf{X},\bb P_\sf{Y}) \geq \eps\} \\
    &\text{versus} \\
    H_1&=\{\bb P_\sf{X}^{\otimes n} \otimes \bb P_\sf{Y}^{\otimes n} \otimes P_\sf{Y}^{\otimes m}:\,\bb P_\sf{X},\bb P_\sf{Y}\in\cal P,\,\TV(\bb P_\sf{X},\bb P_\sf{Y}) \geq \eps\}.
\end{aligned}
\end{equation}

The lower bounds of \Cref{THM:ROBUST} follow from those for \Cref{THM:P_SGDB UPPER,THM:P_D UPPER} so we may focus on the latter.

\subsubsection{Smooth densities}
For concreteness let us focus on the case of $\cal P = \cal P_\sf{H}$. We take $\bb P_0$ to be uniform on $[0,1]^d$ and $\bb P_\eta$ to have density 
\begin{equation}\label{eqn:main P_S construction}
    p_\eta = 1 + \sum_{j\in[\kappa]^d} \eta_j h_j
\end{equation}
with respect to $\bb P_0$. Here $\kappa \in \N$, each $\eta \in \{\pm1\}^{\kappa^d}$ is uniform and $h_j$ is a bump function supported on the $j$'th cell of the regular grid of size $\kappa^d$ on $[0,1]^d$. The parameters $\kappa,h_j$ of the construction are set in a way to ensure $\bb P_\eta \in \cal P_\sf{H}$ and $\TV(\bb P_0, \bb P_\eta) \geq \eps$ with probability $1$ over $\eta$. We have
\begin{align}
    1+ \chi^2(\bb E_\eta \bb P_\eta^{\otimes m}\|\bb P_0^{\otimes m}) &=\int_{[0,1]^{dm}} \left(\bb E_\eta \prod\limits_{i=1}^n p_\eta(x_i)\right)^2 \D x_1 \dots \D x_m\nonumber \\
    &= \bb E_{\eta\eta'}  \langle p_\eta, p_{\eta'}\rangle_{L^2}^m \nonumber \\ 
    &= \bb E (1+\|h_\one\|_2^2\langle\eta,\eta'\rangle)^m \label{eqn:lower bound sketch} \\
    &\leq \exp(m^2 \|h_\one\|_2^4 \kappa^d), \nonumber
\end{align}
where $\eta, \eta'$ are i.i.d. uniform and we assume $\|h_\one\|_2 = \|h_j\|_2$ for all $j\in[\kappa]^d$. The above approach is what Ingster used in his seminal paper \cite{ingster1987minimax} on goodness-of-fit testing, which we adapt to likelihood-free hypothesis testing \eqref{eqn:lfht lower main}. Take $P_0 = \bb P_\eta^{\otimes n} \otimes \bb P_0^{\otimes n} \otimes \bb P_\eta^{\otimes m}$ and $P_1 = \bb P_\eta^{\otimes n} \otimes \bb P_0^{\otimes n} \otimes \bb P_0^{\otimes m}$ in \Cref{lem:lower bound main}. Bounding $\TV(\bb E P_0, \bb E P_1)$ proceeds in multiple steps: first, we drop the $Y$-sample using the data-processing inequality. Then, we use Pinsker's inequality and the chain rule to bound $\TV$ by the $\KL$ divergence of $Z$ conditioned on $X$. We bound $\KL$ by $\chi^2$, arriving at the same equation \eqref{eqn:lower bound sketch}. However, the mixing parameters $\eta,\eta'$ are no longer independent, instead, given $X$ they're independent from the posterior. In the remaining steps we use the fact that the posterior factorizes over the bins and the calculation is reduced to just a single bin where it can be done explicitly. 

Let us now turn to the lower bound in \Cref{THM:HELLINGER GOF}. The difference in the rate is a consequence of the fact that $\H$ and $\TV$ behave differently for densities near zero. Inspired by this, we slightly modify the construction \eqref{eqn:main P_S construction} by putting the perturbations at density level $\eps^2$ as opposed to $1$. Bounding $\TV$ then proceeds analogously to the steps outlined above. 

\subsubsection{Bounded discrete distributions} The construction is entirely analogous to the case of $\cal P_\sf{H}$ and we refer to the \AoScite{supplement \cite{supplement}}{appendix} for details. In the computer science community the construction of $p_\eta$ is attributed to Paninski \cite{paninski2008coincidence}. 

\subsubsection{Gaussian sequence model} The null distribution $\bb P_0$ is the no signal case $\otimes_{i=1}^\infty \cal N(0, 1)$ while the alternative is $\bb P_\theta = \otimes_{i=1}^\infty \cal N(\theta_i, 1)$ where $\theta$ has prior distribution $\otimes_{i=1}^\infty \cal N(0, \gamma_i)$ for an appropriate sequence $\gamma \in \R^\N$. We refer to the \AoScite{supplement \cite{supplement}}{appendix} for more details. 

\subsubsection{General discrete distributions} Once again, the irregular case $\cal P_\sf{D}$ requires special consideration. Clearly the lower bound for $\cal P_\sf{Db}$ carries over. However, in the regime $k \gtrsim 1/\eps^4$ said lower bound becomes suboptimal, and we need a new construction, for which we utilize the moment-matching based approach of Valiant \cite{valiant2011testing} as a black-box. The  construction is derived from that used for two-sample testing by Valiant, namely the pair $(\bb P_\sf{X}, \bb P_\sf{Y})$ is chosen uniformly at random from $\{(p\circ\pi, q\circ\pi)\}_{\pi\in S_k}$. Here we write $S_k$ for the symmetric group on $[k]$ and 
\begin{equation*}
    p(i) = \begin{cases} \frac{1-\eps}{n} &\text{for } i \in[n] \\ \frac{4\eps}{k} &\text{for } i\in[\frac k2, \frac{3k}{4}] \\ 0&\text{otherwise,} \end{cases} 
\end{equation*}
where we assume that $m\leq n \leq k/2$ and define $q(i)=p(i)$ for $i \in [k/2-1]$ and $q(i)=p(3k/2-i)$ for $i\in[k/2,k]$. This construction gives a lower bound matching our upper bound in the regime $m \lesssim n \lesssim k$. The final piece of the puzzle follows by the reduction from two-sample testing with unequal sample size \eqref{eqn:LF TS equiv}, as this shows that likelihood-free hypothesis testing is at least as hard as two-sample testing in the $n\leq m$ regime, and known lower bounds on the sample complexity of two-sample testing \cite{bhattacharya2015testing} (see also \Cref{table:prior results TV}) let us conclude.

\section{Open problems}\label{sec:open problems}

A natural follow-up direction to the present paper would be to study multiple hypothesis testing where $\bb P_\sf{X}$ and $\bb P_\sf{Y}$ are replaced by  $\bb P_{\sf{X}_1}, \dots, \bb P_{\sf{X}_M}$ with corresponding hypotheses $H_1,\dots,H_M$. The geometry of the family $\{\bb P_{\sf{X}_j}\}_{j \in [M]}$ might have interesting effects on the sample complexities. 

\begin{open_problem}
Study the dependence on $M>2$ of likelihood-free testing with $M$ hypotheses. 
\end{open_problem}

Another possible avenue of research is the study of local minimax/instance optimal rates, which is the focus of recent work \cite{valiant2017automatic, balakrishnan2019hypothesis, chhor2020sharp, chhor2021goodness, lam2022local} in the case of goodness-of-fit and two-sample testing.

\begin{open_problem}
Define and study the local minimax rates of likelihood-free hypothesis testing.
\end{open_problem}

Our discussion of the Hellinger case in \Cref{sec:results H} is quite limited, natural open problems in this direction include the following. 
\begin{open_problem}
 Let $\cal P \in \{\cal P_\sf{H}(\beta, d, C), \cal P_\sf{Db}(k,C_\sf{Db}), \cal P_\sf{D}(k)\}$. 
 \begin{enumerate}[(i)]
     \item Study $n_\sf{GoF}$ and $n_\sf{TS}$ for $\cal P$ under Hellinger separation.
     \item Determine the trade-off $\cal R_\sf{LF}$ for $\cal P$ under Hellinger separation.  
 \end{enumerate}
\end{open_problem}
More ambitiously, one might ask for a characterization of `regular` models $(\cal P,\sf{d})$ for which goodness-of-fit testing and two-sample testing are equally hard and the region $\cal R_\sf{LF}$ is given by the trade-off in \Cref{THM:P_SGDB UPPER}.
\begin{open_problem}
 Find a general family of `regular` models $(\cal P, \sf{d})$ for which
 \begin{align*}
     n_\sf{GoF}(\eps, \sf{d}, \cal P) &\asymp n_\sf{TS}(\eps, \sf{d}, \cal P) \text{ and }\\
     \cal R_\sf{LF}(\eps, \sf{d}, \cal P) &\asymp \{m\geq1/\eps^2, n\geq n_\sf{GoF}(\eps, \sf{d}, \cal P), mn \geq n_\sf{GoF}^2(\eps,\sf{d}, \cal P)\}. 
 \end{align*}

\end{open_problem}

Recent follow-up work \cite{gerber2023minimax} showed that Scheff\'e's test is also minimax optimal and achieves the entire trade-off in Figure \ref{fig:phase diagram}. It appears that the optimality of Scheff\'e's test is a consequence of the minimax point of view. Basically, in the worst-case the log-likelihood ratio between the hypotheses is close to being binary, hence quantizing it to $\{0,1\}$ does not lose optimality. Consequently, an important future direction is to better understand the competitive properties of various tests and studying some notion of regret, see \cite{acharya2012competitive} for prior related work. 

\begin{open_problem}
 Study the competitive optimality of likelihood-free hypothesis testing algorithms, and Scheff\'e's test in particular. 
\end{open_problem}

\section*{Acknowledgements}
We thank Julien Chhor for pointing out the connection between flattening and the Jensen-Shannon divergence. PG was supported in part by the NSF award IIS-1838071. YP was supported in part by the NSF under grant No CCF-2131115 and by the MIT-IBM Watson AI Lab.

\bibliographystyle{alpha}
\bibliography{main}

\addtocontents{toc}{\protect\setcounter{tocdepth}{1}}
\appendix

\section{Proof of achievability in Theorem \ref{THM:P_SGDB UPPER} and \ref{THM:P_D UPPER}} \label{sec:thm12 upper proof}
Let $\mu$ be a measure on the measurable space $(\cal X, \cal F)$. Let $\{\phi_i\}_{i \in [r]}$ be a sequence of orthonormal functions in $L^2(\mu)$, where we use the notation $[r]\eqdef \{1,2,\dots,r\}$.
For $f \in L^2(\mu)$, define its projection onto the span
of $\{\phi_1,\dots,\phi_r\}$ as
\begin{equation*}
P_r(f) \eqdef \sum\limits_{i\in[r]} \langle f\phi_i\rangle \phi_i,
\end{equation*}
where we write $\langle\cdot\rangle $ for integration with respect to $\mu$ and $\|\cdot\|_p$ for $\|\cdot\|_{L^p(\mu)}$. Given an i.i.d. sample $X=(X_1,\dots,X_n)$ from some density $f$, define its
empirical projection as
\begin{equation*}
\widehat{P}_r[X] \eqdef \sum\limits_{i \in [r]} \left(\frac1n\sum\limits_{j=1}^n \phi_i(X_j)\right)\phi_i.
\end{equation*}
We define the difference in $L^2$-distances statistics to be
\begin{equation}\label{eqn:T_LF def}
T_\sf{LF} = \|\widehat{P}_r[X]-\widehat{P}_r[Z]\|_2^2-\|\widehat{P}_r[Y]-\widehat{P}_r[Z]\|_2^2, 
\end{equation}
for an appropriate choice of $\mu$ and $\{\phi_j\}_{j\geq1}$ depending on the class $\cal P$. Before calculating the mean and variance, we separate out the diagonal terms in $T_\sf{LF}$ thereby decomposing the statistic into two terms:
\begin{equation}\label{eqn:diag def}
    T_\sf{LF} \eqdef T^{-\sf{d}}_\sf{LF} + \underbrace{\frac{1}{n^2} \sum_{i\in[r]} \sum_{j\in[n]}\big(\phi_i^2(X_j)-\phi_i^2(Y_j)\big)}_{\eqdef D}, 
\end{equation}
which will simplify our proofs somewhat.

To ease notation in the results below, we define the quantities
\begin{equation}\label{eqn:AB def}
\begin{aligned}
A_{fgh} &= \langle f\big[P_r(g-h)\big]^2\rangle \\
B_{fg} &= \sum\limits_{i=1}^r \langle f\phi_i P_r(g\phi_i)\rangle
\end{aligned}
\end{equation}
for $f,g,h \in L^2(\mu)$, assuming the quantities involved are well-defined. We are ready to state our meta-result from which we derive all our likelihood-free hypothesis testing upper bounds.
\begin{proposition}\label{prop:meta lfht}
Let $f,g,h$ denote probability densities on $\cal X$ with respect to $\mu$, and suppose
we observe independent samples $X, Y, Z$
of size $n,n,m$ from $f,g,h$ respectively. Then 
\begin{align*}
\bb E T^{-\sf{d}}_\sf{LF} &= \|P_r(f-h)\|_2^2-\|P_r(g-h)\|_2^2+\frac1n(\|P_r(g)\|_2^2-\|P_r(f)\|_2^2) \\
\var(T^{-\sf{d}}_\sf{LF}) &\lesssim \frac{A_{ffh}+A_{ggh}}{n} + \frac{A_{hfg}}{m} + \frac{\|f+g+h\|_2^4+|B_{fh}|+|B_{gh}|}{nm} \\&\qquad  + \frac{|B_{ff}|+|B_{gg}|+\|f+g+h\|_2^4 + \sqrt{A_{ff0} A_{ffh} + A_{gg0}A_{ggh}}}{n^2} \\
&\qquad + \frac{|B_{ff}|+|B_{gg}| + \|f+g+h\|_2^4 + A_{ff0} + A_{gg0}}{n^3}, 
\end{align*}
where the implied constant is universal.
\end{proposition}

Proposition \ref{prop:meta lfht} is used to test \eqref{eqn:LFHT definition} by rejecting the null whenever $T^{-\sf{d}}_\sf{LF} \geq 0$. To prove that this procedure performs well we show that  $T^{-\sf{d}}_\sf{LF}$ concentrates around its mean by Chebyshev's inequality. For this we find sufficient conditions on the sample sizes $n,m$ so that $(\bb E T^{-\sf{d}}_\sf{LF})^2 \gtrsim \var(T^{-\sf{d}}_\sf{LF})$ for a small enough implied constant on the left.

While \Cref{prop:meta lfht} is enough to conclude the proof of our main theorems, notice that it uses the statistic $T_\sf{LF}^{-\sf{d}}$ which has the diagonal terms removed. For completeness we show that rejecting when $T_\sf{LF}\geq0$ is also minimax optimal, that is, the diagonal term $D$ in \eqref{eqn:diag def} can be included without degrading performance.

\subsection{The class $\cal P_\sf{Db}$}
\begin{proposition}\label{prop:P_Db upper}
For any $C>1$ there exists a constant $c > 0$ such that
\begin{equation*}
    \cal R_\sf{rLF}(\eps, \cal P_\sf{Db}(k,C), \sf{B}_\cdot) \supset c\left\{m\geq 1/\eps^2, n\geq \sqrt{k}/\eps^2, mn \geq k/\eps^4\right\}, 
\end{equation*}
where $\sf{B}_u = \{u \in \cal P_\sf{Db}(k,C) : \|u-v\|_2\leq \eps/(2\sqrt{k})\}$. 
\end{proposition}
\begin{proof}
\textbf{Choice of $\mu$ and $\phi$.}
Take $\cal X = [k]$ and let $\mu=\sum_{i=1}^k \delta_i$ be the counting measure. Let $\phi_i(j) = \one_{\{i=j\}}$ and choose $r=k$
so that $P_r=P_k$ is the identity. By the Cauchy-Schwarz inequality $\|u\|_1\leq\sqrt{k}\|u\|_2$ for all $u \in \bb R^k$.

\textbf{Applying Proposition \ref{prop:meta lfht}.}
Recall the notation of Proposition \ref{prop:meta lfht}, so that $f,g,h$ are the pmfs of $\bb P_\sf{X}, \bb P_\sf{Y}, \bb P_\sf{Z}$ respectively. We analyse the performance of the test $\one\{T^{-\sf{d}}_\sf{LF} \geq 0\}$ under the null hypothesis, the proof under the alternative is analogous. The inequality $$\|f-h\|_2 \leq \frac{\eps}{2\sqrt k} \leq \frac{\|f-g\|_1}{4\sqrt k} \leq \frac{\|f-g\|_2}{4}$$ along with the reverse triangle inequality yields
\begin{align*}
    \|g-h\|_2^2-\|f-h\|_2^2 &\geq (\|f-g\|_2-\|f-h\|_2)^2 - \|f-h\|_2^2 \\
    &= \|f-g\|_2^2 - 2\|f-g\|_2\|f-h\|_2 \\
    &\geq \|f-g\|_2^2/2. 
\end{align*}
Combining the above inequality with \Cref{prop:meta lfht}, we get that $-\E T^{-\sf{d}}_\sf{LF} \geq \|f-g\|_2^2/2 + R$, where the residual term $R$ can be bounded as
\begin{align*}
|R| &= \left| \frac{\|f\|_2^2-\|g\|_2^2}{n}\right| \\
&\leq 2 C \frac{\|f-g\|_2}{n \sqrt k}. 
\end{align*}
Therefore, $- \E T_\sf{LF}^{-\sf d} \geq \frac{\|f-g\|_2^2}{4}$ holds 
provided $2 C \|f-g\|_2/(n\sqrt k) \leq \|f-g\|_2^2/4$, which in turn is implied by $n \gtrsim 1/\eps$ and is thus always satisfied. 

Turning towards the variance, we apply \Cref{prop:meta lfht} to see that 
\begin{equation}\label{eqn:P_Db sufficient}
    \var(T^{-\sf{d}}_\sf{LF}) \lesssim \frac{\|f-g\|_2^2}{k}\left(\frac1n + \frac1m\right) + \frac{1}{k}\left(\frac{1}{n^2} + \frac{1}{nm}\right), 
\end{equation}
where we use the trivial bounds 
\begin{align*}
\|f+g+h\|_2 &\lesssim \sqrt{\frac Ck} \lesssim \sqrt{\frac1k}\\
|B_{ff}| + |B_{gg}| + |B_{fh}| + |B_{gh}| &\lesssim \frac Ck \lesssim\frac1k\\
A_{ffh} + A_{ggh} + A_{hfg} &\lesssim \frac Ck\|f-g\|_2^2 \lesssim \frac1k\|f-g\|_2^2\\
    A_{ff0} + A_{gg0} &\lesssim \left(\frac Ck\right)^2\lesssim\frac{1}{k^2}.   
\end{align*}

Applying Chebyshev's inequality and looking at each term separately in \eqref{eqn:P_Db sufficient} and using that $\|f-g\|_2 \geq \eps/(2\sqrt k)$ yields the desired bounds on $n,m$. 

\textbf{The diagonal.}
While the above test using $T_\sf{LF}^{-\sf{d}}$ already achieves the minimax optimal sample complexity, here we show for completeness that the diagonal $D$, defined in \eqref{eqn:diag def}, can be included without degrading the test's performance. Indeed, we always have
\begin{align*}
    D &= \frac{1}{n^2} \sum_{i\in[r]} \sum_{j \in [n]} \big(\one\{X_j=i\}^2-\one\{Y_j=i\}^2\big) \\
    &= 0.
\end{align*}
Therefore, trivially, the test $\one\{T_\sf{LF} \geq 0\}$ has the same performance as the one analyzed above. 
\end{proof}

\subsection{The class $\cal P_\sf{H}$}
\begin{proposition}\label{prop:P_S upper}
For every $C>1, \beta >0$ and $d\geq1$ there exist two constants $c,c_\sf{r}>0$ such that
\begin{equation}\label{eqn:P_S R definition}
    \cal R_\sf{rLF}(\eps, \cal P_\sf{H}(\beta, d, C), \sf{B}_\cdot) \supset c\left\{m\geq 1/\eps^2, n\geq 1/\eps^{(2\beta+d/2)/\beta}, mn \geq 1/\eps^{2(2\beta+d/2)/\beta}\right\}, 
\end{equation}
where $\sf B_u=\{v\in\cal P_\sf{H}(\beta, d, C):\|v-u\|_2\leq c_\sf{r}\eps\}$. 
\end{proposition}

\begin{proof}
\textbf{Choice of $\mu$ and $\phi$.}
Take $\cal X = [0,1]^d$, let $\mu$ the Lebesgue measure on $\cal X$. Let $\{\phi_i\}_{1 \leq i \leq \kappa^d}$ be the indicators of the cells of the regular grid with $\kappa^d$ bins, normalized to have $L^2(\mu)$-norm equal to $1$, that is, the indicator is multiplied by $\kappa^d$, which is one over the volume of one grid cell. By \cite[Lemma 7.2]{arias2018remember}  for any resolution $r=\kappa^d$ and $u \in \cal C(\beta, d, 2C)$ we have
\begin{equation}\label{eqn:ingster approx}
\|P_r(u)\|_2 \geq c_1\|u\|_2 - c_2 \kappa^{-\beta}
\end{equation}
for constants $c_1,c_2 > 0$ that don't depend on $r$. In particular, the inequalities
\begin{equation}\label{eqn:ingster useful}
    \|P_r(u)\|_2 \geq \frac{c_1}{2} \|u\|_2
\end{equation}
holds for any $\|u\|_2\geq\eps$ provided we choose $\kappa = \left(\frac{2c_2}{c_1\eps}\right)^{1/\beta}$. 

\textbf{Applying Proposition \ref{prop:meta lfht}.} Recall the notation of Proposition \ref{prop:meta lfht} so that $f,g,h$ are the $\mu$-densities of $\bb P_\sf{X}, \bb P_\sf{Y}, \bb P_\sf{Z}$. We analyse the performance of the test $\one\{T^{-\sf{d}}_\sf{LF}\geq0\}$ under the null hypothesis, the proof under the alternative is analogous. Let the radius of robustness be $c_\sf{r} = c_1/4$, and set $\kappa = \left(\frac{2c_2}{c_1\eps}\right)^{1/\beta}$. Then we have
\begin{equation*}
    \|P_r(f-h)\|_2 \leq c_\sf{r}\eps = \frac{c_\sf{r}}{2}\|f-g\|_2 \leq \frac{c_\sf{r}}{c_1}\|P_r(f-g)\|_2
\end{equation*}
by taking $u=f-g$ in \eqref{eqn:ingster useful}. Using the reverse triangle inequality we obtain
\begin{align*}
    \|P_r(g-h)\|_2^2 - \|P_r(f-h)\|_2^2&\geq \left(\|P_r(f-g)\|_2-\|P_r(f-h)\|_2\right)^2 - \|P_r(f-h)\|_2^2 \\
    &= \|P_r(f-g)\|_2^2 - 2\|P_r(f-g)\|_2\|P_r(f-h)\|_2 \\
    &\geq \|P_r(f-g)\|_2^2(1-2\frac{c_\sf{r}}{c_1}) \\
    &= \|P_r(f-g)\|_2^2/2
\end{align*}
Combining the above inequality with \Cref{prop:meta lfht}, we see that $-\E T^{-\sf{d}}_\sf{LF} \geq \|P_r(f-g)\|_2^2/2 + R$ where the residual term $R$ can be bounded as
\begin{align*}
|R| &= \left| \frac{\|f\|_2^2-\|g\|_2^2}{n}\right| \\
&\leq 2C \frac{\|f-g\|_2}{n}.
\end{align*}
Therefore, the inequality $-\E T_\sf{LF}^{-\sf d} \geq \|P_r(f-g)\|_2^2/4$ holds provided $2C\|f-g\|_2/n \leq \|P_r(f-g)\|_2^2/4$, which in turn is implied by $n \gtrsim 1/\eps$ and is thus always satisfied. 

Turning to the variance, using \Cref{prop:meta lfht} we obtain
\begin{equation}\label{eqn:P_S sufficient}
    \var(T^{-\sf{d}}_\sf{LF}) \lesssim 
\|P_r(f-g)\|_2^2\left(\frac1n+\frac1m\right) + \eps^{-d/\beta}\left(\frac{1}{n^2} + \frac{1}{nm}\right), 
\end{equation}
where we apply the trivial inequalities
\begin{align*}
    \|f+g+h\|_2 &\lesssim \sqrt C \lesssim 1\\
    |B_{ff}|+|B_{gg}|+|B_{fh}|+|B_{gh}| &\lesssim Cr = C \kappa^{d} \asymp \eps^{-d/\beta} \\
    A_{ffh} + A_{ggh} + A_{hfg} &\lesssim C \|P_r(f-g)\|_2^2 \lesssim \|P_r(f-g)\|_2^2\\
    A_{ff0} + A_{gg0} &\lesssim C^2 \lesssim1. 
\end{align*}
Applying Chebyshev's inequality and looking at each term separately in \eqref{eqn:P_S sufficient} and using that $\|P_r(f-g)\|_2 \gtrsim \|f-g\|_2 \geq \|f-g\|_1 \geq 2\eps$ yields the desired bounds on $n,m$. 

\textbf{The diagonal.} While the above test using $T_\sf{LF}^{-\sf{d}}$ already achieves the minimax optimal sample complexity, for completeness we also note that including the diagonal terms $D$ defined in \eqref{eqn:diag def} doesn't degrade performance. This follows from the simple fact that $D=0$, which is true for reasons analogous to the case of $\cal P_\sf{Db}$ that we already covered. 
\end{proof}

\subsection{The class $\cal P_\sf{G}$}
\begin{proposition}\label{prop:P_G upper}
For all $s,C > 0$ there exists a constant $c>0$ such that
\begin{equation*}
    \cal R_\sf{rLF}(\eps, \cal P_\sf{G}(s, C),\sf{B}_\cdot) \supset c\left\{m\geq 1/\eps^2, n\geq 1/\eps^{(2s+1/2)/s}, mn \geq 1/\eps^{2(2s+1/2)/s}\right\}, 
\end{equation*}
where $\sf{B}_{\mu_\theta}=\{\mu_{\theta'}:\theta'\in\cal E(s, C), \|\theta-\theta'\|_2\leq \eps/4\}$. 
\end{proposition}
\begin{proof}
\textbf{Choosing $\mu$ and $\phi$.}
Let $\cal X = \bb R^{\bb N}$ be the set of infinite sequences and take as the base measure $\mu=\otimes_{d=1}^\infty \cal N(0,1)$, 
the infinite dimensional standard Gaussian. For $\theta \in \ell^2$ write $\mu_\theta = \otimes_{d=1}^\infty \cal N(\theta_i,1)$ so that $\mu_0=\mu$. Take the orthonormal functions $\phi_i(x)=x_i$ in $L^2(\mu)$ for $i \geq 1$, so that 
\begin{equation*}
    P_r\left(\frac{\rm{d}\mu_\theta}{\rm{d}\mu}\right) = \sum\limits_{i=1}^r x_i \theta_i. 
\end{equation*}
Let $\theta,\theta' \in \cal E(s, C)$ with $\TV(\mu_\theta, \mu_{\theta'}) \geq \eps$. By direct computation we obtain
\begin{equation}\label{eqn:P_G approx}
\|P_r\left(\frac{\rm{d}\mu_\theta}{\rm{d} \mu}-\frac{\rm{d}\mu_{\theta'}}{\rm{d} \mu}\right)\|_2^2 = \sum\limits_{i=1}^r (\theta_i-\theta'_i)^2 \geq \|\theta-\theta'\|_2^2 - r^{-2s}\sum\limits_{i>r} (\theta_i-\theta'_i)^2 i^{2s} \geq \|\theta-\theta'\|_2^2 - 4C^2r^{-2s}. 
\end{equation}
In particular, the inequality
\begin{equation}\label{eqn:P_G approx useful}
    \|P_r\left(\frac{\D\mu_\theta}{\D\mu}-\frac{\D\mu_{\theta'}}{\D\mu}\right)\|_2^2 \geq \frac12\|\theta-\theta'\|_2^2
\end{equation}
holds for all $\theta,\theta' \in \cal E(s,C)$ with $\|\theta-\theta'\|_2 \geq \eps$, provided we take $r=(4C/\eps)^{1/s}$. 

\textbf{Applying Proposition \ref{prop:meta lfht}.}
Recall the notation of Proposition \ref{prop:meta lfht}, and let $f,g,h$ be the $\mu$-densities of $\bb P_\sf{X} = \mu_{\theta_\sf{X}}, \bb P_\sf{Y} = \mu_{\theta_\sf{Y}}, \bb P_\sf{Z}=\mu_{\theta_\sf{Z}}$ respectively. We analyse the test $\one\{T^{-\sf{d}}_\sf{LF} \geq 0\}$ only under the null hypothesis, as the analysis under the alternative is analogous. Note also that by \Cref{lem:pinsker} the inequality
\begin{equation*}
\TV(\mu_\theta, \mu_{\theta'}) \leq \H(\mu_\theta, \mu_{\theta'}) =\sqrt{2(1-\exp(-\|\theta-\theta'\|_2^2/8))} \leq \frac{\|\theta-\theta'\|_2}{2}
\end{equation*}
holds for any $\theta,\theta'\in\ell^2$. Therefore, we have
\begin{align*}
    \|P_r(f-h)\|_2 \leq \frac{\eps}{4} \leq \frac{\TV(\mu_\theta, \mu_{\theta'})}{4} \leq \frac{\|\theta-\theta'\|_2}{8} \leq \frac{\|P_r(f-g)\|_2}{4}
\end{align*}
by \eqref{eqn:P_G approx useful}. 

By the reverse triangle inequality we have
\begin{align*}
    \|P_r(g-h)\|_2^2-\|P_r(f-h)\|_2^2 &\geq \left(\|P_r(f-g)\|_2-\|P_r(f-h)\|_2\right)^2 - \|P_r(f-h)\|_2^2 \\
    &= \|P_r(f-g)\|_2^2 - 2\|P_r(f-g)\|_2\|P_r(f-h)\|_2 \\
    &\geq \|P_r(f-g)\|_2^2 /2
\end{align*}
Combining the inequality above with \Cref{prop:meta lfht}, we see that  $-\E T^{-\sf{d}}_\sf{LF} \geq \|P_r(f-g)\|_2^2/2 + R$, where the residual term $R$ can be bounded as
\begin{align*}
|R| &= \left| \frac{\|P_r(f)\|_2^2-\|P_r(g)\|_2^2}{n}\right| \\
&\leq 2C \frac{\|P_r(f-g)\|_2}{n}.
\end{align*}
Therefore, $-\E T_\sf{LF}^{-\sf d} \geq \|P_r(f-g)\|_2^2/4$ provided $2C\|P_r(f-g)\|_2/n \leq\|P_r(f-g)\|_2^2/4$, which in turn is implied by $n \gtrsim 1/\eps$ and is therefore always satisfied. 

Let us turn to the variance of the statistic. Let $u,v,t$ be the $\mu$-densities of the distributions $\mu_{\theta}, \mu_{\theta'}, \mu_{\theta''}$ for some vectors $\theta,\theta',\theta'' \in \cal E(s,C)$ in the Sobolev ellipsoid. Straightforward calculations involving Gaussian random variables produce 
\begin{align*}
    A_{uvt} &= \sum\limits_{ij}^r(\one(i=j)+\theta_i\theta_j)(\theta'_i-\theta''_i)(\theta'_j-\theta''_j) \leq (1+C^2)\|P_r(v-t)\|_2^2 \lesssim \|P_r(v-t)\|_2^2 \lesssim C^2 \lesssim 1\\
    \|u\|_2 &= \bb  \exp\left(\frac12\|\theta\|_2^2\right) \leq \exp(C^2/2) \lesssim 1 \\
    B_{uv} &= \sum\limits_{i=1}^r\Big(1+\theta_i^2+{\theta'}_i^2+\theta_i\theta'_i\sum_{j=1}^r \theta_j{\theta'}_j\Big) \\
    &\leq r + 2C^2 + C^4 \\
    &\lesssim r. 
\end{align*}
Applying \Cref{prop:meta lfht} tells us that 
\begin{equation}\label{eqn:P_G sufficient}
    \var(T^{-\sf{d}}_\sf{LF}) \lesssim \|P_r(f-g)\|_2^2\left(\frac1n+\frac1m\right) + \eps^{-1/s}\left(\frac{1}{n^2} + \frac{1}{nm}\right)
\end{equation}
Applying Chebyshev's inequality and looking at each term separately in \eqref{eqn:P_G sufficient} and using that $\TV(\mu_\theta,\mu_{\theta'}) \lesssim \|P_r(f-g)\|$ yields the desired bounds on $n,m$. 

\textbf{The diagonal.} While the above test using $T_\sf{LF}^{-\sf{d}}$ already achieves the minimax optimal sample complexity, for completeness we show that including the diagonal terms $D$ defined in  \eqref{eqn:diag def} doesn't degrade performance. To this end we compute
\begin{align*}
    \E D &= \E \frac{1}{n^2} \sum_{i \in [r]} \sum_{j \in [n]} \big( \phi_i^2(X_j)-\phi_i^2(Y_j)\big) \\
    &= \frac{1}{n} \sum_{i \in [r]}(\theta_{\sf{X},i}^2-\theta_{\sf{Y},i}^2) \\
    &\leq \frac1n \|\theta_\sf{X}+\theta_\sf{Y}\|_2 \sqrt{\sum_{i \in [r]}(\theta_{\sf X,i}-{\theta_{\sf Y,i}}^2)} \\
    &\leq 2C \frac{\|P_r(f-g)\|_2}{n}.
\end{align*}
We see that $|\E T^{-\sf{d}}_\sf{LF}| \gtrsim |\E D|$ as soon as $n\gtrsim 1/\eps$. Turning to the variance, we have
\begin{align*}
\var(D) &= \frac{1}{n^3} \sum_{i \in [r]} \big(\var(\phi_i^2(X_1)) + \var(\phi_i^2(Y_1))\big) \\
&\lesssim \frac{r C^2}{n^3}, 
\end{align*}
and so the diagonal terms do not inflate the variance by more than a constant factor. Therefore, the sample complexity of the test is unchanged. 
\end{proof}

\subsection{The class $\cal P_\sf{D}$}
\begin{proposition}\label{prop:P_D upper}
Let $\alpha = 1 \lor \left(\frac kn \land \frac km\right)$. There exist constants $c,c',c_\sf{r} > 0$ such that 
\begin{equation*}
    \cal R_\sf{rLF}(\eps, \cal P_\sf{D}(k), \sf{B}_\cdot) \supset \frac{c}{\log(k)} \left\{m\geq 1/\eps^2, n\geq \sqrt{k\alpha}/\eps^2, mn \geq k\alpha/\eps^4\right\}, 
\end{equation*}
where $\sf{B}_u=\{v:\|u-v\|_2\leq c_\sf{r}\eps/\sqrt{k}, \|v/u\|_\infty\leq c'\}$. 
\end{proposition}

\begin{proof}
\textbf{Choosing $\mu$ and $\phi$.}
As for $\cal P_\sf{Db}$, we take $\cal X = [k]$, $\mu=\sum_{i=1}^k \delta_i$, $\phi_i(j) = \one_{\{i=j\}}$ and $r=k$. By the Cauchy-Schwarz inequality $\|h\|_1\leq\sqrt{k}\|h\|_2$ for all $h \in \bb R^k$. 


\textbf{Reducing to the small-norm case.}
Before applying Proposition \ref{prop:meta lfht} we need to `pre-process` our distributions. For an in-depth explanation of this technique see \cite{diakonikolas2016new, goldreich2017introduction}. 
Recall that we write $f,g,h$ for the probability mass functions of $\bb P_\sf{X},\bb P_\sf{Y},\bb P_\sf{Z}$ respectively, from which we observe the samples $X,Y,Z$ of size $n,n,m$ respectively. Recall also that the null hypothesis is that $\|f-h\|_2\leq c_\sf{r}\eps/\sqrt{k}$ while the
alternative says that $\|g-h\|_2\leq c_\sf{r}\eps/\sqrt{k}$, with $\|f-g\|_2\geq2\eps/\sqrt{k}$ guaranteed under both. In the following section we use the standard inequality $\bb P(\lambda-x\geq\poi(\lambda))\leq\exp(-\frac{x^2}{2(\lambda+x)})$ valid for all $x\geq0$ repeatedly. We also utilize the identity
\begin{align}\label{eqn:inverse poisson exp}
    \bb E\left[\frac{1}{\poi(\lambda)+1}\right] &= \begin{cases} 1 &\text{if } \lambda = 0 \\ \frac{1-e^{-\lambda}}{\lambda} &\text{if }\lambda > 0, \end{cases}
\end{align}
which is easily verified by direct calculation. Finally, the following Lemma will come handy.
\begin{proposition}{{{\cite[Corollary 11.6]{goldreich2017introduction}}}}\label{prop:Goldreich norm est}
Given $t$ samples from an unknown discrete distribution $p$, there exists an algorithm that produces an estimate $\widehat{\|p\|_2^2}$ with the property
\begin{equation*}
    \bb P( \widehat{\|p\|_2^2} \notin (\frac12\|p\|_2^2, \frac32\|p\|_2^2)) \lesssim \frac{1}{\|p\|_2 t}, 
\end{equation*}
where the implied constant is universal. 
\end{proposition}

First we describe a random ``filter" $F:\cal P_\sf{D}(k) \to \cal P_\sf{D}(K)$ that maps distributions on $[k]$ to distributions on the inflated alphabet $[K]$. Let $(n_\sf{X},n_\sf{Y},n_\sf{Z}) = \frac12(n\land k, n\land k, m \land k)$ and let $N^\sf{X} \sim \poi(n_\sf{X}/2)$ independently of all other randomness, and define $N^\sf{Y}, N^\sf{Z}$ similarly. We take the first $N^\sf{X}, N^\sf{Y}, N^\sf{Z}$ samples from the data sets $X,Y,Z$ respectively. In the event $N^\sf{X}\lor N^\sf{Y} > n$ or $N^\sf{Z} > m$ let our output to the likelihood-free hypothesis test be arbitrary, this happens with exponentially small probability. Let $N^\sf{X}_i$ be the number of the samples $X_1, \dots, X_{N^\sf{X}}$ falling in bin $i$, so that $N^\sf{X}_i \sim \poi(n_\sf{X}f_i/2)$ independently for each $i\in[k]$, and define $N^\sf{Y}_i,N^\sf{Z}_i$ analogously. The filter $F$ is defined as follows: $$\text{divide each support element $i\in\{1,2,\dots,k\}$ uniformly into $1+N^\sf{X}_i+N^\sf{Y}_i+N^\sf{Z}_i$ bins.}$$The filter has the following properties trivially:
\begin{enumerate}
\item The construction succeeds with probability $\geq 1-3\exp(-n\land m \land k/16)$, focus on this event from here on. 
    \item The construction uses at most $n_\sf{X}, n_\sf{Y}, n_\sf{Z}$ samples from $X,Y,Z$ respectively and satisfies $K\leq 5k/2$. 
    \item For any $u,v \in \cal P_\sf{D}(k)$ we have $\TV(F(u),F(v))=\TV(u,v)$ and $\|F(u)-F(v)\|_2 \leq \|u-v\|_2$.
    \item Given a sample from an unknown $u \in \cal P_\sf{D}(k)$ we can generate a sample from $F(u)$ and vice-versa.
\end{enumerate}
Let $\tilde{f} \eqdef F(f)$ be the probability mass function after processing and define $\tilde{g},\tilde{h}$ analogously. By properties $1-2$ of the filter, we may assume with probability $99\%$ that the new alphabet's size is at most $5k/2$ and that we used at most half of our samples $X,Y,Z$. We immediately get $2\eps \leq \|f-g\|_1 = \|\tilde{f}-\tilde{g}\|_1 \leq \sqrt{5k/2}\|\tilde{f}-\tilde{g}\|_2$ and $\|\tilde{f}-\tilde{h}\|_2\leq\|f-h\|_2, \|\tilde{g}-\tilde{h}\|_2\leq\|g-h\|_2$. Notice that $$\sum_{i\in[K]} \tilde f_i\tilde g_i = \sum_{i\in[k]} \frac{f_ig_i}{1 + N_i^\sf{X} + N_i^\sf{Y} + N_i^\sf{Z}}$$holds, and similar statements can be derived for the inner product between $\tilde f, \tilde h$ etc. Recall that we set $$\alpha = \max\left\{1, \min\left\{\frac kn, \frac km\right\}\right\}.$$Adopting the convention $0/0=1$ and using \eqref{eqn:inverse poisson exp} we can bound inner products between the mass functions as
\begin{align*}
    \bb E\left[B_{\tilde{f}\tilde{h}} + B_{\tilde{g}\tilde{h}}\right] = \bb E\left[\langle \tilde{f}\tilde{h}\rangle + \langle\tilde{g}\tilde{h}\rangle\right] &\leq 4\sum\limits_{i \in [k]} \frac{f_ih_i+g_ih_i}{(n\land k)(f_i+g_i)+(m\land k)h_i} \leq \frac{8}{(n \lor m)\land k} = \frac{8\alpha}{k} \\
    \bb E\left[B_{\tilde{f}\tilde{f}} + B_{\tilde{g}\tilde{g}}\right] = \bb E \left[\|\tilde{f}\|_2^2 + \|\tilde{g}\|_2^2\right] &\leq 4\sum_{i \in [k]} \frac{f_i^2 + g_i^2}{(n\land k)(f_i+g_i) + (m\land k)h_i} \leq \frac{8}{n\land k} \\
    \bb E\|\tilde{h}\|_2^2 &\leq 4\sum_{i\in[k]} \frac{h_i^2}{(n\land k)(f_i+g_i) + (m\land k)h_i} \leq \frac{4}{m \land k}.
\end{align*}
By Markov's inequality we may assume that the inequalities in the preceding display hold not only in expectation but with $99\%$ probability overall with universal constants. Notice that under the null hypothesis $\|\tilde{f}-\tilde{h}\|_2\leq c_\sf{r}\eps/\sqrt{k}$ and thus $\|\tilde{f}\|_2 \leq \|\tilde{h}\|_2 + c_\sf{r}\eps/\sqrt{k} \leq \|\tilde{f}\|_2 + 2c_\sf{r}\eps/\sqrt{k}$, and similarly with $\tilde{f}$ replaced by $\tilde{g}$  under the alternative. We restrict our attention to $c_\sf{r} \in (0,1)$ so that $c_\sf{r}$ is treated as a constant where appropriate. Notice that $\eps/\sqrt{k} \lesssim 1/\sqrt{(n\lor m)\land k}$ holds trivially. Thus, we obtain $\|\tilde{f}\|_2 \lor \|\tilde{h}\|_2 \leq c/\sqrt{(m \lor n)\land k}$ under the null and $\|\tilde{g}\|_2 \lor \|\tilde{h}\|_2 \leq c/\sqrt{(n\lor m)\land k}$ under the alternative for a universal constant $c$. We would like to ensure that 
\begin{equation}\label{eqn:reduction norm goal}
    \|\tilde{f}\|_2 \lor \|\tilde{g}\|_2 \lor \|\tilde{h}\|_2 \lesssim \frac{1}{\sqrt{(m \lor n)\land k}} = \sqrt{\frac{\alpha}{k}}.
\end{equation}
To this end we apply Proposition \ref{prop:Goldreich norm est} using $(n/4,n/4)$ of the remaining, transformed but otherwise untouched $X,Y$ samples. Let $\widehat{\|\tilde{f}\|_2^2}, \widehat{\|\tilde{g}\|_2^2}$ denote the estimates, which lie in $(\frac12\|\tilde{f}\|_2^2, \frac32\|\tilde{f}\|_2^2)$ and $(\frac12\|\tilde{g}\|_2^2, \frac32\|\tilde{g}\|_2^2)$ respectively, with probability at least $1-\cal O((|\tilde{f}\|_2^{-1}+\|\tilde{g}\|_2^{-1})/n) \geq 1 - \cal O(\sqrt{k}/n)$, since $\|\tilde{f}\|_2\land\|\tilde{g}\|_2\geq\sqrt{2/(5k)}$ by the Cauchy-Schwarz inequality. Assuming that $n \gtrsim \sqrt{k}$ this probability can be taken to be arbitrarily high, say $99\%$. Now we perform the following procedure: if $\widehat{\|\tilde{f}\|_2^2} > \frac32c^2/((n\lor m)\land k)$ reject the null hypothesis, otherwise if $\widehat{\|\tilde{g}\|_2^2} > \frac32c^2/((n\lor m)\land k)$ accept the null hypothesis, otherwise proceed with the assumption that \eqref{eqn:reduction norm goal} holds. By design this process, on our $97\%\leq$ probability event of interest, correctly identifies the hypothesis or correctly concludes that \eqref{eqn:reduction norm goal} holds. The last step of the reduction is ensuring that the quantities $A_{\tilde{f}\tilde{f}\tilde{h}}, A_{\tilde{g}\tilde{g}\tilde{h}}, A_{\tilde{h}\tilde{f}\tilde{g}},A_{\tilde f\tilde f0}, A_{\tilde g\tilde g0}$ are small. The first two and last two may be bounded easily as
\begin{equation}\label{eqn:A_ffh A_ggh bound}
\begin{aligned}
    A_{\tilde{f}\tilde{f}\tilde{h}} + A_{\tilde{g}\tilde{g}\tilde{h}} &= \langle\tilde{f}(\tilde{f}-\tilde{h})^2\rangle + \langle\tilde{g}(\tilde{g}-\tilde{h})^2\rangle \\
    &\leq \|\tilde{f}\|_2\|\tilde{f}-\tilde{h}\|_4^2 + \|\tilde{g}\|_2\|\tilde{g}-\tilde{h}\|_4^2 \\
    &\lesssim \frac{\|\tilde{f}-\tilde{h}\|_2^2 + \|\tilde{g}-\tilde{h}\|_2^2}{\sqrt{(n\lor m)\land k}} \\
    &\lesssim \frac{\|\tilde{f}-\tilde{g}\|_2^2 + c_\sf{r}^2\eps^2/k}{\sqrt{(n\lor m)\land k}} \lesssim \frac{\|\tilde{f}-\tilde{g}\|_2^2}{\sqrt{(n \lor m)\land k}} = \sqrt{\frac{\alpha}{k}}\|\tilde f-\tilde g\|_2^2\\
    A_{\tilde f\tilde f0} + A_{\tilde g\tilde g0} &= \|\tilde f\|_3^3 + \|\tilde g\|_3^3 \leq \|\tilde f\|_2^3 + \|\tilde g\|_2^3 \lesssim \frac{1}{((n\lor m)\land k)^{3/2}} = \left(\frac{\alpha}{k}\right)^{3/2}. 
\end{aligned}
\end{equation}
To bound $A_{\tilde{h}\tilde{f}\tilde{g}}$ we need a more sophisticated method. Recall that by definition 
\begin{equation*}
    A_{\tilde{h}\tilde{f}\tilde{g}} = \sum_{i \in [k]} \frac{h_i(f_i-g_i)^2}{(1+N^\sf{X}_i+N^\sf{Y}_i+N^\sf{Z}_i)^2}. 
\end{equation*}
Fix an $i\in[k]$ and let $P \eqdef N^\sf{X}_i+N^\sf{Y}_i+N^\sf{Z}_i \sim \poi((n\land k)(f_i+g_i)/4+(m\land k)h_i/4)$ and take a constant $c > 0$ to be specified. We have
\begin{align*}
    \bb P\left(\frac{1}{1+P} > c\log(k)\frac{1}{\bb E P}\right) &= \begin{cases} 0&\text{if } \bb E P \leq c\log(k) \\ \bb P\left(\bb E P - \left(\bb E P\left(1-\frac{1}{c\log(k)}\right)+1\right) > P\right) &\text{if } \bb E P > c\log(k). \end{cases}
\end{align*}
Assuming that $i$ is such that $\bb E P \geq c \log(k)$ and taking $k$ large enough so that $c \log(k) \geq 2$, we can proceed as
\begin{align*}
    \bb P\left(\bb E P - \left(\bb E P\left(1-\frac{1}{c\log(k)}\right)+1\right) > P\right)  &\leq \exp(-\frac12\frac{(\bb E P(1-\frac{1}{c\log(k)})+1)^2}{\bb E P(2-\frac{1}{c \log(k)})+1}) \\
    &\leq \exp(-\frac{1}{16}\bb E P) \\
    &\leq \frac{1}{k^{c/16}}. 
\end{align*}
Choosing $c = 32$ and taking a union bound, the inequality
\begin{equation*}
    A_{\tilde{h}\tilde{f}\tilde{g}} \lesssim \frac{\log(k)}{m\land k}\sum_{i \in [k]} \frac{(f_i-g_i)^2}{1+N^\sf{X}_i+N^\sf{Y}_i+N^\sf{Z}_i} \asymp \frac{\log(k)}{m\land k} \|\tilde{f}-\tilde{g}\|_2^2
\end{equation*}
holds with probability at least $1-1/k$. Using that $\|h/f\|_\infty \land \|h/g\|_\infty \lesssim 1$ by assumption, we obtain $A_{\tilde{h}\tilde{f}\tilde{g}} \lesssim \frac{\log(k)}{n\land k} \|\tilde{f}-\tilde{g}\|_2^2$ similarly. Combining the two bounds yields
\begin{equation}\label{eqn:A_hfg bound}
    A_{\tilde{h}\tilde{f}\tilde{g}} \lesssim \frac{\log(k)}{(m\lor n)\land k} \|\tilde{f}-\tilde{g}\|_2^2 = \frac{\log(k)\alpha}{k}\|\tilde f-\tilde g\|_2^2. 
\end{equation}
To summarize, under the assumptions that $n\gtrsim\sqrt{k}$, and at the cost of inflating the alphabet size to at most $\frac52k$ and a probability of error at most $3\%+\frac1k$, we may assume that the inequalities \eqref{eqn:reduction norm goal}, \eqref{eqn:A_ffh A_ggh bound} and \eqref{eqn:A_hfg bound} hold with universal constants.

\textbf{Applying Proposition \ref{prop:meta lfht}.}
We only analyse the type-I error, as the type-II error follows analogously. As explained earlier, we apply the test $\one\{T^{-\sf{d}}_\sf{LF}\geq0\}$ to the transformed samples with probability mass functions $\tilde f, \tilde g, \tilde h$. Note that taking $c_\sf{r}$ small eonugh shows that 
\begin{equation*}
    \|\tilde{g}-\tilde{h}\|_2^2-\|\tilde{f}-\tilde{h}\|_2^2 \gtrsim \|\tilde{f}-\tilde{g}\|_2^2
\end{equation*}
for a universal implied constant. Therefore, by \Cref{prop:meta lfht} we see that $-\E T^{-\sf{d}}_\sf{LF} \geq c \|\tilde f - \tilde g\|_2^2 + R$ for some universal constant $c > 0$, where the residual term $R$ can be bounded as 
\begin{align*}
    |R| &= \left|\frac{\|\tilde f\|_2^2-\|\tilde g\|_2^2}{n}\right| \\
    &\lesssim \frac{\|\tilde f-\tilde g\|_2}{n\sqrt{k\land (m\lor n)}}, 
\end{align*}
where we used \eqref{eqn:reduction norm goal}. We have $-\E T^{-\sf{d}}_\sf{LF} \gtrsim \|\tilde f-\tilde g\|_2^2$ provided $n\gtrsim 1/(\|\tilde f-\tilde g\|_2\sqrt{k\land(m\lor n)}) \asymp \sqrt{\alpha}/\eps$, which we assume from here on. Plugging in the bounds derived above, the test $\one\{T_\sf{LF} \geq 0\}$ on the transformed observations has type-I probability of error bounded by $1/3$ provided
\begin{equation*}
    \|\tilde{f}-\tilde{g}\|_2^4 \gtrsim  \frac1n \sqrt{\frac{\alpha}{k}} \|\tilde{f}-\tilde{g}\|_2^2 + \frac1m\frac{\log(k)\alpha}{k}\|\tilde{f}-\tilde{g}\|_2^2 + \frac{\alpha}{k}\left(\frac{1}{nm}+\frac{1}{n^2}\right)
\end{equation*}
for a small enough implied constant on the left. Looking at each term separately yields the sufficient conditions
\begin{equation}
    \underbrace{m \gtrsim \frac{\log(k)\alpha}{\eps^2}}_{(I)} \qquad\text{and}\qquad
    n \gtrsim \frac{\sqrt{k\alpha}}{\eps^2} \qquad\text{and}\qquad
    mn \gtrsim \frac{k\alpha}{\eps^4}.\label{eqn:mn sufficient}
\end{equation}

The final step is to check that the sufficient conditions in \eqref{eqn:mn sufficient} are implied by what is indicated in the statement of \Cref{THM:P_D UPPER}. Recall from the statement of the Theorem, that it states that 
\begin{equation}\label{eqn:P_D thm sufficient}
    m \gtrsim \frac{\log(k)}{\eps^2} \qquad\text{and}\qquad n \gtrsim \frac{\sqrt{k\alpha}}{\eps^2} \qquad\text{and}\qquad mn \gtrsim \frac{k\log(k)\alpha}{\eps^4}
\end{equation}
is sufficient to successfully perform the test, where we have replaced the generic $\gtrsim_{\log(k)}$ notation with the precise dependence on $\log(k)$ that we require. Note that the only difference between \eqref{eqn:mn sufficient} and \eqref{eqn:P_D thm sufficient} is the condition on $m$, that is, the first term in the equations \eqref{eqn:mn sufficient} and \eqref{eqn:P_D thm sufficient}. Suppose now that \eqref{eqn:P_D thm sufficient} holds, and let us split this discussion into cases. 
\begin{enumerate}
    \item Suppose $\max\{m,n\} \geq k$. In this case $\alpha = 1$, and $(I)$ is implied by $m \gtrsim \log(k)/\eps^2$. For this the first condition of \eqref{eqn:P_D thm sufficient} is clearly sufficient. 
    \item Suppose $n \leq m \leq k$. In this case $\alpha=k/m$, and $(I)$ is implied by $m\gtrsim\sqrt{k\log(k)}/\eps$. By the third condition of \eqref{eqn:P_D thm sufficient} we know that $m^2n\gtrsim k^2/\eps^4$. Using that $n \leq m$, this implies that $m \gtrsim k^{2/3}/\eps^{4/3}$, which is clearly sufficient. 
    \item Suppose $m \leq n \leq k$. In this case $\alpha=k/n$, and $(I)$ is implied by $mn \gtrsim k\log(k)/\eps^2$. By the third condition of \eqref{eqn:P_D thm sufficient} we know that $mn^2\gtrsim k^2\log(k)/\eps^4$. After noting that $n \leq k$ we get $mn \gtrsim k\log(k)/\eps^4$, which is sufficient. 
\end{enumerate}

\textbf{The diagonal.}
See the discussion at the end of the proof for $\cal P_\sf{Db}$. 
\end{proof}

\section{Lower bounds of Theorem \ref{THM:P_SGDB UPPER} and \ref{THM:P_D UPPER}}\label{sec:thm12 lower proof}
Let $\cal M(\cal X)$ be the set of all probability measures on some space $\cal X$, and $\cal P\subseteq \cal M(\cal X)$ be some family of distributions. In this section we prove lower bounds for likelihood-free hypothesis testing problems. For clarity, let us formally state the problem as testing between the hypotheses
\begin{equation}\label{eqn:lfht lower}
\begin{aligned}
    H_0&=\{\bb P_\sf{X}^{\otimes n} \otimes \bb P_\sf{Y}^{\otimes n} \otimes \bb P_\sf{X}^{\otimes m}:\,\bb P_\sf{X},\bb P_\sf{Y}\in\cal P,\,\TV(\bb P_\sf{X},\bb P_\sf{Y}) \geq \eps\} \\
    &\text{versus} \\
    H_1&=\{\bb P_\sf{X}^{\otimes n} \otimes \bb P_\sf{Y}^{\otimes n} \otimes P_\sf{Y}^{\otimes m}:\,\bb P_\sf{X},\bb P_\sf{Y}\in\cal P,\,\TV(\bb P_\sf{X},\bb P_\sf{Y}) \geq \eps\}.
\end{aligned}
\end{equation}
Our strategy for proving lower bounds relies on the following well known result proved in the main text.
\begin{lemma}\label{lem:TV lower}
Take hypotheses $H_0,H_1 \subseteq \cal M(\cal X)$ and $P_0,P_1 \in \cal M(\cal X)$ random. Then
\begin{align*}
    \inf\limits_{\psi}\max_{i=0,1}\sup\limits_{P \in H_i} P(\psi\neq i) \geq \frac12\left(1-\TV(\bb E P_0, \bb E P_1)\right)-\sum_i\bb P(P_i\notin H_i),
\end{align*}
where the infimum is over all tests $\psi:\cal X \to \{0,1\}$.
\end{lemma}
The following will also be used multiple times throughout:
\begin{lemma}[{{{\cite[Lemmas 2.3 and 2.4]{tsybakov}}}}]\label{lem:pinsker}
For any probability measures $\bb P_0,\bb P_1$, 
\begin{equation*}
    \frac{1}{4}\H^4(\bb P_0, \bb P_1) \leq \TV^2(\bb P_0,\bb P_1) \leq \H^2(\bb P_0, \bb P_1) \leq \KL(\bb P_0 \| \bb P_1) \leq \chi^2(\bb P_0\|\bb P_1).
\end{equation*}
\end{lemma}
Note that some of the inequalities in \Cref{lem:pinsker} can be improved, but since such improvements have no effect on our results, we present their simplest available version.  The inequalities between $\TV$ and $\H$ are attributed to Le Cam, while the bound $\TV\leq\sqrt{\KL/2}$ is due to Pinsker. The use of the $\chi^2$-divergence for bounding the total variation distance between mixtures of products was pioneered by Ingster \cite{ingster2003nonparametric}, and is sometimes referred to as the \textit{Ingster-trick}. \\

In our bounds we will also rely on the following simple technical result. 

\begin{lemma}\label{lem:multi expectation}
Suppose that $a,b,c > 0$ and $N=(N_1,\dots,N_k) \sim \operatorname{Multinomial}(n, (\frac1k,\dots,\frac1k))$. Then
\begin{equation*}
    \bb E_N \prod\limits_{j \in [k]} (a+b(1+c)^{N_j}) \leq (a+be^{cn/k})^k.
\end{equation*}
\end{lemma}

Recall that the necessity of $m\gtrsim n_\sf{HT}(\eps, \cal P)$ and $n\gtrsim n_\sf{GoF}(\eps, \cal P)$ were shown in \Cref{PROP:REDUCTIONS}. Thus, most of our work lies in obtaining the lower bound on the product $m n$. 

\subsection{The class $\cal P_\sf{H}$}\label{sec:P_S lower}
\begin{proposition}\label{prop:P_S lower}
For any $\beta>0, C>1$ and $d\geq1$ there exists a finite $c$ independent of $\eps$ such that 
\begin{equation*}
    c \{m\geq 1/\eps^2, n\geq \eps^{-(2\beta+d/2)/\beta}, mn \geq \eps^{-2(2\beta+d/2)/\beta}\} \supseteq \cal R_\sf{LF}(\eps, \cal P_\sf{H}(\beta, d, C))
\end{equation*}
for all $\eps \in (0,1)$. 
\end{proposition}
\begin{proof}
\textbf{Adversarial construction.}
Take a smooth function $h:\bb R^d \to \bb R$ supported on of $[0,1]^d$ with $\int_{[0,1]^d} h(x)\D x=0$ and $\int_{[0,1]^d} h(x)^2\D x=1$.
Let $\kappa \geq 1$ be an integer, and for $j \in [\kappa]^d$ define the scaled and translated functions $h_j$ as
\begin{equation*}
    h_j(x) = \kappa^{d/2} h(\kappa x - j+1).
\end{equation*}
Then $h_j$ is supported on the cube $[(j-1)/\kappa, j/\kappa]$ and $\int_{[0,1]^d} h_j(x)^2\D x = 1$, where we write $j/\kappa=(j_1/\kappa, \dots, j_d/\kappa)$.
Let $\rho > 0$ be small and for each $\eta \in \{-1,0, 1\}^{\kappa^d}$
define the function
\begin{equation*}
    f_\eta(x) = 1 + \rho \sum\limits_{j \in [\kappa]^d} \eta_j h_{j}(x).
\end{equation*}
In particular, $f_0=1$ is the uniform density. Clearly $\int_{[0,1]^d} f_\eta(x)\D x = 1$, and to make it positive we choose $\rho,\kappa$ such that $\rho \kappa^{d/2}\|h\|_\infty \leq 1/2$. By \cite{arias2018remember}, choosing
\begin{equation}\label{eqn:density condition}
\rho\kappa^{d/2+\beta} \leq C/(4\|h\|_{\mathcal C^{\lfloor \beta\rfloor}}\lor 2\|h\|_{\cal C^{\lfloor \beta\rfloor+1}})
\end{equation}
ensures that $f_\eta \in \cal P(\beta, d, C)$. Note also that $\|f_\eta-1\|_1=\rho\kappa^{d/2}$. For $\eps\in(0,1)$ we set $\kappa\asymp\eps^{-1/\beta}$ and $\rho\asymp\eps^{(2\beta+d)/(2\beta)}$.  These ensure that \eqref{eqn:density condition} and $\TV(f_\eta, f_0)\gtrsim\eps$ hold, where as usual the constants may depend on $(\beta, d, C)$. Noting that $\|\sqrt{f_\eta}-1\|_2 \asymp \|f_\eta-1\|_1 \gtrsim \eps$, we immediately obtain that $m\gtrsim1/\eps^2$ is necessary for testing, by reduction from binary hypothesis testing \eqref{eqn:LF -> HT}. Observe also that for any $\eta,\eta'$,
\begin{equation}\label{eqn:P_S innerprod}
    \int_{[0,1]^d} f_\eta(x)f_{\eta'}(x) \rm{d}x = 1+\rho^2\langle\eta,\eta'\rangle
\end{equation}
which will be used later.

\textbf{Goodness-of-fit testing.}\label{sec:P_S lower TST}
Let $\eta$ be drawn uniformly at random. We show that $\TV(f_0^{\otimes n}, \bb E f_\eta^{\otimes n})$ can be made arbitrarily small provided $n \lesssim \eps^{-(2\beta+d/2)/\beta}$, which yields a lower bound on $n$ via reduction from goodness-of-fit testing \eqref{eqn:LF -> GoF}. By Lemma \ref{lem:pinsker} we can focus on bounding the $\chi^2$ divergence. Via Ingster's trick we have
\begin{align*}
    \chi^2(\bb E_\eta[f_\eta^{\otimes n}], f_0^{\otimes n}) + 1 &= \int\limits_{[0,1]^d\times\cdots\times[0,1]^d} \left(\bb E_\eta \prod\limits_{i=1}^n f_\eta(x_i)\right)^2 \rm{d}x_1\cdots\rm{d}x_n \\
    &= \bb E_{\eta\eta'} \prod\limits_{i=1}^n \left(\int_{[0,1]^d} f_\eta(x) f_{\eta'}(x)\rm{d}x\right),
\end{align*}
where $\eta,\eta'$ are i.i.d.. By \eqref{eqn:P_S innerprod} and the inequalities $1+x\leq e^x, \cosh(x)\leq\exp(x^2)$ for all $x\in\bb R$, we have
\begin{align*}
    &= \bb E_{\eta\eta'} \left(1+\rho^2\langle\eta,\eta'\rangle\right)^n \\
    &\leq \bb E_{\eta\eta'} \exp(n\rho^2\langle\eta,\eta'\rangle) \\
    &= \cosh(n\rho^2)^{\kappa^d} \\
    &\leq \exp(n^2 \rho^4 \kappa^d).
\end{align*}
Thus, goodness-of-fit testing is impossible unless $n\gtrsim\rho^{-2}\kappa^{-d/2} \asymp 1/\eps^{(2\beta+d/2)/\beta}$. 

\textbf{Likelihood-free hypothesis testing.}
We are now ready to show the lower bound on the product $mn$. Once again $\eta\in\{\pm1\}^{\kappa^d}$ is drawn uniformly at random and we apply Lemma \ref{lem:TV lower} with the choices $P_0 = f_\eta^{\otimes n} \otimes f_0^{\otimes n} \otimes f_\eta^{\otimes m}$ against $P_1 = f_\eta^{\otimes n} \otimes f_0^{\otimes n+m}$. Let $\bb P_{0,XYZ}, \bb P_{1,XYZ}$ denote the joint distribution of the samples $X,Y,Z$ under the measures $\bb E P_0, \bb E P_1$ respectively. By Pinsker's inequality and the chain rule we have
\begin{align*}
    \TV(\bb P_{0,XYZ}, \bb P_{1,XYZ})^2 &= \TV(\bb P_{0,XZ}, \bb P_{1,XZ})^2 \\
    &\leq \KL(\bb P_{0,XZ} \| \bb P_{1,XZ}) \\
    &= \KL(\bb P_{0,Z|X} \| \bb P_{1,Z|X} | \bb P_{0,X}) + \underbrace{\KL(\bb P_{0,X} \| \bb P_{1,X})}_{=0}, 
\end{align*}
where the last line uses that the marginal of $X$ is equal under both measures. Clearly $\bb P_{1,Z|X}$ is simply $\operatorname{Unif}([0,1]^d)^{\otimes m}$ and $\bb P_{0,X}, \bb P_{0,Z|X}$ have densities $\bb E_\eta f_\eta^{\otimes n}$ and $\bb E_{\eta|X} f_\eta^{\otimes m}$ respectively. Given $X$, let $\eta'$ be an independent copy of $\eta$ from the posterior given $X$. By Ingster's trick we have
\begin{align*}
    \KL(\bb P_{0,Z|X}\|\bb P_{1,Z|X}|\bb P_{0,X}) &\leq \chi^2(\bb P_{0,Z|X}\|\bb P_{1,Z|X}|\bb P_{0,X}) \\
    &= -1 + \bb E_X \int_{[0,1]^d\times\cdots\times[0,1]^d}  \bb E_{\eta|X}\bb E_{\eta'|X} \prod\limits_{i=1}^m f_\eta(z_i)f_{\eta'}(z_i) \rm{d}z_1\dots\rm{d}z_m \\
    &= -1 + \bb E_{\eta\eta'} (1+\rho^2\langle\eta,\eta'\rangle)^m,
\end{align*}
where the last line uses \eqref{eqn:P_S innerprod}. Let $N=(N_1,\dots,N_{\kappa^d})$ be the vector of counts indicating the number of $X_i$ that fall into each bin $\{[(j-1)/\kappa,j/\kappa]\}_{j\in[\kappa]^d}$. Clearly $N\stackrel{d}{\sim} \operatorname{Multinomial}(n, (\frac{1}{\kappa^d}, \dots, \frac{1}{\kappa^d}))$. Using that $\eta_j\eta'_j$ depends on only those $X_i$ that fall in bin $j$ and the inequality $1+x\leq \exp(x)$ valid for all $x \in \bb R$, we can write
\begin{align*}
    \chi^2(\bb P_{0,Z|X}\|\bb P_{1,Z|X}|\bb P_{0,X}) + 1 &\leq \bb E_N \bb E_{\eta\eta'|N} \prod\limits_{j \in [\kappa]^d} \exp(\rho^2 m \eta_j \eta'_j) \\
    &= \bb E_N \prod\limits_{j \in [\kappa]^d} \bb E_{\eta_j\eta'_j|N_j}\exp(\rho^2m\eta_j\eta'_j).
\end{align*}
We now focus on a particular bin $j$. Define the bin-conditional densities
\begin{equation}
    p_\pm = \kappa^d(1\pm \rho h_j)\one_{[(j-1)/\kappa, j/\kappa]},
\end{equation}
where we drop the dependence on $j$ in the notation. Let $X^{(j)}\eqdef(X_{i_1}, \dots, X_{i_{N_j}})$ be those $X_i$ that fall in bin $j$. Note that $\{i_1,\dots,i_{N_j}\}$ is a uniformly distributed size $N_j$ subset of $[n]$ and given $N_j$, the density of $X_{i_1}, \dots, X_{i_{N_j}}$ is $\frac12(p_+^{\otimes N_j}+p_-^{\otimes N_j})$. We can calculate
\begin{align*}
    \bb P(\eta_j\eta_j'=1|N_j) &= \bb E_{X^{(j)}|N_j} \bb P(\eta_j\eta'_j = 1 | X^{(j)}) \\
    &= \bb E_{X^{(j)}|N_j}\left[\bb P(\eta_j = 1 | X^{(j)})^2 + \bb P(\eta_j=-1| X^{(j)})^2\right] \\
    &= \bb E_{X^{(j)}|N_j} \left[ \frac{\frac14(p_+^{\otimes N_j})^2 + \frac14(p_-^{\otimes N_j})^2}{\frac14(p_+^{\otimes N_j}+p_-^{\otimes N_j})^2}\right] \\
    &= \frac12 + \frac14\left(\chi^2(p_+^{\otimes N_j}\|\frac12(p_+^{\otimes N_j}+p_-^{\otimes N_j})) + \chi^2(p_-^{\otimes N_j}\|\frac12(p_+^{\otimes N_j}+p_-^{\otimes N_j}))\right).
\end{align*}
By convexity of the $\chi^2$ divergence in its arguments and tensorization, we have
\begin{align*}
    \bb P(\eta_j\eta_j'=1|N_j) &\leq \frac12 + \frac18\left(\chi^2(p_+^{\otimes N_j}\|p_-^{\otimes N_j})+\chi^2(p_-^{\otimes N_j}\|p_+^{\otimes N_j})\right) \\
    &= \frac14 + \sum_{\omega\in\{\pm1\}}  \left(\kappa^d \int_{[(j-1)/\kappa, j/\kappa]} \frac{(1+\omega\rho h_j(x))^2}{1-\omega\rho h_j(x)} \rm{d}x\right)^{N_j}.
\end{align*}
Using that $\rho\|h_j\|_\infty \leq 1/2$ by construction, we have
\begin{align*}
    \int_{[(j-1)/\kappa, j/\kappa]} \frac{(1+\rho h_j(x))^2}{1-\rho h_j(x)} \rm{d}x &= \frac{1}{\kappa^d} + \int_{[(j-1)/\kappa, j/\kappa]} \frac{4\rho^2 h_j^2(x)}{1-\rho h_j(x)} \D x \\
    &\leq \frac{1}{\kappa^d} + 8\rho^2.
\end{align*}
The same bound is obtained for the other integral term. We get
\begin{equation*}
    \chi^2(\bb P_{0,Z|X}\| \bb P_{1,Z|X}|\bb P_{0,X})+1 \leq \bb E_N \prod\limits_{j\in[\kappa]^d} \left(\frac14\left(e^{\rho^2m}-e^{-\rho^2m}\right)(1+(1+8\rho^2\kappa^d)^ {N_j})+e^{-\rho^2m}\right) = (\dagger).
\end{equation*}
The final step is to apply Lemma \ref{lem:multi expectation} to pass the expectation through the product. Assuming that $m \lor n\lesssim \rho^{-2} \asymp \eps^{-(2\beta+d)/\beta}$ for a small enough implied constant, using the inequalities $e^x\leq1+x+x^2, 1-x\leq e^{-x}\leq 1-x+x^2/2$ valid for all $x\in[0,1]$, and  Lemma \ref{lem:multi expectation}, we obtain
\begin{align*}
    (\dagger) &\leq (e^{-\rho^2m}+\frac14\left(e^{\rho^2m}-e^{-\rho^2m}\right)(1+e^{8\rho^2n}))^{\kappa^d} \\
    &\leq (1+c\rho^4mn)^{\kappa ^d} \\
    &\leq \exp(c\rho^4\kappa^d mn)
\end{align*}
for a universal constant $c>0$. Therefore, if $m \lor n \lesssim \eps^{-(2\beta+d)/\beta}$ likelihood-free hypothesis testing is impossible unless $mn\gtrsim \rho^{-4}\kappa^{-d} \asymp 1/\eps^{2(2\beta+d/2)/\beta}$. 

Suppose now that $m \lor n \gtrsim \eps^{-(2\beta+d)/\beta}$ instead. We have two cases:
\begin{enumerate}
    \item If $n\gtrsim\eps^{-(2\beta+d)/\beta}$ then from \Cref{prop:P_S upper} we know that $m \asymp 1/\eps^2$ is enough for achievability. However, by the first part of the proof we know that $m\gtrsim 1/\eps^2$ must always hold, which provides the matching lower bound in this case. 
    \item If $m\gtrsim \eps^{-(2\beta+d)/\beta}$ then we can assume $m\gtrsim n$ also holds, otherwise the first case above would apply. From the goodness-of-fit testing lower bound we know that $n\gtrsim\eps^{-(2\beta+d/2)/\beta}$ must always hold, and from \Cref{prop:P_S upper} we know that $(m,n) \asymp (\eps^{-(2\beta+d/2)/\beta}, \eps^{-(2\beta+d/2)/\beta})$ is achievable, so we get matching bounds in this case too. 
\end{enumerate}

Summarizing, we've shown that for succesful testing $m \gtrsim 1/\eps^{2}, n \gtrsim 1/\eps^{(2\beta+d/2)/\beta}$ and $mn \gtrsim \eps^{-2(2\beta+d/2)/\beta}$ must hold, which concludes our proof. 
\end{proof}

\subsection{The class $\cal P_\sf{G}$}
\begin{proposition}\label{prop:P_G lower}
For any $s,C>0$ there exists a finite constant $c$ independent of $\eps$ such that
\begin{equation*}
    c \{m\geq 1/\eps^2, n\geq \eps^{-(2s+1/2)/s}, mn \geq \eps^{-2(2s+1/2)/s}\} \supseteq \cal R_\sf{LF}(\eps, \cal P_\sf{G}(s, C))
\end{equation*}
for all $\eps\in(0,1)$. 
\end{proposition}
\begin{proof}
\textbf{Adversarial construction.}
Let $\gamma \in \ell^1$ be a non-negative sequence, and let $\theta \sim \otimes_{k=1}^\infty \cal N(0,\gamma_k)$. Define the random measure $\mu_\theta = \otimes_{j=1}^\infty \cal N(\theta_j,1)$. Let $\eps \in (0,1)$ be given. For our proofs we use
\begin{equation}\label{eqn:gamma def}
    \gamma_k =\begin{cases} c_1\eps^{(2s+1)/s} &\text{for } 1\leq k\leq c_2\eps^{-1/s} \\ 0&\text{otherwise}\end{cases}
\end{equation}
for appropriate constants $c_1,c_2$. Recall our definition of the Sobolev ellipsoid $\cal E(s, C)$ with associated sobolev norm $\|\cdot\|_s$. We have
\begin{align*}
    (\bb E\|\theta\|_s)^2 &\leq \bb E \sum\limits_{j=1}^\infty j^{2s} \theta_i^2 = \|\sqrt{\gamma}\|_s^2 = c_1\eps^{(2s+1)/s} \sum\limits_{j=1}^{c_2\eps^{-1/s}} j^{2s} \leq c_1c_2^{2s+1} \\
    \TV(\bb P_\gamma, \bb P_0) &\geq \frac{1\land \|\theta\|_2}{200},
\end{align*}
where last line holds by \cite[Theorem 1.2]{devroye2018total}. 

First, we need to verify that our construction is valid, that is, that $\bb P_\gamma \in \cal P_\sf{G}(s, C)$ and $\TV(\bb P_\gamma, \bb P_0) \geq \eps$ with high probability. For standard Gaussian $Z \sim \cal N(0,1)$ it holds that 
\begin{equation*}
    \E \exp(\lambda(Z^2 - 1)) \leq \exp(2\lambda^2)
\end{equation*}
for all $|\lambda| \leq 1/4$. Therefore, for a sequence of independent standard Gaussians $Z_1,Z_2,\dots$ we get
\begin{align*}
    \E \exp(\lambda\sum_{j=1}^\infty \gamma_j (Z_j^2-1)) &\leq \exp(2 \lambda^2 \|\gamma\|_2^2)
\end{align*}
for all $|\lambda| \leq \min_j (4\gamma_j)^{-1} = c_1^{-1}\eps^{-(2s+1)/s}/4$. Assuming that $c_1\eps^{(2s+1)/s} \leq \|\gamma\|_2$, standard sub-Exponential concentration bounds imply that there exists a universal constant $c_3>0$ such that 
\begin{align*}
    \mathbb P(\|\theta\|_2^2 - \E\|\theta\|_2^2 \leq -t) \leq \exp(-\frac{c_3 t}{\|\gamma\|_2})
\end{align*}
for all $t \geq 0$. Since $\E\|\theta\|_2^2 = \|\gamma\|_1 = c_1c_2\eps^2$, and $\|\gamma\|_2^2 = c_2c_1^2\eps^{\frac{4s+1}{s}}$, we can set $t = \frac12\|\theta\|_2^2$ to get
\begin{align*}
    \mathbb P(\|\theta\|_2^2 \leq \frac12c_1c_2\eps^2) \leq \exp(-\frac12c_3\sqrt{c_2}\eps^{-1/(2s)}). 
\end{align*}
Now choose $c_1$ and $c_2$ to satisfy
\begin{equation}\label{eqn:c_1c_2 choice}
    100c_1c_2^{2s+1} = C \qquad\text{and}\qquad c_1c_2=2. 
\end{equation}
and $\eps$ small enough to satisfy
\begin{equation*}
    c_1\eps^{(2s+1)/s} \leq \|\gamma\|_2 = \sqrt c_1 c_1 \eps^{(2s+1/2)/s} \qquad\text{and}\qquad \frac12c_3\sqrt c_2\eps^{-1/(2s)} \geq \log(100). 
\end{equation*}
Long story short, these conditions ensure that $\mathbb P(\mu_\gamma \in \cal P_\sf{G}(s,C), \TV(\mu_\gamma, \mu_0)\geq\eps) \geq 0.98$ for all $\eps$ small enough in terms of $C$ and $s$, and therefore we can proceed to computation using \Cref{lem:TV lower}. 

Note that we immediately get the binary hypothesis testing lower bound $m\gtrsim1/\eps^2$ via our reduction \eqref{eqn:LF -> GoF}, as $\H(\mu_0, \mu_{\sqrt{\gamma}}) \asymp \TV(\mu_0, \mu_{\sqrt{\gamma}}) = \sqrt 2\eps$ by \Cref{lem:gaussian equivalence} and the choice \eqref{eqn:c_1c_2 choice}.

\textbf{Goodness-of-fit testing.}
We show that $\TV(\mu_0^{\otimes n}, \bb E\mu_\gamma^{\otimes n})$ can be made arbitrarily small as long as $n\lesssim 1/\eps^{(2s+1/2)/s}$, which yields a lower bound on $n$ via reduction from goodness-of-fit testing \eqref{eqn:LF -> GoF}. Let us compute the distribution $\bb E \mu_\gamma^{\otimes n}$. By independence clearly $\bb E \mu_\gamma^{\otimes n} = \otimes_{k=1}^\infty \bb E_{\theta \sim \cal N(0,\gamma_k)} \cal N(\theta,1)^{\otimes n}$. Focusing on the inner term and and dropping the subscript $k$, for the density we have
\begin{align*}
    \bb E_{\theta \sim \cal N(0,\gamma)} \left[\frac{1}{(2\pi)^{n/2}} \exp(-\frac12\sum_{j=1}^n(x_j-\theta)^2)\right] &\propto \exp(-\frac{\|x\|_2^2}{2}) \bb E \exp(-\frac n2(\theta^2-2\theta\bar{x})),
\end{align*}
where we write $\bar{x}\eqdef\frac1n\sum_jx_j$. Looking at just the term involving $\theta$, we have
\begin{align*}
    \bb E \exp(-\frac n2(\theta^2-2\theta\bar{x})) &\propto \int \exp(-\frac12(\theta^2(n+\frac1\gamma)-2\theta n\bar{x}))\rm{d}\theta \propto \exp(\frac12\frac{n^2\bar{x}^2}{n+\frac1\gamma}). 
\end{align*}
Putting everything together, we see that $\bb E \mu_\gamma^{\otimes n} = \otimes_{k=1}^\infty \cal N(0, (\Id_n-\frac{\gamma_k}{1+n\gamma_k}\one_n\one_n^\T)^{-1})$.
Thus, using Lemma \ref{lem:pinsker} we obtain
\begin{align*}
    \TV^2(\mu_0^{\otimes n}, \bb E \mu_\gamma^{\otimes n}) &\leq \sum_{k=1}^\infty \KL(\cal N(0, \Id_n) \| \cal N(0, {(\Id_n-\frac{\gamma_k}{1+n\gamma_k}\one_n\one_n^\T)}^{-1})) \\
    &= \frac12 \sum_{k=1}^\infty \left(-\frac{n\gamma_k}{n\gamma_k+1} + \log(1+n\gamma_k)\right) \\
    &\leq \frac12\sum\limits_{k=1}^\infty \frac{n^2\gamma_k^2}{1+n\gamma_k} \lesssim \sum_{k=1}^\infty n^2\gamma_k^2.
\end{align*}
Taking $\gamma$ as in \eqref{eqn:gamma def} gives
\begin{equation*}
    \TV^2(\mu_0^{\otimes n},\bb E \mu_\gamma^{\otimes n}) \lesssim n^2 \eps^{2(2s+1/2)/s}.
\end{equation*}
Thus, goodness-of-fit testing is impossible unless $n \gtrsim 1/\eps^{(2s+1/2)/s}$ as desired.

\textbf{Likelihood-free hypothesis testing.}
We apply Lemma \ref{lem:TV lower} with measures $P_0 = \mu_\gamma^{\otimes n}\otimes \mu_0^{\otimes n} \otimes \mu_\gamma^{\otimes m}$ and $P_1 = \mu_\gamma^{\otimes n}\otimes \mu_0^{\otimes n} \otimes \mu_0^{\otimes m}$. By an analogous calculation to that in the previous part, we obtain
\begin{align*}
    \bb E P_0 &= \otimes_{k=1}^\infty \cal N\Bigg(0,\Big(\Id_{2n+m}-\frac{1}{n+m+\frac{1}{\gamma_k}} \begin{pmatrix} \one_n\one_n^\T & 0 & \one_n\one_m^\T \\ 0 & 0 & 0 \\ \one_m\one_n^\T & 0 & \one_m\one_m^\T\end{pmatrix}\Big)^{-1}\Bigg) \eqdef \otimes_{k=1}^\infty \cal N(0, \Sigma_{0k})\\
    \bb E P_1 &= \otimes_{k=1}^\infty \cal N\Bigg(0,\Big(\Id_{2n+m}-\frac{1}{n+\frac{1}{\gamma_k}} \begin{pmatrix} \one_n\one_n^\T & 0 & 0 \\ 0 & 0 & 0 \\ 0 & 0 & 0\end{pmatrix}\Big)^{-1}\Bigg) \eqdef \otimes_{k=1}^\infty \cal N(0, \Sigma_{1k}).
\end{align*}
By the Sherman-Morrison formula, we have
\begin{equation*}
    \Sigma_{0k} = \Id_{2n+m}+\gamma_k\begin{pmatrix} \one_n\one_n^\T & 0 & \one_n\one_m^\T \\ 0 & 0 & 0 \\ \one_m\one_n^\T & 0 & \one_m\one_m^\T\end{pmatrix}
\end{equation*}
Therefore, by Pinsker's inequality and the closed form expression for the KL-divergence between centered Gaussians, we obtain
\begin{align*}
    \TV^2(\bb E P_0, \bb E P_1) &\leq \KL(\bb E P_0 \| \bb E P_1) \\
    &= \frac12\sum_{k=1}^\infty\left(\gamma_km-\log\left(1 + \frac{\gamma_km}{\gamma_k(n+m)+1}\right)\right).
\end{align*}
Once again we choose $\gamma$ as in \eqref{eqn:gamma def}. Using the inequality $\log(1+x) \geq x-x^2$ valid for all $x \geq 0$ we obtain
\begin{align*}
    \TV^2(\bb EP_0, \bb EP_1) &\lesssim \eps^{-2(2s+1/2)/s}(m^2+mn).
\end{align*}
Therefore, likelihood-free hypothesis testing is impossible unless $m \gtrsim \eps^{-(2s+1/2)/s}$ or $nm \gtrsim \eps^{-2(2s+1/2)/s}$. Note that we already have the lower bound $n \gtrsim \eps^{-(2s+1/2)/s}$ by reduction from goodness-of-fit testing \eqref{eqn:LF -> GoF}, so that $m\gtrsim \eps^{-(2s+1/2)/s}$ automatically implies $nm \gtrsim \eps^{-2(2s+1/2)/s}$. Combining everything we get the desired bounds. 
\end{proof}

\subsection{The classes $\cal P_\sf{Db}$ and $\cal P_\sf{D}$}
Our first result in this section derives tight minimax lower bounds for the class $\cal P_\sf{Db}$. Since $\cal P_\sf{D} \supset \cal P_\sf{Db}$ these lower bounds immediately carry over to the larger class. However, to get tight lower bounds for all regimes for $\cal P_\sf{D}$, we have to prove additional results in \Cref{prop:P_D lower n < m,prop:P_D lower n > m} below. 
\begin{proposition}\label{prop:P_D lower}
For any $C>1$ there exists a finite constant $c$ independent of $\eps$ and $k$, such that
\begin{equation*}
    c\{m\geq 1/\eps^2, n\geq \sqrt{k}/\eps^2, mn \geq k/\eps^4\} \supseteq \cal R_\sf{LF}(\eps, \cal P_\sf{Db}(k, C)) \supseteq \cal R_\sf{LF}(\eps, \cal P_\sf{D}(k))
\end{equation*}
for all $\eps \in (0,1)$ and $k\geq2$. 
\end{proposition}

\begin{proof}
The second inclusion is trivial. For the first inclusion we proceed analogously to the case of $\cal P_\sf{H}$. \newline
\textbf{Adversarial construction.}
Let $k$ be an integer and $\eps \in (0,1)$. For $\eta \in \{-1,1\}^{k}$ define the distribution $p_\eta$ on $[2k]$ by
\begin{align*}
    p_\eta(2j-1) &= \frac{1}{2k}(1+\eta_j\eps) \\
    p_\eta(2j) &= \frac{1}{2k}(1-\eta_j\eps),
\end{align*}
for $j \in [k]$. Clearly $\H(p_\eta, p_0) \asymp \TV(p_\eta, p_0) = \eps$, where $p_0 = \operatorname{Unif}[2k]$, so that by reduction from binary hypothesis testing \eqref{eqn:LF -> GoF} we get the lower bound $m\gtrsim1/\eps^2$. Observe also that for any $\eta,\eta'\in\{\pm1\}^k$,
\begin{equation}
\begin{aligned}\label{eqn:P_D inner prod}
    \sum\limits_{j\in[2k]} p_\eta(j)p_{\eta'}(j) &= \frac{1}{2k}(1+\frac{\eps^2\langle\eta,\eta'\rangle}{k}).
\end{aligned}
\end{equation}

\textbf{Goodness-of-fit testing.}
Let $\eta$ be uniformly random. We show that $\TV(p_0^{\otimes n}, \bb E p_\eta^{\otimes n})$ can be made arbitrarily small as long as $n\lesssim\sqrt{k}/\eps^2$, which yields the corresponding lower bound on $n$ by reduction from goodness-of-fit testing \eqref{eqn:LF -> GoF}. Once again, by Lemma \ref{lem:pinsker} we focus on the $\chi^2$ divergence. We have
\begin{align*}
    \chi^2(\bb Ep_\eta^{\otimes n} \| p_0^{\otimes n}) + 1 &=(2k)^n\sum\limits_{j\in[2k]^n} \bb E_{\eta\eta'} \prod\limits_{i=1}^n p_\eta(j_i)p_{\eta'}(j_i) \\
    &=\bb E_{\eta\eta'} (1+\frac{\eps^2\langle\eta,\eta'\rangle}{k})^n \\
    &\leq \exp(n^2\eps^4/k)
\end{align*}
where the penultimate line follows from \eqref{eqn:P_D inner prod} and the last line via the same argument as in \ref{sec:P_S lower TST}. Thus, goodness-of-fit testing is impossible unless $n \gtrsim \sqrt{k}/\eps^2$.

\textbf{Likelihood-free hypothesis testing.}
We apply Lemma \ref{lem:TV lower} with the two random measures $P_0 = p_\eta^{\otimes n} \otimes p_0^{\otimes n} \otimes p_\eta^{\otimes m}$ and $P_1 = p_\eta^{\otimes n}\otimes p_0^{\otimes (n+m)}$. Analogously to the case of $\cal P_\sf{H}$, let $\bb P_{0,XYZ}, \bb P_{1,XYZ}$ respectively denote the distribution of the observations $X,Y,Z$ under $\bb E P_0,\bb E P_1$ respectively. As for $\cal P_\sf{H}$, we have
\begin{align*}
    \TV^2(\bb P_{0,XYZ}, \bb P_{1,XYZ}) &\leq \KL(\bb P_{0,XYZ}\|\bb P_{1,XYZ}) \\
    &\leq \KL(\bb P_{0,Z|X} \| \bb P_{1,Z|X} | \bb P_{0,X}).
\end{align*}
For any $X$ the distribution $\bb P_{1,Z|X}$ is uniform, and $\bb P_{0,Z|X}, \bb P_{0,X}$ have pmf $\bb E_{\eta|X} p_\eta^{\otimes m}$ and $\E_\eta p_\eta^{\otimes n}$ respectively. Once again, by Lemma \ref{lem:pinsker} we may turn our attention to the $\chi^2$-divergence. Given $X$, let $\eta'$ have the same distribution as $\eta$ and be independent of it. Then
\begin{align*}
    \chi^2(\bb P_{0,Z|X} \| \bb P_{1,Z|X} | \bb P_{0,X}) + 1 &= (2k)^m \bb E_X \sum\limits_{j\in[2k]^m} \bb E_{\eta|X} \bb E_{\eta'|X} \prod\limits_{i=1}^n p_\eta(j_i)p_{\eta'}(j_i) \\
    &= \bb E_{\eta\eta'} (1+\frac{\eps^2\langle\eta,\eta'\rangle}{k})^m \\
    &\leq \bb E_{\eta\eta'} \prod\limits_{j \in [k]} \exp(\frac{\eps^2 m \eta_j\eta'_j}{k}),
\end{align*}
where we used Lemma \ref{eqn:P_D inner prod}. Let $N=(N_1,\dots,N_k)$ be the vector of counts indicating the number of the $X_1,\dots,X_n$ that fall into the bins $\{2j-1,2j\}$ for $j\in[k]$. Clearly $N\sim\mult(n, (\frac1k, \dots,\frac1k))$. Let us focus on a specific bin $\{2j-1,2j\}$ and define the bin-conditional pmf
\begin{equation*}
    p_\pm(x) = \begin{cases} \frac12(1\pm\eps) &\text{if }x=2j-1, \\
                             \frac12(1\mp\eps) &\text{if }x=2j \\
                             0 &\text{otherwise}, \end{cases}
\end{equation*}
where we drop the dependence on $j$ in the notation. Let $X_{i_1}, \dots, X_{i_{N_j}}$ be the $N_j$ observations falling in $\{2j-1,2j\}$. Given $N_j$, the pmf of $X_{i_1},\dots,X_{i_{N_j}}$ is $\frac12(p_+^{\otimes N_j} + p_-^{\otimes N_j})$. We have $\eta_j\eta'_j \in \{\pm1\}$ almost surely, and analogously to Section \ref{sec:P_S lower} we may compute
\begin{align*}
    \bb P(\eta_j\eta'_j=1|N_j) &= \bb E_{X|N_j} \bb P(\eta_j\eta'_j=1|X) \\
    &= \bb E_{X|N_j}\left[ \bb P(\eta_j=1|X)^2 + \bb P(\eta_j=-1|X)^2\right] \\
    &= \frac12 + \frac14\left(\chi^2(p_+^{\otimes N_j} \| \frac12(p_+^{\otimes N_j} + p_-^{\otimes N_j})) + \chi^2(p_-^{\otimes N_j} \| \frac12(p_+^{\otimes N_j} + p_-^{\otimes N_j})\right) \\
    &\leq \frac12 + \frac18\left(\chi^2(p_-^{\otimes N_j} \| p_+^{\otimes N_j})+\chi^2(p_+^{\otimes N_j} \| p_-^{\otimes N_j})\right).
\end{align*}
We can bound the two $\chi^2$-divergences by
\begin{align*}
    \chi^2(p_\pm^{\otimes N_j} \| p_\mp^{\otimes N_j}) + 1 &= \left(\frac{1 + \frac32\eps^2}{1-\eps^2}\right)^{N_j} \\
    &\leq (1 + 3\eps^2)^{N_j},
\end{align*}
provided $\eps \leq c$ for some universal constant $c>0$. Using Lemma \ref{lem:multi expectation}, we obtain the bound
\begin{align*}
    \bb E_N \prod_{j\in[k]} \bb E_{\eta\eta'|N_j} \exp(\frac{\eps^2 m \eta_j\eta'_j}{k}) &\leq \bb E_N \prod_{j \in [k]} \left(\frac12(\exp(\frac{\eps^2m}{k})-\exp(-\frac{\eps^2m}{k}))(1+(1+2\eps^2)^{N_j}) + \exp(-\frac{\eps^2m}{k})\right) \\
    &\leq \left(\frac12(\exp(\frac{\eps^2m}{k})-\exp(-\frac{\eps^2m}{k}))(1+\exp(\frac{2\eps^2n}{k}))+\exp(-\frac{\eps^2m}{k})\right)^k. \\
    \intertext{Now, under the assumption that $m \lor n \lesssim k/\eps^2$ for some small enough implied constant, the above can be further bounded by}
    &\leq (1+c\frac{\eps^4mn}{k^2})^k \\
    &\leq \exp(\frac{c\eps^4mn}{k}),
\end{align*}
for a universal constant $c>0$. In other words, for $n\lor m \lesssim k/\eps^2$ likelihood-free hypothesis testing is impossible unless $mn \gtrsim k/\eps^4$. 
The treatment of the case $m\lor n \gtrsim k/\eps^2$ is straightforward, and entirely analogous to our discussion at the end of the proof of \Cref{prop:P_S lower}, so we won't repeat it here. This completes the proof. 
\end{proof}

This takes care of the class $\cal P_\sf{Db}$. To prove tight bounds for $\cal P_\sf{D}$ in the large $k$ regime, we have to work harder. Our second lower bound, \Cref{prop:P_D lower n < m} below, proves tight bounds in the regime $n \leq m$ and follows by reduction to two-sample testing \Cref{PROP:REDUCTIONS}. 

\begin{proposition}\label{prop:P_D lower n < m}
    There exists a finite constant $c$ independent of $\eps$ and $k$, 
    \begin{equation*}
        c\{m\geq1/\eps^2,\, n^2m\geq k^2/\eps^4,\, n \leq m\} \supseteq \cal R_\sf{LF}(\eps, \TV, \cal P_\sf{D}) \cap \N^2_{n \leq m}
    \end{equation*}
    for all $k\geq2,\eps\in(0,2)$, where $\N^2_{n\leq m} = \{(n,m) \in \N^2:n \leq m\}$. 
\end{proposition}

\begin{proof}
    Follows from \eqref{eqn:LF TS equiv} and the lower bound construction in \cite{bhattacharya2015testing}. 
\end{proof}

\subsubsection{Valiant's wishful thinking theorem.}
For our third and final lower bound, which is tight in the regime $m \leq n$, we apply a method developed by Valiant, which we describe below.

\begin{definition}
For distributions $p_1,\dots,p_\ell$ on $[k]$ and $(n_1,\dots,n_\ell) \in \mathbb N^\ell$, we define the $(n_1,\dots,n_\ell)$-based moments of $(p_1,\dots,p_\ell)$ as
\begin{equation*}
    m(a_1,\dots,a_\ell) = \sum\limits_{i=1}^k \prod\limits_{j=1}^\ell (n_j p_j(i))^{a_j}
\end{equation*}
for $(a_1,\dots,a_\ell) \in \mathbb N^\ell$.
\end{definition}
Let $p^+=(p^+_1,\dots,p_\ell^+)$ and $p^-=(p^-_1,\dots,p^-_\ell)$ be $\ell$-tuples of distributions on $[k]$ and suppose we observe samples $\{X^{(i)}\}_{i\in[\ell]}$, where the number of observations in $X^{(i)}$ is $\poi(n_i)$. Let $H^\pm$ denote the hypothesis that the samples came from $p^\pm$, up to an arbitrary relabeling of the alphabet $[k]$. It can be shown that to test $H^+$ against $H^-$, we may assume without loss of generality that our test is invariant under relabeling of the support, or in other words, is a function of the \textit{fingerprints}. The fingerprint $f$ of a sample $\{X^{(i)}\}_{i\in[\ell]}$ is the function $f:\bb N^\ell \to \bb N$ which given $(a_1,\dots,a_\ell)\in\bb N^\ell$ counts the number of bins in $[k]$ which have exactly $a_i$ occurences in the sample $X^{(i)}$.  

\begin{theorem}[{{{\cite[Wishful thinking theorem]{valiant2011testing}}}}]\label{thm:valiant}
 Suppose that $|p_i^\pm|_\infty \leq \eta/n_i$ for all $i\in[\ell]$ for some $\eta > 0$, and let $m^+$ and $m^-$ denote the $(n_1,\dots,n_\ell)$-based moments of $p^+,p^-$ respectively. Let $f^\pm$ denote the distribution of the fingerprint under $H^\pm$ respectively. Then
\begin{equation*}
    \TV(f^+,f^-) \leq 2(e^{\eta\ell}-1) + e^{\ell(\eta/2+\log3)}\sum\limits_{a\in\mathbb N^\ell} \frac{|m^+(a)-m^-(a)|}{ \sqrt{1+m^+(a)\lor m^-(a)}}.
\end{equation*}
\end{theorem}
\begin{proof}
The proof is a straightforward adaptation of \cite{valiant2011testing} and thus we omit it. 
\end{proof}

Although Theorem \ref{thm:valiant} assumes a random (Poisson distributed) number of samples, the results carry over to the deterministic case with no modification, due to the sub-exponential concentration of the Poisson distribution. We are ready to prove our likelihood-free hypothesis testing lower bound using Theorem \ref{thm:valiant}.

\begin{proposition}\label{prop:P_D lower n > m}
There exists a finite constant $c$ independent of $\eps$ and $k$, such that 
\begin{equation*}
    c\{m \geq 1/\eps^2, n^2m \geq k^2/\eps^4,\,m\leq n\} \supseteq \cal R_\sf{LF}(\eps, \TV, \cal P_\sf{D}) \cap \N^2_{m \leq n}
\end{equation*}
for all $\eps\in(0,1)$ and $k\geq2$, where $\N^2_{m\leq n} = \{(n,m) \in \N^2:m\leq n\}$. 
\end{proposition}
\begin{proof}
We focus on the regime $n \leq k$, as otherwise the result is subsumed by Proposition \ref{prop:P_D lower}. Suppose that $\eps \in (0,1/2)$, $\eta = 0.01$ (say) and $n/\eta \leq k/2$. Define $\gamma=n/\eta$ and let $p,q$ be pmfs on $[k]$ with weight $(1-\eps)/\gamma$ on $[\gamma]$ and $k/4$ light elements with weight $4\eps/k$ on $[k/2,3k/4]$ and $[3k/4, k]$ respectively. To apply Valiant's wishful thinking theorem, we take $p^+ = (p,q,p)$ and $p^- = (p,q,q)$ with corresponding hypotheses $H^\pm$. The $(n,n,m)$-based moments of $p^\pm$ are given by
\begin{align*}
    \frac{1}{n^{a+b}m^c} m^+(a,b,c) &= \begin{cases} k&\text{if } a+c=0,b=0 \\
    \left(\frac{1-\eps}{\alpha}\right)^{a+b+c}\alpha + \left(\frac{4\eps}{k}\right)^{a+b+c} \frac k4 &\text{if } a+c=0 \text{ xor } b = 0 \\
    \left(\frac{1-\eps}{\alpha}\right)^{a+b+c} \alpha &\text{if } a+c \geq1,b\geq1,
    \end{cases}\\
    \frac{1}{n^{a+b}m^c} m^-(a,b,c) &= \begin{cases} k&\text{if } a=0,b+c=0 \\
    \left(\frac{1-\eps}{\alpha}\right)^{a+b+c}\alpha + \left(\frac{4\eps}{k}\right)^{a+b+c} \frac k4 &\text{if } a=0 \text{ xor } b+c = 0 \\
    \left(\frac{1-\eps}{\alpha}\right)^{a+b+c} \alpha &\text{if } a \geq1,b+c\geq1.
    \end{cases}
\end{align*}
By the wishful thinking theorem we know that 
\begin{align*}
    \TV(f^+,f^-) &\leq 0.061 + 27.41\sum\limits_{a,b,c \in \mathbb N} \frac{|m^+(a,b,c)-m^-(a,b,c)|}{\sqrt{1+\max(m^+,m^-)}}.
\end{align*}
Let us consider the possible values of $|m^+(a,b,c)-m^-(a,b,c)|$. It is certainly zero if
$a\land b \geq 1$ or $a=b=c=0$. Suppose that $a=0$ so that necessarily $b+c\geq1$. Then
\begin{equation*}
    \frac{1}{n^bm^c}|m^+(0,b,c)-m^-(0,b,c)| = \left(\frac{4\eps}{k}\right)^{b+c} \frac k4 \one(b\land c \geq 1).
\end{equation*}
Using the symmetry between $a$ and $b$ and that $1+m^+\lor m^- \geq n^bm^c((1-\eps)/\gamma)^{b+c}\gamma$ (for $m^+\neq m^-$), we can bound the infinite sum above as
\begin{align*}
    &\lesssim \sum\limits_{b,c \geq 1} \frac{n^bm^c k^{1-(b+c)} \eps^{b+c}}{\sqrt{n^bm^c \gamma^{1-(b+c)}(1-\eps)^{b+c}}} \\
    &\lesssim \sum\limits_{b,c \geq 1} n^{b/2} m^{c/2} \left(\frac{\sqrt\gamma}{k}\right)^{b+c-1} \eps^{b+c}
\end{align*}
Plugging in $\gamma = n/\eta \asymp n$, and using $m \leq n \leq k$, we obtain
\begin{align*}
    \TV(f^+, f^-)-0.061 &\lesssim \sum\limits_{b,c \geq 1} n^{b+\frac c2 - \frac12} m^{c/2} \frac{1}{k^{b+c-1}} \eps^{b+c} \\
    &= \frac{n\sqrt m\eps^2}{k} \sum\limits_{b,c \geq 0} \left(\frac nk\right)^{b+\frac c2} \left(\frac mk\right)^{\frac c2} \eps^{b+c} \\
    &\leq \frac{n\sqrt m\eps^2}{k} \sum\limits_{b,c \geq 0} \eps^{b+c} \\
    &\lesssim \frac{n\sqrt m\eps^2}{k},
\end{align*}
where we use that $\eps < 1/2$. Thus, likelihood-free hypothesis testing is impossible for $m \leq n$ unless $n^2 m \gtrsim k^2/\eps^4$.
\end{proof}

\section{Proof of Theorem \ref{THM:HELLINGER GOF}}\label{sec:thm3 proof}
\subsection{Upper bound}
We deduce the upper bound by applying the corresponding result for $\cal P_\sf{D}$ as a black-box procedure. 
\begin{theorem}[{{{\cite{daskalakis2018distribution}}}}]\label{thm:GoF P_D H}
For a constant independent of $\eps$ and $k$,
\begin{equation*}
    n_\sf{GoF}(\eps, \H, \cal P_\sf{D}) \asymp \sqrt{k}/\eps^2.
\end{equation*}
\end{theorem}
Write $\cal G_\ell$ for the regular grid of size $\ell^d$ on $[0,1]^d$ and let $P_\ell$ denote the $L^2$-projector onto the space of functions piecewise constant on the cells of $\cal G_\ell$. For convenience let us re-state \Cref{prop:H separation}. 

\begin{proposition}
For any $\beta \in (0,1]$, $C>1$ and $d\geq1$ there exists a constant $c>0$ such that 
\begin{equation*}
    c\H(f, g) \leq \H(P_\kappa f, P_\kappa g) \leq \H(f,g)
\end{equation*}
holds for any $f,g \in \cal P_\sf{H}(\beta, d, C)$, provided we set $\kappa = (c\eps)^{-2/\beta}$.
\end{proposition}

With the above approximation result, the proof of Theorem \ref{THM:HELLINGER GOF} is straightforward.
\begin{proof}[Proof of Theorem \ref{THM:HELLINGER GOF}]
Suppose we are testing goodness-of-fit to $f_0\in\cal P_\sf{H}$ based on an i.i.d. sample $X_1,\dots,X_n$ from $f\in\cal P_\sf{H}$. Take $\kappa\asymp\eps^{-2/\beta}$ and bin the observations on $\cal G_\kappa$, denoting the pmf of the resulting distribution as $p_f$. Then, under the alternative hypothesis that $\H(f,f_0) \geq \eps$, by Proposition \ref{prop:H separation}
\begin{equation*}
    \eps \lesssim \H(P_\kappa {f_0}, P_\kappa f) = \H(p_{f_0}, p_f).
\end{equation*}
In particular, applying the algorithm achieving the upper bound in Theorem \ref{thm:GoF P_D H} to the binned observations, we see that $n\gtrsim \sqrt{\kappa^d}/\eps^2 = \eps^{-(2\beta+d)/\beta}$ samples suffice.
\end{proof}

\subsection{Lower bound}
The proof is extremely similar to the $\TV$ case, except we put the perturbations at density level $\eps^2$ instead of $1$.
\begin{proof}
Let $\phi:[0,1]\to[0,1]$ be a smooth function such that $\phi(x) = 0$ for $x \leq 1/3$ and $\phi(x) = 1$ for $x \geq 2/3$. Let $h:\R^d \to \bb R$ be smooth, supported in $[0,1]^d$, and satisfy $\int_{[0,1]^d} h(x)\D x = 0$ and $\int_{[0,1]^d} h(x)^2 \D x=1$. Given $\eps \in (0,1)$ let
\begin{equation*}
    f_0(x) = \eps^2 + \frac{\phi(x_1)}{\|\phi\|_1}(1-\eps^2),
\end{equation*}
which is a density on $[0,1]^d$. For a large integer $\kappa$ and $j \in [\kappa/3]\times[\kappa]^{d-1}$ let
\begin{equation*}
    h_j(x) = \kappa^{d/2} h(\kappa x-j+1)
\end{equation*}
for $x \in [0,1]^d$. Then $h_j$ is supported on $[(j-1)/\kappa,j/\kappa]\subseteq [0,1/3]\times[0,1]^{d-1}$ and $\int h^2_j=1$. For $\eta \in \{\pm1\}^{[\kappa/3]\times[\kappa]^{d-1}}$ and $\rho > 0$ let
\begin{equation*}
    f_\eta(x) = f_0 + \rho \sum\limits_{j \in [\kappa/3]\times[\kappa]^{d-1}} \eta_j h_j(x).
\end{equation*}
Then $f_\eta$ is positive provided that $\eps^2 \geq \rho \kappa^{d/2}|h|_\infty \asymp \rho \kappa^{d/2}$. Further, $\|f_\eta\|_{\cal C^\beta}$ is of constant order provided $\rho\kappa^{d/2+\beta}\lesssim1$. Under these assumptions $f_\eta \in \cal P_\sf{H}$. Note that the Hellinger distance between $f_\eta$ and $f_0$ is
\begin{align*}
    H^2(f_0,f_\eta) &= \sum\limits_{j\in[\kappa/3]\times[\kappa]^{d-1}} \int_{[\frac{j-1}{\kappa},\frac j\kappa]} \left(\sqrt{f_0(x)}-\sqrt{f_\eta(x)}\right)^2 \rm{d} x \\
    &= \sum\limits_{j \in [\kappa/3]\times[\kappa]^{d-1}} \int_{[\frac{j-1}{\kappa},\frac j\kappa]} \frac{\rho^2 h_j^2(x)}{(\sqrt{f_0(x)}+\sqrt{f_\eta(x)})^2} \rm{d}x \\
    &\geq \sum\limits_{j \in [\kappa/3]\times[\kappa]^{d-1}} \int_{[\frac{j-1}{\kappa},\frac j\kappa]} \frac{\rho^2 h_j^2(x)}{4\eps^2} \rm{d}x \\
    &\gtrsim \frac{\rho^2 \kappa^d}{\eps^2}.
\end{align*}
Suppose we draw $\eta$ uniformly at random. Via Ingster's trick we compute
\begin{align*}
    \chi^2(\bb E_\eta f_\eta^{\otimes n} \| f_0^{\otimes n}) + 1 &= \int \bb E_{\eta\eta'} \prod\limits_{i=1}^n \frac{f_\eta(x_i)f_{\eta'}(x_i)}{f_0(x_i)} \rm{d} x_1 \dots \D x_n \\
    &= \bb E_{\eta\eta'} \left(\int \frac{f_\eta(x)f_{\eta'}(x)}{f_0(x)}\rm{d}x\right)^n.
\end{align*}
Looking at the integral term on the inside we get
\begin{align*}
    \int \frac{f_\eta(x)f_{\eta '}(x)}{f_0(x)} \rm{d}x &= \int \frac{\left(f_0(x) + \rho \sum\limits_{j \in [\kappa/3]\times[\kappa]^{d-1}} \eta_j h_j(x)\right)\left(f_0(x) + \rho \sum\limits_{j \in [\kappa/3]\times[\kappa]^{d-1}} \eta'_j h_j(x)\right)}{f_0(x)} \rm{d}x \\
    &= 1 + \rho \sum\limits_{j} (\eta_j + \eta'_j) \int h_j(x) \rm{d}x + \rho^2 \sum\limits_j \eta_j\eta'_j \int \frac{h_j(x)^2}{f_0(x)} \rm{d}x \\
    &= 1 + \frac{\rho^2}{\eps^2} \sum_j \eta_j\eta'_j \int h_j(x)^2 \rm{d}x \\
    &= 1 + \frac{\rho^2}{\eps^2} \langle \eta,\eta'\rangle,
\end{align*}
where we've used that $h_j$ and $h_{j'}$ have disjoint support unless $j=j'$, $\int h_j=0$, $\int h_j^2=1$, and that $f_0(x)=\eps^2$ for all $x$ with $x_1\leq1/3$. Plugging in, using the inequalities $1+x\leq\exp(x)$ and $\cosh(x) \leq \exp(x^2)$ we obtain
\begin{align*}
    \chi^2(\bb E_\eta f_\eta^{\otimes n} \| f_0^{\otimes n}) + 1 &\leq \bb E_{\eta\eta'} (1+\frac{\rho^2}{\eps^2}\langle\eta,\eta'\rangle)^n \\
    &\leq \bb E_{\eta\eta'} \exp(\frac{\rho^2 n}{\eps^2}\langle\eta,\eta'\rangle) \\
    &= \cosh(\frac{\rho^2 n}{\eps^2})^{\kappa^d/3} \\
    &\leq \exp(\frac{\rho^4n^2\kappa^d}{3\eps^4}).
\end{align*}
Choosing $\kappa=\eps^{-2/\beta}$ and $\rho = \eps^{(2\beta+d)/\beta}$ we see that goodness-of-fit testing of $f_0$ is impossible unless
\begin{equation*}
    n \gtrsim \frac{\eps^2}{\rho^2 \kappa^{d/2}} = \eps^{-\frac{2\beta+d}{\beta}}.
\end{equation*}
\end{proof}

\section{Auxiliary technical results}\label{sec:auxilary}

\subsection{Proof of Lemma \ref{lem:n_HT}}\label{sec:proof of lem1}
\begin{proof}
We prove the upper bound first. Let $\bb P_0,\bb P_1 \in \cal P$ be arbitrary. Then by Lemma \ref{lem:pinsker}, 
\begin{align*}
    \inf\limits_{\psi} \max\limits_{i=0,1} \bb P_i^{\otimes m}(\psi \neq i) &\leq \inf\limits_\psi \left(\bb P_0^{\otimes m}(\psi=1)+\bb P_1^{\otimes m}(\psi=0)\right) \\
    &= 1 - \TV(\bb P_0^{\otimes m}, \bb P_1^{\otimes m}) \\
    &\leq 1 - \frac12\H^2(\bb P_0^{\otimes m}, \bb P_1^{\otimes m}) \eqdef (\dagger).
\end{align*}
By tensorization of the Hellinger affinity, we have
\begin{align}\label{eqn:H tensor}
    \H^2(\bb P_0^{\otimes m}, \bb P_1^{\otimes m}) &= 2 - 2\left(1-\frac12 \H^2(\bb P_0, \bb P_1)\right)^m.
\end{align}
Plugging in, along with $1+x\leq e^x$ gives
\begin{align*}
    (\dagger) \leq \exp(-\frac m2 \H^2\left(\bb P_0^{\otimes m}, \bb P_1^{\otimes m})\right).
\end{align*}
Taking $m > 2\log(3)/\H^2(\bb P_0, \bb P_1)$ shows the existence of a successful test. Let us turn to the lower bound. Using Lemma \ref{lem:pinsker} we have
\begin{align*}
     \inf\limits_{\psi} \max\limits_{i=0,1} \bb P_i^{\otimes m}(\psi \neq i) &\geq \frac12\left(1-\TV(\bb P_0^{\otimes m}, \bb P_1^{\otimes m})\right) \\
     &\geq \frac12\left(1-\H(\bb P_0^{\otimes m}, \bb P_1^{\otimes m})\right). 
\end{align*}
Note that it is enough to restrict the maximization in Lemma \ref{lem:n_HT} to $\bb P_0,\bb P_1 \in \cal P$ with $\H^2(\bb P_0, \bb P_1) < 1$. Now, by \eqref{eqn:H tensor} and the inequalities $e^{-2x} \leq 1-x$ valid for all $x\in[0,1/2]$ and $1-x \leq e^{-x}$ valid for all $x \in \R$, we obtain
\begin{align*}
\H^2(\bb P_0^{\otimes m}, \bb P_1^{\otimes m}) &= 2-2\left(1-\frac12\H^2(\bb P_0,\bb P_1)\right)^m \\
&\leq 2 - 2\exp(-m\H^2(\bb P_0, \bb P_1)) \\
&\leq 2m\H^2(\bb P_0, \bb P_1). 
\end{align*}
Taking $m=1/(18\H^2(\bb P_0, \bb P_1))$ concludes the proof via Lemma \ref{lem:TV lower}.
\end{proof}

\subsection{Proof of Lemma \ref{lem:gaussian equivalence}}
\begin{proof}
By standard inequalities between divergences (see e.g. \AoScite{\cite[Chapter 7]{polyanskiywu}}{\Cref{lem:pinsker}}), omitting the argument $(\mu_\theta, \mu_0)$ for simplicity we have
\begin{equation*}
    \TV \leq \H \leq \sqrt{\KL} \leq \sqrt{\chi^2} = \sqrt{\exp(\|\theta\|_2^2)-1} \lesssim \|\theta\|_2.  
\end{equation*}
For the lower bound we obtain $\TV(\mu_\theta, \mu_0) \geq \min\{1, \|\theta\|_2 / 200\} \gtrsim \|\theta\|_2$ by \cite[Theorem 1.2]{devroye2018total}. 
\end{proof}

\subsection{Proof of Proposition \ref{prop:H separation}}
Let us write $a_+ \eqdef a \lor 0$ for both functions and real numbers. We start with some known results of approximation theory.
\begin{definition}
For $f:[0,1]^d \to \bb R$ define the modulus of continuity as
\begin{equation*}
    \omega(\delta; f) = \sup\limits_{\|x-y\|_2\leq\delta} |f(x)-f(y)|.
\end{equation*}
\end{definition}
\begin{lemma}\label{lem:root modulus}
For any real-valued function $f$ and $\delta \geq 0$,
\begin{equation*}
    \omega(\delta; \sqrt{f_+}) \leq \omega(\delta; f)^{1/2}.
\end{equation*}
\end{lemma}
\begin{proof}
Follows from the inequality $|\sqrt{a_+}-\sqrt{b_+}|^2 \leq |a-b|$ valid for all $a,b \in \bb R$.
\end{proof}

\begin{lemma}\label{lem:smooth modulus}
Let $f:[0,1]^d \to \bb R$ be $\beta$-smooth for $\beta \in (0,1]$. Then
\begin{align*}
    \omega(\delta; f) \leq c\,\delta^\beta
\end{align*}
for a constant $c$ depending only on $\|f\|_{\cal C^\beta}$.
\end{lemma}
\begin{proof}
Follows by the definition of H\"older continuity.
\end{proof}

\begin{lemma}[{{{\cite[Theorem 4]{newman1964jackson}}}}]\label{lem:polynomial approximation}
For any continuous function $f:[0,1]^d\to\bb R$ the best polynomial approximation $p_n$ of degree $n$ satisfies
\begin{equation*}
    \|p_n-f\|_\infty \leq c\,\omega\left(\frac{d^{3/2}}{n}; f\right)
\end{equation*}
for a universal constant $c>0$.
\end{lemma}

\begin{definition}
Given a function $f:[0,1]^d \to \bb R$, $\ell \geq 1$ and $j \in [\ell]^d$, let $\pi_{j,\ell} f :[0,1]^d \to \bb R$ denote the function
\begin{equation*}
\pi_{j,\ell} f(x) \eqdef f\left(\frac{x+j-1}{\ell}\right).
\end{equation*}
In other words, $\pi_{j,\ell} f$ is equal to $f$ zoomed in on the $j$'th bin of the regular grid $\cal G_\ell$.
\end{definition}
Recall that here $P_\ell$ denotes the $L^2$ projector onto the space of functions piecewise constant on the bins of $\cal G_\ell$. We are ready for the proof of Proposition \ref{prop:H separation}.

\begin{proof}
Let $\kappa \geq r \geq 1$ whose values we specify later. We treat the parameters $\beta,d,\|f\|_{\cal C^\beta},\|g\|_{\cal C^\beta}$ as constants in our analysis. Let $u_f:[0,1]^d \to \bb R$ denote the (piecewise polynomial) function that is equal to the best polynomial approximation of $\sqrt{f}$ on each bin of $\cal G_{\kappa/r}$ with maximum degree $\alpha$. By \cref{lem:root modulus,lem:smooth modulus} for any $\ell \geq 1$ and $j \in [\ell]^d$
\begin{equation}
    \omega(\delta; \pi_{j,\ell}\sqrt{f}) \leq \omega(\delta/\ell; \sqrt{f}) \lesssim (\delta/\ell)^{\beta/2},
\end{equation}
so that by Lemma  \ref{lem:polynomial approximation}
\begin{align*}
    |u_f-\sqrt{f}|_\infty &= \sup\limits_{j \in [\kappa/r]^d} |\pi_{j,\kappa/r}(u_f - \sqrt{f})|_\infty \\
    &\lesssim \sup\limits_{j \in [\kappa/r]^d} \omega(d^{3/2}/\alpha; \pi_{j,\kappa/r}\sqrt{f}) \\
    &\lesssim (\alpha\kappa/r)^{-\beta/2}.
\end{align*}
Regarding $r$ as a constant independent of $\kappa$, $\alpha$ can be chosen large enough independently of $\kappa$ such that $|u_f-\sqrt{f}|_\infty \leq c_1\kappa^{-\beta/2}$ for $c_1$ arbitrarily small. Define $u_g$ analogously to $u_f$. We have the inequalities
\begin{align*}
    \H(f,g) &= \|\sqrt{f}-\sqrt{g}\|_2 \\
    &\leq \|\sqrt{f}-u_f\|_2 + \|u_f-u_g\|_2 + \|u_g-\sqrt{g}\|_2 \\
    &\leq 2c_1\kappa^{-\beta/2} + \|u_f-u_g\|_2.
\end{align*}
We can write
\begin{align*}
    \|u_f-u_g\|_2^2 &= \frac{1}{(\kappa/r)^d} \sum\limits_{j \in [\kappa/r]^d} \|\pi_{j,\kappa/r}(u_f-u_g)\|_2^2
\end{align*}
Now, by \cite[Lemma 7.4]{arias2018remember} we can take $r$ large enough (depending only on $\beta,d,\|f\|_{\cal C^\beta}, \|g\|_{\cal{C}^\beta}$)  such that
\begin{equation*}
    \|\pi_{j,\kappa/r}(u_f-u_g)\|_2 \leq c_2\|P_r \pi_{j,\kappa/r}(u_f-u_g)\|_2
\end{equation*}
where the implied constant depends on the same parameters as $r$. Thus, we get
\begin{align*}
    \H^2(f,g) &\leq 8c_1^2\kappa^{-\beta} + \frac{2c_2^2}{(\kappa/r)^d} \sum\limits_{j \in [\kappa/r]^d} \|P_r \pi_{j,\kappa/r}(u_f-u_g)\|_2^2 \\
    &\leq 8c_1^2\kappa^{-\beta} + \frac{6c_2^2}{(\kappa/r)^d} \sum\limits_{j\in[\kappa/r]^d}\left( \|P_r\pi_{j,\kappa/r}u_f - \sqrt{P_r\pi_{j,\kappa/r} f}\|_2^2 + \|P_r\pi_{j,\kappa/r}u_g - \sqrt{P_r\pi_{j,\kappa/r} g}\|_2^2\right) \\
    &\qquad + 6c_2^2\H^2(P_\kappa f, P_\kappa f),
\end{align*}
where $c_1, c_2$ depend only on the unimportant parameters, and $c_1$ can be taken arbitrarily small compared to $c_2$. We also used the fact that $P_r \pi_{j,\kappa/r} = \pi_{j,\kappa/r} P_\kappa$. Looking at the terms separately, we have
\begin{align*}
    \|P_r \pi_{j,\kappa/r}u_f-\sqrt{P_r \pi_{j,\kappa/r}f}\|_2 &\leq \|P_r \pi_{j,\kappa/r}u_f-P_r\sqrt{\pi_{j,\kappa/r}f}\|_2 + \|P_r\sqrt{\pi_{j,\kappa/r}f} - \sqrt{P_r \pi_{j,\kappa/r}f}\|_2 \\
    &\leq c\kappa^{-\beta/2} + \|P_r\sqrt{\pi_{j,\kappa/r}f} - \sqrt{P_r \pi_{j,\kappa/r}f}\|_2,
\end{align*}
since $P_r$ is a contraction by Lemma \ref{lem:properties of P_r}. We can decompose the second term as
\begin{align*}
    &\|P_r\sqrt{\pi_{j,\kappa/r}f} - \sqrt{P_r \pi_{j,\kappa/r}f}\|_2^2 = \\
    &\qquad = \sum\limits_{\ell \in [r]^d} \int_{\left[\frac{\ell-1}{r}, \frac{\ell}{r}\right]} \left(r^d\int_{\left[\frac{\ell-1}{r},\frac\ell r\right]} \sqrt{\pi_{j,\kappa/r} f(x)}\rm{d} x - \sqrt{r^d\int_{\left[\frac{\ell-1}{r},\frac\ell r\right]} \pi_{j,\kappa/r} f(x) \rm{d} x}\right)^2 = (\dagger).
\end{align*}
For $x \in [(\ell-1)/r,\ell/r]$ we always have
\begin{equation*}
    |\pi_{j,\kappa/r} f(x) - \pi_{j,\kappa/r} f(\ell/r)| \leq \omega(\frac{\sqrt{d}}{r}; \pi_{j,\kappa/r} f) \lesssim \left(\frac{\sqrt{d}/r}{\kappa/r}\right)^\beta \lesssim \kappa^{-\beta}.
\end{equation*}
Using the inequality $\sqrt{a+b}-\sqrt{(a-b)_+} \leq 2\sqrt{b}$ valid for all $a,b \geq 0$, we can bound $(\dagger)$ by $\kappa^{-\beta}$ up to constant and the result follows.
\end{proof}

\subsection{Proof of Proposition \ref{prop:meta lfht}}

For $f \in L^2(\mu)$ write $f_i = \langle f\phi_i\rangle$ and $f_{ii'} = \langle f\phi_i\phi_{i'}\rangle$,
assuming that the quantities involved are well-defined. We record some useful identities related to $P_r$ that will be instrumental in our proof of \Cref{prop:meta lfht}.
\begin{lemma}\label{lem:properties of P_r}
$P_r$ is self-adjoint and has operator norm 
\begin{equation*}
    \|P_r\| \eqdef \sup_{f\in L^2(\mu):\|f\|_2\leq1} \|P_r(f)\|_2 \leq 1.
\end{equation*}
Suppose that $f,g,h,t \in L^2(\mu)$ and that each quantity below is finite. Then
\begin{align*}
    \sum_{ii'} f_ig_{i'}h_{ii'} &= \langle h P_r(f)P_r(g)\rangle,  \\
    \sum_{ii'} f_ig_ih_{i'}t_{i'} &= \langle f P_r(g)\rangle \langle h P_r(t)\rangle \\
    \sum_{ii'} f_{ii'}g_{ii'} &= \sum_i \langle f\phi_iP_r(g\phi_i)\rangle, 
\end{align*}
where the summation is over $i,i' \in [r]$. 
\end{lemma}
\begin{proof}
Let $P^\perp_r$ denote the orthogonal projection onto the orthogonal complement of $\operatorname{span}(\{\phi_1, \dots, \phi_r\})$. Then for any $f,g \in L^2(\mu)$ we have
\begin{align*}
    \langle f P_r(g)\rangle &= \langle (P_r(f) + P_r^\perp (f)) P_r (g) \rangle = \langle P_r(f) P_r (g)\rangle = \langle P_r(f) g\rangle, 
\end{align*}
where the last equality is by symmetry. We also have
\begin{equation*}
    \|P_r(f)\|^2_2 \leq \|P_r(f)\|^2_2 + \|P_r^\perp (f)\|^2_2 = \|P_r(f)+P_r^\perp(f)\|^2=\|f\|^2_2. 
\end{equation*}
Let $f,g,h,t \in L^2(\mu)$. Then 
\begin{align*}
    \sum\limits_{ii'} f_ig_{i'}h_{ii'} &= \sum\limits_i f_i\sum\limits_{i'} g_{i'}h_{ii'} = \sum_if_i\sum_{i'} \langle g P_r(h\phi_i)\rangle = \sum_if_i \langle P_r (g) h \phi_i\rangle  = \langle P_r(f)h P_r(g)\rangle \\
    \sum_{ii'} f_ig_ih_{i'}t_{i'} &= (\sum_i f_ig_i)(\sum_{i'} h_{i'}t_{i'}) = \langle fP_r(g)\rangle\langle h P_r(t)\rangle \\
    \sum_{ii'} f_{ii'} g_{ii'} &= \sum_i \langle f\phi_i  \sum_{i'}\langle g\phi_i \phi_{i'}\rangle \phi_{i'}\rangle = \sum_i\langle f\phi_i P_r(g\phi_i)\rangle.
\end{align*}
\end{proof}

\begin{proof}[Proof of Proposition \ref{prop:meta lfht}]
Let us label the different terms of the statistic $T^{-\sf{d}}_\sf{LF}$:
\begin{align*}
    T_\sf{LF}^{-\sf{d}} &= \sum\limits_{i=1}^r \Bigg\{\frac{2}{n^2} \sum\limits_{j < j'}^n \phi_i(X_j)\phi_i(X_{j'})-\frac{2}{n^2}\sum\limits_{j < j'}^n \phi_i(Y_j)\phi_i(Y_{j'})\\
    &\qquad -\frac{2}{nm}\sum\limits_{j=1}^n\sum\limits_{u=1}^m\phi_i(X_j)\phi_i(Z_u) + \frac{2}{nm}\sum\limits_{j=1}^n\sum\limits_{u=1}^m\phi_i(Y_j)\phi_i(Z_u)\Bigg\} \\
    &= \frac{2}{n^2}\sf{I} - \frac{2}{n^2}\sf{II} -\frac{2}{nm}\sf{III} + \frac{2}{nm}\sf{IV}.
\end{align*}
Recall that $X,Y,Z \sim f^{\otimes n},g^{\otimes n}, h^{\otimes m}$ respectively. A straightforward computation yields
\begin{equation*}
    \mathbb E T_\sf{LF} = \|P_r(f-h)\|_2^2-\|P_r(g-h)\|_2^2 - \frac1n\big(\|P_r(f)\|_2^2-\|P_r(g)\|_2^2\big).
\end{equation*}
We decompose the variance as
    \begin{align*}
        \var(T_\sf{LF}) &= \frac{4}{n^4}\var(\sf{I}) + \frac{4}{n^4}\var(\sf{II}) + \frac{4}{n^2m^2}\var(\sf{III}) + \frac{4}{n^2m^2} \var(\sf{IV}) \\
        &\quad - \frac{8}{n^3m} \cov(\sf{I},\sf{III}) - \frac{8}{n^3m} \cov(\sf{II},\sf{IV}) - \frac{8}{n^2m^2} \cov(\sf{III},\sf{IV}),
\end{align*}
where we used independence of the pairs $(\sf{I},\sf{II}), (\sf{I},\sf{IV}), (\sf{II},\sf{III})$. Expanding the variances we obtain
\begin{align*}
    \var(\sf{I}) &= \sum\limits_{ii'} \left({n \choose 2}(f^2_{ii'}-f^2_if^2_{i'}) + ({n\choose2}^2-{n\choose2}-{4\choose2}{n\choose4})(f_if_{i'}f_{ii'}-f^2_if^2_{i'})\right) \\
    \var(\sf{II}) &= \sum\limits_{ii'} \left({n \choose 2}(g^2_{ii'}-g^2_ig^2_{i'}) + ({n\choose2}^2-{n\choose2}-{4\choose2}{n\choose4})(g_ig_{i'}g_{ii'}-g^2_ig^2_{i'})\right) \\
    \var(\sf{III}) &= \sum\limits_{ii'} \Big(nm(f_{ii'}h_{ii'}-f_if_{i'}h_ih_{i'})+nm(m-1)(f_{ii'}h_ih_{i'}-f_if_{i'}h_ih_{i'}) + \\
    &\qquad \qquad +mn(n-1)(f_if_{i'}h_{ii'}-f_if_{i'}h_ih_{i'})\Big) \\
    \var(\sf{IV}) &= \sum_{ii'}\Big( nm (h_{ii'}g_{ii'}-h_ih_{i'}g_ig_{i'}) + mn(n-1)(h_{ii'}g_ig_{i'}-h_ih_{i'}g_ig_{i'}) \\
    &\qquad\qquad  + nm(m-1)(g_{ii'}h_ih_{i'}-h_ih_{i'}g_ig_{i'})\Big). 
\end{align*}
For the covariance terms we obtain
\begin{align*}
    \cov(\sf{I},\sf{III}) &= \sum\limits_{ii'} 2m{n\choose2}(f_{ii'}f_ih_{i'}-f_i^2f_{i'}h_{i'}) \\
    \cov(\sf{II},\sf{IV}) &= \sum\limits_{ii'} 2m{n\choose2}(g_{ii'}g_ih_{i'}-g_i^2g_{i'}h_{i'}) \\
    \cov(\sf{III},\sf{IV}) &= \sum\limits_{ii'} mn^2(h_{ii'}f_ig_{i'}-f_ig_{i'}h_ih_{i'}).
\end{align*}
We can now start collecting the terms, applying the calculation rules from Lemma \ref{lem:properties of P_r} repeatedly. Note that ${n\choose2}^2-{n\choose2}-{4\choose2}{n\choose4}=n^3-3n^2+2n$, and
by inspection we can conclude that $1/n,1/m,1/nm, 1/n^2$ and $1/n^3$ are the only terms with nonzero coefficients.
We look at each of them one-by-one:
\begin{align*}
    \coef\left(\frac1n\right) &= \sum_{ii'}^r \Big( \underbrace{4(f_if_{i'}f_{ii'}-f_i^2f_{i'}^2)}_{\var(I)} + \underbrace{4(g_ig_{i'}g_{ii'}-g_i^2g_{i'}^2)}_{\var(II)} +
    \underbrace{4(h_ih_{i'}f_{ii'}-f_if_{i'}h_ih_{i'})}_{\var(III)} + \\
     &\qquad\underbrace{4(g_{ii'}h_ih_{i'}-h_ih_{i'}g_ig_{i'})}_{\var(IV)}
    -\underbrace{8(f_{ii'}f_ih_{i'}-f_i^2f_{i'}h_{i'})}_{\cov(I,III)} - \underbrace{8(g_{ii'}g_ih_{i'}-g_i^2g_{i'}h_{i'})}_{\cov(II,IV)}\Big) \\
    &= 4\langle fP_r(f)^2\rangle - 4\langle f P_r(f)\rangle^2 + 4\langle gP_r(g)^2\rangle - 4\langle gP_r(g)\rangle^2 + 4\langle fP_r(h)^2\rangle - 4\langle fP_r(h)\rangle^2 \\
    &\qquad + 4\langle gP_r(h)^2 - 4\langle hP_r(g)\rangle^2 - 8\langle fP_r(f)P_r(h)\rangle + 8\langle fP_r(f)\rangle\langle fP_r(h)\rangle \\
    &\qquad - 8\langle gP_r(g)P_r(h)\rangle + 8\langle gP_r(g)\rangle\langle gP_r(h)\rangle \\
    &= 4\langle f (P_r(f-h))^2\rangle + 4\langle g (P_r(g-h))^2\rangle - 4\langle P_r(f-h)\rangle^2 - 4\langle P_r(g-h)\rangle^2 \\
    &\leq 4A_{ffh} + 4A_{ggh}, 
\end{align*}
recalling the definition $A_{uvt}=\langle u\big[P_r(v-t)\big]^2\rangle$ for $u,v,t \in L^2(\mu)$. Similarly, we get
\begin{align*}
    \coef\left(\frac1m\right) &= \sum_{ii'}^r \Big(\underbrace{4(h_{ii'}f_if_{i'}-f_if_{i'}h_ih_{i'})}_{\var(III)} + \underbrace{4(h_{ii'}g_ig_{i'}-h_ih_{i'}g_ig_{i'})}_{\var(IV)} - \underbrace{8(h_{ii'}f_ig_{i'}-f_ih_ih_{i'}g_{i'})}_{\cov(III,IV)} \\
    &= 4\langle h(P_r(f-g))^2\rangle - 4\langle hP_r(f-g)\rangle^2 \\
    &\leq 4A_{hfg}.
\end{align*}
For the lower order terms we obtain
\begin{align*}
    \coef\left(\frac{1}{nm}\right) &= \sum_{ii'}^r \Big(\underbrace{4(f_{ii'}h_{ii'}-f_if_{i'}h_ih_{i'})-4(f_{ii'}h_ih_{i'}-f_if_{i'}h_ih_{i'})-4(f_if_{i'}h_{ii'}-f_if_{i'}h_ih_{i'})}_{\var(III)} \\
    &\qquad + \underbrace{4(h_{ii'}g_{ii'}-h_ih_{i'}g_ig_{i'})-4(h_{ii'}g_ig_{i'}-h_ih_{i'}g_ig_{i'})-4(g_{ii'}h_ih_{i'}-h_ih_{i'}g_ig_{i'})}_{\var(IV)}\Big) \\
    &= 4B_{fh}-4\langle fP_r(h)\rangle^2 - 4\langle fP_r(h)^2\rangle + 4\langle fP_r(h)\rangle^2 \\
    &\qquad - 4\langle hP_r(f)^2\rangle + 4\langle fP_r(h)\rangle^2 + 4B_{gh}-4\langle gP_r(h)\rangle^2 \\
    &\qquad - 4\langle hP_r(g)^2\rangle + 4\langle gP_r(h)\rangle^2 - 4\langle gP_r(h)^2\rangle + 4\langle gP_r(h)\rangle^2 \\
    &\leq 4\langle fP_r(h)\rangle^2 + 4\langle gP_r(h)\rangle^2 + 4B_{fh}+4B_{gh} \\
    &\lesssim |B_{fh}|+|B_{gh}|+\|f+g+h\|_2^4
\end{align*}
where we recall the definition $B_{uv} = \sum_i \langle u\phi_iP_r(v\phi_i)\rangle$ for $u,v\in L^2(\mu)$ and apply the Cauchy-Schwarz inequality. Next, we look at the coefficient of $1/n^2$ and find
\begin{align*}
\coef\left(\frac{1}{n^2}\right) &= \sum_{ii'} \Big(\underbrace{2(f_{ii'}^2-f_i^2f_{i'}^2)-12(f_{ii'}f_if_{i'}-f_i^2f_{i'}^2)}_{\var(\sf{I})} + \underbrace{2(g_{ii'}^2-g_i^2g_{i'}^2)-12(g_{ii'}g_ig_{i'}-g_i^2g_{i'}^2)}_{\var(\sf{II})} \\
&\qquad + \underbrace{8(f_{ii'}f_ih_{i'}-f_i^2f_{i'}h_{i'})}_{\cov(\sf{I},\sf{III})} + \underbrace{8(g_{ii'}g_ih_{i'}-g_i^2g_{i'}h_{i'})}_{\cov(\sf{II}, \sf{IV})}\Big) \\
&= 2B_{ff} - 2\langle fP_r(f)\rangle^2 - 12\langle fP_r(f)^2\rangle + 12 \langle fP_r(f)\rangle^2  \\
&\qquad + 2B_{gg}-2\langle gP_r(g)\rangle^2 - 12 \langle gP_r(g)^2\rangle + 12\langle gP_r(g)\rangle^2 \\
&\qquad + 8\langle f P_r(f)P_r(h)\rangle - 8\langle fP_r(f)\rangle\langle fP_r(h)\rangle + 8\langle gP_r(g)P_r(h)\rangle - 8\langle gP_r(g)\rangle\langle gP_r(h)\rangle \\
&\leq 2B_{ff} + 2B_{gg} + 8\langle fP_r(f)P_r(h-f)\rangle + 8\langle gP_r(g)P_r(h-g)\rangle + 40 \|f+g+h\|_2^4 \\
&\lesssim |B_{ff}|+|B_{gg}|+\|f+g+h\|_2^4 + \sqrt{A_{ff0} A_{ffh} + A_{gg0}A_{ggh}}. 
\end{align*}
Finally, we look at the coefficient of $1/n^3$:
\begin{align*}
    \coef\left(\frac{1}{n^3}\right) &= \sum_{ii'} \Big(\underbrace{-2(f_{ii'}^2-f_i^2f_{i'}^2)+8(f_{ii'}f_if_{i'}-f_i^2f_{i'}^2)}_{\cov(\sf{I},\sf{III})} \underbrace{-2(g_{ii'}^2-g_i^2g_{i'}^2)+8(g_{ii'}g_ig_{i'}-g_i^2g_{i'}^2)}_{\cov(\sf{I},\sf{III})}\Big) \\
    &= -2 B_{ff} + 2\langle fP_r(f)\rangle^2 + 8\langle fP_r(f)^2\rangle - 8\langle fP_r(f)\rangle^2 \\
    &\qquad - 2B_{gg} + 2\langle gP_r(g)\rangle^2 + 8\langle gP_r(g)^2\rangle - 8\langle gP_r(g)\rangle^2 \\
    &\lesssim |B_{ff}|+|B_{gg}| + \|f+g+h\|_2^4 + A_{ff0} + A_{gg0}.
\end{align*}
\end{proof}

\subsection{Proof of Lemma \ref{lem:multi expectation}}

\begin{proof}
Expanding via the binomial formula and using the fact that sums of $N_j$'s are binomial random variables, we get
\begin{align*}
    \bb E_N \prod\limits_{j \in k} (a+b(1+c)^{N_j}) &= \bb E \sum\limits_{\ell=0}^k {k \choose \ell} b^\ell(1+c)^{\bin(n, \ell/k)} a^{k-\ell} \\
    &= \sum\limits_{\ell=0}^k {k\choose \ell}b^\ell(1+\frac{c\ell}{k})^na^{k-\ell} \\
    &\leq (a+be^{cn/k})^k,
\end{align*}
where we used $1+x\leq e^x$ for all $x \in \bb R$.
\end{proof}

\end{document}